\documentclass[10pt]{amsart}
\usepackage{latexsym, amsmath,amssymb}

\usepackage{hyperref}
\usepackage{xcolor}

\renewcommand{\epsilon}{\varepsilon}
\newcommand{\newsection}[1]
{\subsection{#1}\setcounter{theorem}{0} \setcounter{equation}{0}
\par\noindent}

\newtheorem{theorem}{Theorem}

\newtheorem{lemma}[theorem]{Lemma}
\newtheorem{corr}[theorem]{Corollary}

\newtheorem{proposition}[theorem]{Proposition}
\newtheorem{deff}[theorem]{Definition}

\newcommand{\bth}{\begin{theorem}}
\newcommand{\ble}{\begin{lemma}}
\newcommand{\bcor}{\begin{corr}}

\newcommand{\bdeff}{\begin{deff}}

\newcommand{\bprop}{\begin{proposition}}
\newcommand{\ele}{\end{lemma}}
\newcommand{\ecor}{\end{corr}}
\newcommand{\edeff}{\end{deff}}

\newcommand{\eprop}{\end{proposition}}

\newcommand{\Rn}{{\mathbb R}^n}

\newcommand{\la}{\lambda}

\newcommand{\e}{\varepsilon}

\renewcommand{\Pi}{\varPi}
\renewcommand{\Re}{\mathrm{Re} \,}
\renewcommand{\Im}{\mathrm{Im} \,}
\renewcommand{\epsilon}{\varepsilon}

\newcommand{\R}{{\mathbb R}}

\newcommand{\Tn}{{\mathbb T}^n}

\newcommand{\Dom}{\mathrm{Dom}}

\newcommand{\1}{{\rm 1\hspace*{-0.4ex}%
\rule{0.1ex}{1.52ex}\hspace*{0.2ex}}}

\begin{document}

\title[Uniform Sobolev Estimates involving singular potentials]
{Uniform Sobolev Estimates on compact manifolds involving   singular potentials}
%
%
%
%
%
%

\thanks{M.D.B.~was partially supported by NSF Grant DMS-1565436, Y.S.~was partially supported by the Simons Foundation, and X.H. and C.D.S. were partially supported by NSF Grants DMS-1665373 and DMS-1953413}

 \author{Matthew D. Blair}
\address[M.D.B.]{Department of Mathematics and Statistics,
University of New Mexico, Albuquerque, NM 87131,  USA}
\email{blair@math.unm.edu}

\author{Xiaoqi Huang}
\address[X.H.]{Department of Mathematics,  Johns Hopkins University,
Baltimore, MD 21218}
\email{xhuang49@math.jhu.edu}

\author{Yannick Sire}
 \address[Y.S.]{Department of
  Mathematics, Johns Hopkins University, Baltimore, MD 21218, USA}
\email{sire@math.jhu.edu}

\author{Christopher D. Sogge}
\address[C.D.S.]{Department of Mathematics,  Johns Hopkins University,
Baltimore, MD 21218}
\email{sogge@jhu.edu}

\begin{abstract}We obtain generalizations of the uniform Sobolev inequalities of
Kenig, Ruiz and the fourth author~\cite{KRS} for Euclidean spaces and Dos Santos Ferreira, Kenig and Salo~\cite{DKS} for compact Riemannian manifolds
involving critically singular potentials $V\in L^{n/2}$.
We also obtain the analogous improved quasimode estimates of the the first, third and fourth authors \cite{BSS} ,
Hassell and Tacy~\cite{HassellTacy}, the first and fourth author~\cite{SBLog}, and Hickman~\cite{Hickman}
as well as analogues of the improved uniform Sobolev estimates of \cite{BSSY} and \cite{Hickman} involving
such potentials.   Additionally, on $S^n$, we obtain sharp uniform Sobolev inequalities involving such potentials for the optimal range of exponents, which extend the results of S. Huang and the fourth author~\cite{SHSo}. For general Riemannian manifolds
we improve the earlier results in \cite{BSS} by obtaining quasimode estimates for a larger (and optimal) range of exponents
under the weaker assumption that $V\in L^{n/2}$.
\end{abstract}

\maketitle
\setcounter{secnumdepth}{3}

\newsection{Introduction and main results}


The main purpose of this paper is to extend the uniform Sobolev inequalities on compact Riemannian manifolds $(M,g)$ of \cite{DKS}, \cite{BSSY}, and \cite{ShYa} to include Schr\"odinger operators,
\begin{equation}\label{1.1}
H_V=-\Delta_g+V(x),
\end{equation}
with critically singular potentials $V$, which are always assumed to be real-valued.  
For the most part, we  shall merely assume that
\begin{equation}\label{1.2}V
\in 
L^{n/2}(M).
\end{equation} 


In an earlier work of three of the authors \cite{BSS}, 
in addition to \eqref{1.2}, 
it was assumed that $V\in {\mathcal K}$, the Kato class (see \S~\ref{abstractsec}).
The spaces $L^{n/2}$ and ${\mathcal K}$ have the same
scaling properties, and both obey the scaling law of the Laplacian, which accounts for their criticality.  As was shown in \cite{BSS}, this condition that $V$ be
a Kato potential is necessary to obtain quasimode estimates for $q=\infty$.  On the other hand, for the exponents arising in uniform Sobolev assumptions we merely need to assume \eqref{1.2}.  There is also recent related work of the
second and fourth authors \cite{HuS} and Frank and Sabin ~\cite{FS} involving the Weyl counting problem for Kato potentials.  Using the uniform Sobolev estimates that we shall prove, we shall easily be able to obtain $L^q$ quasimode estimates for the
optimal range of exponents \eqref{1.5'}, and if we assume, in addition to \eqref{1.2}, that $V_-=\max\{0,-V\}$ is in
the Kato space ${\mathcal K}(M)$, we shall also be able to prove quasimode estimates for larger exponents.  In an earlier work,
the stronger assumption that $V\in {\mathcal K}(M)$ was used to obtain results for large exponents.



As we shall show in an appendix, if we assume \eqref{1.2}
then $H_V$ is essentially self-adjoint and bounded from
below with discrete spectrum, $\mathrm{Spec  } \, H_V$.  After adding a constant to $V$, we may without loss of
generality assume, as we shall throughout, that
\begin{equation}\label{1.3}
0\in \mathrm{Spec } \, H_V \, \, \, \text{and } \, \, \, \text{Spec }H_V \subset \R_+=[0,\infty).
\end{equation}



In order to prove these uniform Sobolev estimates we shall use the following generalized second resolvent formula, which holds for all $n\ge 3$ if $V$ satisfies \eqref{1.2},
\begin{multline}\label{1.4}
(-\Delta_g+V-\zeta)^{-1}-(-\Delta_g-\zeta)^{-1} \\= 
-\bigl[|V|^{1/2} \, (-\Delta_g-\overline{\zeta}\, )^{-1}\bigr]^* \circ
\bigl[ V^{1/2}(\, -\Delta_g+V-\zeta)^{-1}\bigr], \quad \mathrm{Im} \, \zeta \ne 0,
\end{multline}
along with quasimode estimates and 
 uniform Sobolev estimates for the 
unperturbed operator $H_0=-\Delta_g$ from \cite{sogge88} \cite{DKS}, \cite{BSSY}
and \cite{ShYa}.  Here $V^{1/2}=(\text{sgn} V)|V|^{1/2}$, and $[\,\,\cdot\,\,]$ denotes the (unique) bounded extension to the whole space. The resolvent formula \eqref{1.4} also holds for a more general class of potentials;
see e.g.,  \cite{kato1966wave} and \cite{kuroda1973scattering} for more details.

We shall also mention that, for $n\ge 5$, we have the following simpler form of the second resolvent formula,
\begin{equation}\label{1.4'}
(-\Delta_g+V-\zeta)^{-1}-(-\Delta_g-\zeta)^{-1} = -(-\Delta_g-\zeta)^{-1}\, V(\, -\Delta_g+V-\zeta)^{-1},
\end{equation}
since, as we shall show in the appendix, for these dimensions, the operator domains of $H_V-\zeta$ and $-\Delta_g-\zeta$ coincide if  $\mathrm{Im} \, \zeta \ne 0$.

The universal uniform Sobolev estimates and quasimode estimates that we can obtain are
the following.

\begin{theorem}\label{unifSob}
Let $n\ge 3$ and suppose that
\begin{equation}\label{1.5}
\min\bigl(q, \, p(q)'\bigr)\ge \tfrac{2(n+1)}{n-1},
\quad \text{and } \, \tfrac1{p(q)}-\tfrac1q=\tfrac2n.
\end{equation}
 Then
if $V\in L^{n/2}(M)$ satisfies \eqref{1.3}, 
and $\delta>0$ are fixed we have the uniform bounds
\begin{equation}\label{1.6}
\|u\|_q\le C_V\bigl\|\bigl(H_V-\zeta\bigr)u\bigr\|_{p(q)}, \quad \text{if } \, 
\, \, \zeta\in \Omega_\delta,
\end{equation}
where
\begin{equation}\label{1.7}
 \Omega_\delta = \{\zeta \in {\mathbb C}: \,
(\Im \zeta)^2 \ge \delta |\Re \zeta| \, \, 
\text{if } \, \, \Re \zeta \ge 1, \, \, \text{and }  \, \,
\mathrm{dist }(\zeta, \R_+)
 \ge \delta
\, \, 
\text{if } \, \, \Re \zeta <1\}.
\end{equation} 
Also, suppose that 
\begin{equation}\label{1.5'}
2<q\leq \tfrac{2n}{n-4}, \,\,\,\text{if}\,\,\, n\ge 5,\,\,\,\text{or}\,\,\,2<q< \infty,\,\,\ \text{if}\,\,\, n=3, 4.
\end{equation}
Then if $u\in\text{Dom}(H_V)$, we have
\begin{equation}\label{qm}
\|u\|_q\lesssim \la^{\sigma(q)-1}\|(H_V-\la^2+i\la) u\|_2,
\, \, \, \text{if } \, \la\ge1,
\end{equation}
if
\begin{equation}\label{1.12}
\sigma(q)=
\begin{cases}n\bigl(\tfrac12-\tfrac1q)-\tfrac12, \quad q\ge \tfrac{2(n+1)}{n-1},
\\ \\
\tfrac{n-1}2(\tfrac12-\tfrac1q), \quad 2 \le q<\tfrac{2(n+1)}{n-1}.
\end{cases}
\end{equation}
\end{theorem}

%

Here, $\Dom(H_V)$  denotes the domain of $H_V$.     Also, $r'$ denotes the conjugate exponent for $r$, i.e.,
the one satisfying $1/r+1/r'=1$.  
Additionally, we are using the notation that $A\lesssim B$ means that $A$ is bounded from above by a constant times $B$.  The implicit constant might depend on the parameters involved, such as $(M,g)$, $q$ and $V$ in
\eqref{qm}.  

 
The range of exponents in \eqref{1.5} for the uniform Sobolev estimates \eqref{1.6} is more restrictive
than the corresponding 
estimates for $\Rn$ in \cite{KRS} since we require certain $L^2\to L^r$ quasimode estimates from
\cite{sogge88} for both $r=q$ and $r=p(q)'$, which are only valid when the first part of
\eqref{1.5} holds.  Succinctly put, our proof of \eqref{1.6} requires that we use the
manifold version of the Stein-Tomas extension theorem \cite{TomasRestriction}, which is only valid when this condition
holds (see \cite{SFIO2} for more details).

  The condition in the uniform Sobolev inequalities for $\Rn$ in \cite{KRS} is that we replace 
\eqref{1.5} with the weaker requirement that
\begin{equation}\label{1.17}\min\bigl(q, \, p(q)'\bigr)> \tfrac{2n}{n-1}
\quad \text{and } \, \, \, \tfrac1{p(q)}-\tfrac1q=\tfrac2n,
\end{equation} 
which was shown to be
be sharp in \cite{KRS}.
The gap condition in \eqref{1.5} and \eqref{1.17} that $ \tfrac1{p(q)}-\tfrac1q=\tfrac2n$, follows from scaling considerations,
while the necessity of the first part of \eqref{1.17}  is related to the fact that the Fourier transform of surface measure on the
sphere in $\Rn$ is not in $L^q(\Rn)$ if $q\le \tfrac{2n}{n-1}$.

Even though the range of exponents for the uniform Sobolev estimates above might be non-optimal, the ones in \eqref{1.5'}
for the quasimode estimates \eqref{qm} are best possible.
For $n\ge4$ this is
 due to a counterexample for the $V\equiv 0$ case in \cite{SoggeTothZelditch}
(see also \cite{SoggeZelditchQMNote}), and for $n=3$ it follows from a counterexample in \cite[\S 1] {BSS} involving
a nontrivial $L^{n/2}$ potential.
It was a bit surprising to us that, even though the range of exponents for the
uniform Sobolev estimates \eqref{1.6} might be a bit restrictive, we can use them along with their proof to obtain
quasimode bounds as in \eqref{qm} for the optimal range of exponents.  

In an earlier work \cite{BSS} bounds of the
form \eqref{qm} were only obtained for the smaller range where $q<\tfrac{2n}{n-3}$.    Moreover, the bounds \eqref{qm} also improve
the earlier ones since
we are only assuming that $V\in L^{n/2}(M)$ and not
that $V$ is a Kato potential, i.e.,
$V\in {\mathcal K}(M)$.  

As we mentioned before,  if in addition to \eqref{1.2}, we also assume that the negative part of $V$ satisfies $V_-\in {\mathcal K}(M)$ then we can
also obtain the (modified) quasimode estimates in \eqref{qm} and the
related spectral projection estimates for larger
exponents.   See the end of \S~\ref{abstractsec}.

We would also like to record that by using the quasimode estimates \eqref{qm} in Theorem~\ref{unifSob} we can obtain, as
a corollary, Sobolev estimates for $H_V$ in higher dimensions which appear to be new since they only involve the assumption $V\in L^{n/2}(M)$
under which favorable heat kernel estimates need not be valid (see Aizenman and Simon~\cite{AZ} and Simon~\cite{SimonSurvey}).

\begin{corr}\label{Sob}
Let $(M,g)$ be a compact Riemannian manifold of dimension $n\ge 5$ and assume that $H_V$ is as above with $V\in L^{n/2}(M)$.
Then 
\begin{equation}\label{sobes}
\|(H_V+1)^{-\alpha/2} f\|_{L^q(M)}\lesssim \| f \|_{L^p(M)},
\end{equation}
provided that 
\begin{equation}\label{condition}
n(1/p-1/q)=\alpha \quad \text{and } \, \, \, \tfrac{2n}{n+4}\le p\le 2\le q\le \tfrac{2n}{n-4}.
\end{equation}
\end{corr}

The proof is simple.  
Since we are assuming \eqref{1.3}, we obtain
from the spectral theorem and the special case of \eqref{qm} with $\la=1$ that $(H_V+1)^{-1}: \, L^2(M)\to L^{\frac{2n}{n-4}}(M)$, and
by duality it also maps $L^{\frac{2n}{n+4}}(M)\to L^2(M)$.  By applying Stein's interpolation theorem, the spectral theorem and the trivial
$L^2$ bounds, we deduce that $(H_V+1)^{-\alpha/2}: \, L^2(M)\to L^q(M)$ for $2\le q \le \tfrac{2n}{n-4}$ with
$\alpha =n(1/2-1/q)$, and also $(H_V+1)^{-\alpha/2}: \, L^p(M)\to L^2(M)$ for $\tfrac{2n}{n+4}\le p\le 2$ with
$\alpha =n(1/p-1/2)$.  Since these two facts yield the desired $L^p(M)\to L^q(M)$ bounds for 
$(H_V+1)^{-\alpha/2}$ the proof is complete.



As in \cite{BSSY}, in certain geometries we can obtain
improved uniform Sobolev estimates and quasimode estimates using improved
bounds for the unperturbed operator $H_0$.  

First, if we use the improved spectral projection estimates of Hassell and Tacy~\cite{HassellTacy}
and two of us \cite{SBLog}, we can obtain the following.

\begin{theorem}\label{nonpossob}
Let $n\ge 3$ and suppose that
\begin{equation}\label{1.8}
\min\bigl(q, \, p(q)'\bigr)> \tfrac{2(n+1)}{n-1},
\quad \text{and } \, \tfrac1{p(q)}-\tfrac1q=\tfrac2n.
\end{equation}
Assume also that $(M,g)$ has nonpositive sectional curvatures,
$V\in L^{n/2}(M)$ satisfies \eqref{1.3}, and that $\delta>0$ are fixed.  Then we have
\begin{equation}\label{1.9}
\|u\|_q\le C\bigl\|\bigl(H_V-\zeta\bigr)u\bigr\|_{p(q)}, \quad \text{if } \, 
\, \zeta \in \Omega_{\e,\delta},
\end{equation}
where
\begin{multline}\label{1.10}
 \Omega_{\e,\delta}=
\{\zeta: \,
(\Im \zeta)^2 \ge \delta \, 
\bigl(\e(\la)\bigr)^2 \, 
|\Re \zeta| \, \, \, 
\text{if } \,  \Re \zeta \ge 1, 
\\
 \text{and }  \, \,
\mathrm{dist }(\zeta, \R_+) \ge \delta
\, \, \text{if } \,  \Re \zeta <1
\},
\end{multline} 
with
\begin{equation}\label{1.11}
\e(\la)=\bigl(\log(2+\la)\bigr)^{-1}.
\end{equation}
Also, suppose that 
\begin{equation}\label{1.5''}
\tfrac{2(n+1)}{n-1}<q\leq \tfrac{2n}{n-4}, \,\,\,\text{if}\,\,\, n\ge 5,\,\,\,\text{or}\,\,\,\tfrac{2(n+1)}{n-1}<q< \infty,\,\,\ \text{if}\,\,\, n=3, 4.
\end{equation}
Then if $\sigma(q)$ is as in \eqref{1.12} and
$u\in \Dom(H_V)$,
\begin{equation} \label{1.13}
\|u\|_q\lesssim \bigl(\sqrt{\e(\la)}\bigr)^{-1} \, \la^{\sigma(q)-1}\, \| (H_V-\la^2+i\e(\la)\la)u\|_{2},
\,\,\, \text{if } \, \la \ge 1.
\end{equation}
Finally, if $q=q_c=\tfrac{2(n+1)}{n-1}$ we have for some $\delta_n>0$ depending on the dimension
\begin{equation}\label{3.20}
\|u\|_{q_c}\lesssim \la^{\sigma(q_c)-1}\, \bigl(\e(\la)\bigr)^{-1+\delta_n} \, 
\bigl\| (H_V-(\la+i\e(\la))^2)u\bigr\|_2
\end{equation}
\end{theorem}




The quasimode estimates \eqref{1.13} improve those
in \cite{BSS} in several ways.  First, as noted before, we
are not assuming that $V$ is a Kato potential, only 
\eqref{1.2}.  Moreover, unlike \cite{BSS} we also do not
have to assume that $V$ has small $L^{n/2}$-norm.  
We also obtain the bounds in \eqref{1.13} for the optimal range of exponents given by \eqref{1.5'}, and the bounds \eqref{3.20} for the critical
exponent $q=q_c$ are new.  We have only stated the bounds of the form \eqref{3.20} for $q=q_c$; however, if one interpolates
with the trivial $L^2$ estimate one sees that bounds of the form \eqref{3.20} also hold for all $q\in (2,q_c)$ if one replaces
$\delta_n$ with the appropriate $\delta_{n,q}>0$.

As
we noted after Theorem~\ref{unifSob} we also can obtain
quasimode bounds for exponents larger than 
the ones in \eqref{1.13}
if we assume that $V_-\in {\mathcal K}(M)$, and here too, in this case,  we can drop the smallness assumption that was
used in \cite{BSS}.

By results in \cite{SoggeZelditchQMNote} the bounds in 
\eqref{1.13} are equivalent to the following spectral projection bounds
\begin{equation}\label{1.14}
\|\chi^V_{[\la, \la+(\log \la)^{-1}]}\|_{L^2(M)\to L^q(M)}
\lesssim (\log \la)^{-1/2} \la^{\sigma(q)}, \quad
\la \ge 1,
\end{equation}
for $q$ as in \eqref{1.5''}, if 
$\chi^V_{[\la, \la+(\log 2+\la)^{-1}]}$ denotes the spectral projection operator which projects onto the part of the spectrum of $\sqrt{H_V}$
 in the corresponding shrinking intervals 
$[\la, \la+(2+\log \la)^{-1}]$.
If in addition to \eqref{1.2} we also assume that $V$
is in the Kato class then we 
also have \eqref{1.14}, as in the
$V\equiv 0$ case in Hassell and Tacy~\cite{HassellTacy}
for all $p>\tfrac{2(n+1)}{n-1}$.  The bounds in \eqref{3.20} extend the log-improvements of two of
us \cite{SBLog} to include singular potentials as above.  Just as was the case for \eqref{1.14}, the quasimode
estimates in \eqref{3.20} yield the equivalent log-improved spectral projection estimates
\begin{equation}\label{3.21}
\bigl\|\chi^V_{ [\la, \la+(\log(2+\la))^{-1}] } f \bigr\|_{q_c} \lesssim \la^{\sigma(q_c)} \, (\log(2+\la))^{-\delta_n} \, \|f\|_2.
\end{equation}

Additionally, in \S~\ref{2d}, we shall obtain quasimode estimates of the form
\eqref{1.13} and \eqref{3.20} when $n=2$; however, as in \cite{BSS} (which handled small potentials), in this case we shall have to assume that
$V \in L^1(M)\cap {\mathcal K}(M)$.  We improve the corresponding results in \cite{BSS}, though, by dropping the smallness assumption
on $V$.

As was shown in Hickman~\cite{Hickman} in higher dimensions and Bourgain, Shao, Yao and one of
us~\cite{BSSY} for $n=3$,  one can use 
the decoupling theorem of Bourgain and Demeter~\cite{BourgainDemeterDecouple} to obtain
substantial improvements of \eqref{1.14}
when $M={\mathbb T}^n$ is the torus, which correspond to taking $\e(\la)=\la^{-1/3+c}$ for
all $c>0$.  Using these improved quasimode
estimates we can prove the corresponding 
stronger version of Theorem~\ref{nonpossob} for
tori.

\begin{theorem}\label{torussob}
Let $n\ge3$ and assume that $p(q)$ and $q$ are as in 
\eqref{1.17}.  Then for $V\in L^{n/2}({\mathbb T}^n)$ satisfying \eqref{1.3},
$\delta>0$ and $c_0>0$ fixed, 
we have
\begin{equation}\label{1.15}
\|u\|_{L^q(\Tn)}\le C\bigl\|\bigl(H_V-\zeta\bigr)u\bigr\|_{L^{p(q)}(\Tn)}, \quad \text{if }
\, \, \zeta\in \Omega_{\e,\delta},
\end{equation}
where $\Omega_{\e,\delta}$ is as in \eqref{1.10} with
\begin{equation}\label{1.16}
\e(\la)=\begin{cases}\la^{-\beta_1(n, p(q)^\prime)+c_0}\,\,\,\text{if}\,\,\, \frac{2n}{n-1}< q<\frac{2n}{n-2} \\
\la^{-\beta_1(n,q)+c_0}\,\,\,\text{if}\,\,\, \frac{2n}{n-2}\le q<\frac{2n}{n-3},
\end{cases}
\end{equation}
for certain $\beta_1(n,r)>0$ and $p(q)^\prime$ such that $\frac{1}{p(q)^\prime}+\frac{1}{p(q)}=1$.
Also, suppose that 
\begin{equation}
\nonumber
\e(\la)=\begin{cases}\la^{-\beta(n, q)+c_0}, \,\,\,\text{if}\,\,\, \frac{2(n+1)}{n-1}< q<\frac{2n}{n-2},\\
\la^{-\frac13+c_0}, \,\,\,\frac{2n}{n-2}\le q\leq \tfrac{2n}{n-4}, \,\,\,\text{if}\,\,\, n\ge 5,\,\,\,\text{or}\,\,\,\frac{2n}{n-2}\le q< \infty,\,\,\ \text{if}\,\,\, n=3, 4,
\end{cases}
\end{equation}
where 
$$\beta(n, q)=\min\{ \beta_1(n, p(q)^\prime),  \tfrac{(n-1)^2q-2(n-1)(n+1)}{(n+1)(n-1)q-2(n+1)^2+8}\}.
$$
Then we have the analog of \eqref{1.13} on ${\mathbb T}^n$
for $q$ satisfying \eqref{1.5''}.
\begin{equation} \label{1.15'}
\|u\|_{L^q(\Tn)}\lesssim \bigl(\sqrt{\e(\la)}\bigr)^{-1} \, \la^{\sigma(q)-1}\, \| (H_V-\la^2+i\e(\la)\la)u\|_{L^2(\Tn)},
\,\,\, \text{if } \, \la \ge 1.
\end{equation}
Additionally, for the critical point $q_c=\frac{2(n+1)}{n-1}$, suppose that $\e(\la)=\la^{-\beta_1(n, p(q_c)^\prime)+c_0}$ which satisfies \eqref{1.16}, or more explicitly 
\begin{equation}\label{1.18}
\,\,\,\e(\la)=\la^{-\frac15+c_0},
\,\,\text{if}\,\,n\ge 4,\,\,\text{and}\,\,\,  \e(\la)=\la^{-\frac{3}{16}+c_0},\,\,\text{if}\,\,n=3,
\end{equation}
we have for $u\in\text{Dom}(H_V)$
\begin{equation}\label{1.19}
\|u\|_{L^{q_c}(\Tn)}
\lesssim \la^{\e_0}
(\e(\la))^{-\frac{n+3}{2(n+1)}}
\la^{-\frac{n+3}{2(n+1)}} 
\|(H_V-(\la + i\e(\la))^2)u\|_{L^2(\Tn)},
\, \, \la \ge 1,
\end{equation}
\end{theorem}

We shall give the explicit definition of $\beta_1(n, q)$  later in \eqref{4.48}. As we shall see, $\beta_1(n, q)$ is a number that decreases from $1/3$ to $0$ when $q$ increases from $\tfrac{2n}{n-2}$ to $\tfrac{2n}{n-3}$. Similarly, by an explicit calculation, $\beta(n, q)$ is a number that increases from $0$ to $1/3$ when $q$ increases from $\tfrac{2(n+1)}{n-1}$ to $\tfrac{2n}{n-2}$, in particular, when $q=\frac{2n}{n-2}$, $\beta_1(n, q)=\beta(n, q)=1/3$. As a result, \eqref{1.15} generalizes the uniform resolvent estimates of Hickman \cite{Hickman} to the setting of 
Schr\"odinger operators with $V\in L^{n/2}(\Tn)$, which also gives us certain uniform resolvent estimates on the torus for general pairs of exponents $(p, q)$ satisfying \eqref{1.17}. On the other hand, when $q=\frac{2n}{n-2}$, if we take $u$ in \eqref{1.15'} to be $\chi^V_{[\la,\la+\e(\la)]}f$, we have 
$$\|\chi^V_{[\la,\la+\rho)]}f\|_{L^{\frac{2n}{n-2}}(\Tn)}\le (\rho\la)^{1/2}\|f\|_{L^{2}(\Tn)}, \,\,\forall \delta_0>0, \,\,\,\rho\ge \la^{-\frac13+\delta_0},
$$
which generalizes the spectral projection estimates in \cite{Hickman} (and \cite{BSSY} for the $n=3$ case) to the setting of Schr\"odinger operators.


Theorems~\ref{nonpossob} and \ref{torussob} represent an improvement in terms of the $\e(\la)$ defining $\Omega_{\e,\delta}$ as well as the parameter occurring in the quasimode estimates \eqref{1.14}
over Theorem~\ref{unifSob} which corresponds to 
$\e(\la)\equiv 1$.  


 For the sphere no such improvement over the case where $\e(\la)\approx 1$ is possible since one cannot have
$\e(\la)\to 0$ as $\la\to +\infty$ in this case (see \cite{SHSo} and \cite{sogge86}).  Notwithstanding,  for $S^n$, we can get an improvement over
Theorems~\ref{nonpossob} and \ref{torussob}  for the uniform Sobolev estimates by obtaining bounds for the optimal
range of exponents satisfying \eqref{1.17}.
This improvement is possible due to the fact that when $M=S^n$
uniform Sobolev estimates for $H_0$
are known for this  range of exponents (see \cite{SHSo}).

\begin{theorem}\label{sphere}  Consider the standard sphere $S^n$ for $n\ge 3$ and assume
that $V\in L^{n/2} (S^n)$.  If \eqref{1.17} is valid we have
\begin{equation}\label{1.21}
\|u\|_q\le C\bigl\|\bigl(H_V-\zeta\bigr)u\bigr\|_{p(q)}, \quad \text{if } \, 
 \, \, \zeta\in \Omega_\delta,
\end{equation}
where $\Omega_\delta$ is as in \eqref{1.7}. Also, for q satisfying \eqref{1.5'}, if $\sigma(q) $ is as in \eqref{1.12} and $u\in \Dom(H_V)$
\begin{equation} \label{1.22}
\|u\|_q\lesssim \, \la^{\sigma(q)-1}\, \| (H_V-\la^2+i \la)u\|_{2},
\quad \text{if } \, \la \ge 1.
\end{equation}
\end{theorem}

It would be interesting to see if the  uniform Sobolev bounds \eqref{1.21} are universally true or hold for generic Riemannian manifolds. 

 
The study of Schr\"odinger operators can be found in a vast amount of literatures, especially in the Euclidean case, see e.g, \cite{jensen1979spectral}, \cite{journe1991decay}, \cite{rodnianski2004time}. In a companion paper \cite{HuS2} the second and fourth authors will obtain related uniform Sobolev estimates for $\Rn$ which
improve those in \cite{BSS} and provide natural generalizations of those in \cite{KRS}.

The authors are grateful to R.\ Frank and J.\ Sabin for sharing their recent work which influenced this paper.  We are grateful
to R.\ Frank for helpful suggestions which helped us to weaken the hypothesis on our potentials, and also to the referees for several helpful suggestions which improved our exposition.

\newsection{Universal Sobolev inequalities on compact manifolds: Abstract universal bounds}\label{abstractsec}

The purpose  of this section is to prove  simple abstract
theorems that will allow us to 
prove Theorems \ref{unifSob}--\ref{sphere}, and to also
improve  the quasimode
estimates of \cite{BSS} for the operators $H_V$, 
provided that 
we have the analogous improved estimates (quasimode and uniform Sobolev) for the
unperturbed operators $H_0=-\Delta_g$.  


%
Throughout this section we shall assume that $n\ge3$
since we shall be using uniform Sobolev estimates for $-\Delta_g$
which break down in two-dimensions.  We shall obtain
improved quasimode estimates compared to those in \cite{BSS} later by adapting the arguments here.

In this section we shall consider a pair of exponents $(p,q)$ which are among those in the sharp range of exponents in the uniform Sobolev estimates
in \cite{KRS} for the Euclidean case, i.e.,
$1<p<2<q<\infty$, and, moreover,
\begin{equation}\label{2.1}
\tfrac1p-\tfrac1q=\tfrac2n, \quad
\min(q,p')>\tfrac{2n}{n-1}.
\end{equation}
For later use, observe that if the pair $(p,q)$ is as in \eqref{2.1} then so is $(q',p')$.  We also note that if 
$(p,q)$ is as in \eqref{2.1} then $\tfrac{2n}{n-1}<q<\tfrac{2n}{n-3}$.

For both of the exponents in \eqref{2.1}, we shall assume that
we have improvements of the classical quasimode
estimates of the fourth author \cite{sogge88} of the form
\begin{multline}\label{2.2}
\|u\|_r\le C \delta(\la,r) \,
\la^{\sigma(r)-1} \, \bigl(\e(\la)\bigr)^{-1}
\, \bigl\|(-\Delta_g-\la^2+i\e(\la) \la) u\bigr\|_2,
\\
\text{for } \, \, r=q, \, p' \quad \text{and } \, \la \ge 1.
\end{multline}
where 
$\sigma(r)$ is as in \eqref{1.12}.
The  $\delta(\la,r)$ and $\e(\la)$ are assumed to be continuous functions of 
$\la\in [1,\infty)$. In practice
they are nonpositive powers of $\la$ or $\log(2+\la)$.

In order to have improvements over the results in
\cite{sogge88} for $\e(\la)\equiv 1$ we shall assume
that
\begin{equation}\label{2.3}
\e(\la)\searrow  \, \, 
\, \,  \text{and }\, \,
\e(\la)\in [1/\la,1], \, \, \la\ge 1.
\end{equation}
We make the assumption that $\e(\la)\ge 1/\la$
since on compact manifolds it is unreasonable to expect meaningful bounds of
the form \eqref{2.2} when $\e(\la)$ is smaller than
the associated wavelength $1/\la$ with $\la$ large.
The estimates in \cite{sogge88} and the spectral theorem
imply that \eqref{2.2} is valid when $\delta(\la,r)\equiv 1$, and so we shall also assume that
\begin{equation}\label{2.4}
(\e(\la))^{1/2}\le \delta(\la,r)\le 1
\quad \text{and } \, \, 
\delta(\la,r)\searrow \, , \, \, \la \ge 1.
\end{equation}
We assume that $\delta(\la,r)\ge (\e(\la))^{1/2}
$ since, by (5.1.12) and (5.1.13) in \cite{SFIO2}, \eqref{2.2}
cannot hold if $(\e(\la))^{1/2}/\delta(\la,r)\to \infty$
as $\la\to +\infty$.  

Note that  \eqref{1.13} corresponds to the ``critical case'' where
$\delta(\la,r)=(\e(\la))^{1/2}$ for $\e(\la)$ as in \eqref{1.11} in the case of manifolds of nonpositive
curvature, as do the results of \cite{BSSY} for $n=3$ and \cite{Hickman} for $n\ge 4$ with a more favorable numerology
on tori.

Although a bit more cryptic at first glance, it is also natural to assume that
\begin{equation}\label{2.5}
\limsup_{\la\to \infty} \, 
\la^{\sigma(q)+\sigma(p')-2}
\, \bigl(\e(\la)\bigr)^{-2}\, 
\delta(\la,q)\, \delta(\la,p')=0.
\end{equation}
This condition arises naturally in the proofs, and
one can check that, for the the exponents in \eqref{2.1},
it holds for the special case where $\e(\la)=
\delta(\la,q)=\delta(\la,p')\equiv 1$, which will
be a useful observation when we prove certain estimates
on $S^n$.  Also, by the first part of \eqref{2.4},
we have \eqref{2.5} if 
\begin{equation}\label{2.5'}\tag{2.5$'$}
\limsup_{\la\to \infty} \, \la^{\sigma(q)+\sigma(p')-2}
\, \bigl(\e(\la)\bigr)^{-2}\, =0,
\end{equation}
which is a bit more palatable.  

In addition to these quasimode estimates we shall assume
that we have the related uniform Sobolev estimates
for the unperturbed operators:
\begin{equation}\label{2.6}
\|u\|_q\le C_{\delta_0}\bigl\|(-\Delta_g-\la^2 +i\mu \e(\la)\la))u\bigr\|_p, 
\, \, \text{when } \, \, \, 
\la \ge 1 \, \, \,  \text{and } \, \, |\mu|\ge \delta_0,
\end{equation}
if $\delta_0>0$.  
Here and in what follows $\mu\in \R$.
Similar to the remark after \eqref{2.1}, observe that if $(p,q)$ are exponents for which 
\eqref{2.6} is valid, then, by duality this is also true for the pair $(q',p')$.

The abstract theorem that will allow us to prove Theorems~\ref{unifSob}--\ref{sphere} then is the following.

\begin{theorem}\label{thm2.1}  Assume  $(M,g)$ is a compact Riemannian
manifold of dimension $n\ge3$.  Assume further that $(p,q)$  is a pair of exponents 
 satisfying \eqref{2.1}.   Suppose further that \eqref{2.2}, \eqref{2.5}
and \eqref{2.6}
are valid with $\e(\la)$, $\delta(\la,r)$, 
 satisfying \eqref{2.3} and \eqref{2.4}, respectively with $r=p',q$ in the latter.
Then if $V\in L^{n/2}(M)$ we have
\begin{equation}\label{2.7}
\|u\|_q\le C\bigl\|(-\Delta_g+V-\la^2+i\mu \e(\la)\la)u\bigr\|_p, 
 \, \, \, \text{if } \, \, 
|\mu|\ge 1 \, \, \text{and } \, \, 
\la \ge \Lambda,
\end{equation}
assuming that
$\Lambda=\Lambda(M,q,V)\ge 1$ sufficiently large.
\end{theorem}

The assumption that $\lambda$  in \eqref{2.7} is large
arises for  technical reasons from the fact that since we only are assuming that $V\in L^{n/2}$, we only know via \eqref{s51} in the appendix that $u\in L^q(M)$ for $q\le \tfrac{2n}{n-2}$ if $u\in \Dom(H_V)$.  On the other hand, after proving Theorem~\ref{thm2.1} we can use its proof to establish
the following much more favorable results.

\begin{corr}\label{corr2.2}
Assume the hypotheses in Theorem~\ref{thm2.1}.  Then for $u\in \Dom H_V$ 
\begin{multline}\label{2.8}
\|u\|_r\le C_{V,r} \,  \delta(\la,r) \,
\la^{\sigma(r)-1} \, \bigl(\e(\la)\bigr)^{-1}
\, \bigl\|(-\Delta_g+V-\la^2+i\e(\la) \la) u\bigr\|_2,
\\
\text{if } \, \, 
\la \ge 1 \, \, \text{and } \, \, 
r=q \, \, \text{or } \, \, r=p'.
\end{multline}
Additionally,
\begin{equation}\label{2.9}
\|u\|_r\le C_{\delta,V,r}\bigl\|(-\Delta_g+V-\la^2 +i\mu \e(\la)\la))u\bigr\|_s, 
\, \, \text{when } \, \, \, 
\la \ge 1 \, \, \,  \text{and } \, \, |\mu|\ge \delta_0,
\end{equation}
if $\delta_0>0$ and $(r,s)=(q,p)$ or $(p',q')$.
\end{corr}

To prove these results we shall appeal to the following simple lemma.

\begin{lemma}\label{simple}  Assume that $n\ge3$.  Let $(p,q)$ be as in \eqref{2.1} and $W\in L^n(M)$.  Then
if \eqref{2.6} is valid
\begin{multline}\label{l.1}
\bigl\| \, \bigl[ W\, (-\Delta_g-\la^2 -i\mu \e(\la)\la)^{-1}\bigr]^*\, \bigr\|_{L^{\overline{p}}(M) \to L^q(M)} 
\\
\le C_{\delta_0}\|W\|_{L^n(M)}, \quad \text{if } \, \frac1{\overline{p}}=\frac1p-\frac1n,
\end{multline}
and, if \eqref{2.2} is valid for $r=s'$
\begin{multline}\label{l.1'}
\bigl\| \, \bigl[ W\, (-\Delta_g-\la^2 -i\mu \e(\la)\la)^{-1}\bigr]^*\, \bigr\|_{L^{\overline{s}}(M) \to L^2(M)} 
\\
\le C\|W\|_{L^n(M)} \delta(\la,s') \la^{\sigma(s')-1} \, \bigl(\e(\la)\bigr)^{-1}, \quad
\text{if } \, \, \frac1{\overline{s}}=\frac1s-\frac1n.
\end{multline}
Finally, if \eqref{2.2} is valid for $r=q$ and if $W\in L^\infty(M)$
\begin{multline}\label{l.2}
\bigl\| \, \bigl[ W\, (-\Delta_g-\la^2 -i\mu \e(\la)\la)^{-1}\bigr]^*\, \bigr\|_{L^{2}(M) \to L^q(M)} 
\\
\le C\|W\|_{L^\infty(M)} \, \delta(\la,q) \, \la^{\sigma(q)-1} \, \bigl(\e(\la)\bigr)^{-1}.
\end{multline}
\end{lemma}

\begin{proof}
Notice that, since we are assuming $n\ge3$, the operators in \eqref{l.1}--\eqref{l.2} are bounded on $L^2(M)$ by duality,
H\"older's inequality and Sobolev estimates.

Also, by duality, \eqref{l.1} is a consequence of the following
\begin{equation}\label{l.3}
\bigl\| W(-\Delta_g-\la^2-i\mu\e(\la)\la)^{-1}h\bigr\|_{L^{(\overline{p})'}(M)}
\le C_{\delta_0}\|W\|_{L^n(M)}\|h\|_{L^{q'}(M)}.
\end{equation}
To prove this we first observe that
$$\frac1{(\overline{p})'}=1-\frac1{\overline{p}}=1-\frac1p+\frac1n=\frac1{p'}+\frac1n.$$
Thus, by H\"older's inequality and the dual version of \eqref{2.6} we have
\begin{align*}
\bigl\| W(-\Delta_g-&\la^2-i\mu\e(\la)\la)^{-1}h\bigr\|_{L^{(\overline{p})'}(M)} 
\\
&\le \|W\|_{L^n(M)}
\| (-\Delta_g-\la^2-i\mu\e(\la)\la)^{-1}h\|_{L^{p'}(M)}
\\
&\le C_{\delta_0}\|W\|_{L^n(M)}\|h\|_{L^{q'}(M)},
\end{align*}
as desired.

This argument also yields \eqref{l.1'}.  One obtains the dual version of \eqref{l.1'} by applying \eqref{2.2} and H\"older's inequality.

Similarly \eqref{l.2} is equivalent to
\begin{align*}
\|W(-\Delta_g-&\la^2-i\mu\e(\la)\la)^{-1}h\|_{L^2(M)}
\\
&\le C\|W\|_{L^\infty(M)} \delta(\la,q) \, \la^{\sigma(q)-1} \, \bigl(\e(\la)\bigr)^{-1} \, \|h\|_{L^{q^\prime}(M)}.
\end{align*}
This follows immediately from the dual version of \eqref{2.2}.
\end{proof}

\begin{proof}[Proof of Theorem~\ref{thm2.1}]

Let us first note that proving \eqref{2.7} is equivalent
to showing that
$$\bigl\| (H_V-\la^2+i\mu\e(\la)\la)^{-1}\bigr\|_{L^p\to
L^q}\lesssim 1, \quad \text{if } \, \,
\la \ge \Lambda \, \, \text{and } \, \, |\mu|\ge1,$$
with $\Lambda$ sufficiently large and $(p,q)$ as
in \eqref{2.1}.  By duality, it suffices prove this inequality when 
\begin{equation}\label{2.10}
\tfrac{2n}{n-1}<q\le \tfrac{2n}{n-2}.
\end{equation}
Thus, our task is to show that
\begin{equation}\label{2.11}
\bigl\| (H_V-\la^2+i\mu\e(\la)\la)^{-1}f\bigr\|_{L^q(M)}
\le C\|f\|_{L^p(M)} \quad \text{if } \, \,
\la \ge \Lambda \, \, \text{and } \, \, |\mu|\ge1,
\end{equation}
with $(p,q)$ satisfying \eqref{2.1} and \eqref{2.10}.
As in Theorem~\ref{thm2.1} we are also assuming that \eqref{2.2} and \eqref{2.6} are valid for this pair of exponents.

We are assuming \eqref{2.10} since by \eqref{s51}
in the appendix we have
$$u\in L^q(M), \quad 2\le q\le \tfrac{2n}{n-2}
\, \, \text{if } \, \, \,
(H_V-\la^2+i\mu\e(\la)\la)u\in L^2.$$
Thus for $q$ as in \eqref{2.10}
\begin{equation}\label{2.12}
\bigl\| (H_V-\la^2+i\mu\e(\la)\la)^{-1}f\bigr\|_{L^q(M)}
<\infty \quad \text{if } \, \, f\in L^2(M).
\end{equation}
In proving \eqref{2.11} since $L^2$ is dense in $L^p$
we may and shall assume that $f\in L^2(M)$ to be
able to use \eqref{2.12} to justify a bootstrapping
argument that follows.

The bootstrapping argument shall also exploit the
simple fact that if we let
\begin{equation}\label{2.13}
V_{\le N}(x)=
\begin{cases}V(x), \, \, \, \text{if } \, \,
|V(x)|\le N,
\\
0, \, \, \, \text{otherwise},
\end{cases}
\end{equation}
then, of course,
\begin{equation}\label{2.14}
\|V_{\le N}\|_{L^\infty}\le N,
\end{equation}
and, if $V_{>N}(x)=V(x)-V_{\le N}(x)$,
\begin{equation}\label{2.15}
\|V_{>N}\|_{L^{n/2}(M)}\le \delta(N), 
\quad \text{with } \, \, \delta(N) \searrow
0, \, \, \, \text{as } \, \, N\to \infty,
\end{equation}
since we are assuming that $V\in L^{n/2}(M)$.

To exploit this we use the second resolvent formula
\eqref{1.4} to write
\begin{align}\label{2.16}
&(H_V-\la^2+i\mu \e(\la)\la)^{-1}f
\\
&= (-\Delta_g-\la^2+i\mu \e(\la)\la)^{-1}f
\notag
\\
&-\bigl[ \,  |V_{>N_1}|^{\frac12}] \, (-\Delta_g-\la^2-i\mu\e(\la)\la)^{-1} \, \bigr]^* \, \bigl(
\bigl(V^{1/2}\cdot (H_V-\la^2+i\mu \e(\la)\la)^{-1}f \bigr)\bigr)
\notag
\\
&-\bigl[ \,  |V_{\le N_1}|^{\frac12}] \, (-\Delta_g-\la^2-i\mu\e(\la)\la)^{-1} \, \bigr]^* \, \bigl(
\bigl((V_{>N_2})^{\frac12}\cdot (H_V-\la^2+i\mu \e(\la)\la)^{-1}f \bigr)\bigr)
\notag
\\
&-\bigl[ \,  |V_{\le N_1}|^{\frac12}] \, (-\Delta_g-\la^2-i\mu\e(\la)\la)^{-1} \, \bigr]^* \, \bigl(
\bigl((V_{\le N_2})^{\frac12}\cdot (H_V-\la^2+i\mu \e(\la)\la)^{-1}f \bigr)\bigr)
\notag
\\
&=I-II-III-IV. \notag
\end{align}
Here and for the remainder of the proof of 
Theorem~\ref{thm2.1} we are
assuming that
$$|\mu|\ge1.$$
We shall not appeal to our assumption that $\la$ is
large until the end of the proof.

By the uniform Sobolev estimates \eqref{2.6} for the 
unperturbed operator we have
\begin{equation}\label{2.17}
\|I\|_q\le C\|f\|_p.
\end{equation}
Also, by \eqref{l.1} and H\"older's inequality
\begin{equation}\nonumber
\begin{aligned}
\|II\|_q &\le C \| \, |V_{>N_1}|^{1/2} \|_{L^{n}}
\bigl\|V^{\frac12}\cdot
(H_V-\la^2+i\mu \e(\la)\la)^{-1}f
\bigr\|_{\overline{p}} \\ &\le 
C\|V_{>N_1}\|^{\frac12}_{L^{n/2}}\cdot \|V\|^{\frac12}_{L^{n/2}}
\cdot
\bigl\| (H_V-\la^2+i\mu \e(\la)\la)^{-1}f\bigr\|_{L^q},
\end{aligned}
\end{equation} 
since $\frac {1}{\overline{p}}=\frac1p-\frac1n$.
 By \eqref{2.15} we can
fix $N_1$ large enough so that
$C\|V_{>N_1}\|^{\frac12}_{L^{n/2}}\cdot \|V\|^{\frac12}_{L^{n/2}}<1/6$, yielding the bounds
\begin{equation}\label{2.18}
\|II\|_q<\frac16 \, \bigl\|
(H_V-\la^2+i\mu \e(\la)\la)^{-1}f
\bigr\|_q.
\end{equation}
Similarly, 
\begin{equation}\nonumber
\begin{aligned}
\|III\|_q &\le C \|V_{\le N_1}\|^{\frac12}_{L^{n/2}}
\bigl\|(V_{>N_2})^{\frac12}\cdot
(H_V-\la^2+i\mu \e(\la)\la)^{-1}f
\bigr\|_{\overline{p}} \\ &\le 
C\|V\|^{\frac12}_{L^{n/2}}\cdot \|V_{>N_2}\|^{\frac12}_{L^{n/2}}
\cdot
\bigl\| (H_V-\la^2+i\mu \e(\la)\la)^{-1}f\bigr\|_{L^q},
\end{aligned}
\end{equation}
by \eqref{2.15} we can
fix $N_2$ large enough so that
$C\|V\|^{\frac12}_{L^{n/2}}\cdot \|V_{>N_2}\|^{\frac12}_{L^{n/2}}<1/6$, which implies
\begin{equation}\label{2.18'}
\|III\|_q<\frac16 \, \bigl\|
(H_V-\la^2+i\mu \e(\la)\la)^{-1}f
\bigr\|_q.
\end{equation}

It remains to estimate the norm of $IV$ in \eqref{2.16}.
We first note that by 
\eqref{l.2}
\begin{multline}\label{2.19}
\|IV \|_q\le CN^{\frac12}_1\delta(\la,q)\la^{\sigma(q)-1}
(\e(\la))^{-1}
\bigl\| (V_{\le N_2})^{\frac12}\cdot (H_V-\la^2+i\mu \e(\la)\la)^{-1}f
\bigr\|_2
\\
\le CN^{\frac12}_1N^{\frac12}_2 \delta(\la,q)\la^{\sigma(q)-1}
(\e(\la))^{-1}
\|  (H_V-\la^2+i\mu \e(\la)\la)^{-1}f
\|_2.
\end{multline}

We can estimate the last factor by appealing to the second
resolvent formula one more time.  Here there is no need
to split the potential, and, instead, we write
\begin{multline}\label{2.20}
(H_V-\la^2+i\mu \e(\la)\la)^{-1}f=
(-\Delta_g-\la^2+i\mu\e(\la)\la)^{-1}f
\\
-\bigl[ \, |V|^{\frac12} (-\Delta_g-\la^2-i\mu\e(\la)\la)^{-1} \, \bigr]^*
\bigl(\bigl(V^{\frac12}\cdot (H_V-\la^2+i\mu \e(\la)\la)^{-1}f\bigr)\bigr)
=A-B.
\end{multline}
By the dual version of \eqref{2.2} with $r=p'$, we have
\begin{equation}\label{2.21}
\|A\|_2\le C\delta(\la,p') \la^{\sigma(p')-1}
(\e(\la))^{-1}\|f\|_p.
\end{equation}
Also, if $\frac {1}{\overline{p}}=\frac1p-\frac1n$, then
 by  \eqref{l.1'} and H\"older's inequality,
\begin{multline}\label{2.22}
\|B\|_2\le C\delta(\la,p') \la^{\sigma(p')-1}
(\e(\la))^{-1}
\|V\|^{\frac12}_{L^{n/2}}
\|V^{\frac12} (H_V-\la^2+i\mu \e(\la)\la)^{-1}f\bigr\|_{\overline{p}} \\
\le C\delta(\la,p') \la^{\sigma(p')-1}
(\e(\la))^{-1}
\|V\|_{L^{n/2}}
\| (H_V-\la^2+i\mu \e(\la)\la)^{-1}f\bigr\|_q.
\end{multline}

If we combine \eqref{2.20}, \eqref{2.21} and 
\eqref{2.22} and use \eqref{2.19} we
conclude that
\begin{align}\label{2.23}
\|IV \|_q &\le CN^{\frac12}_1N^{\frac12}_2 \la^{\sigma(q)+\sigma(p')-2}
\, \bigl(\e(\la)\bigr)^{-2}\, 
\delta(\la,q)\, \delta(\la,p')
\\
&\qquad\qquad \qquad
\times \bigl(\|f\|_p +\|V\|_{L^{n/2}}
\| (H_V-\la^2+i\mu \e(\la)\la)^{-1}f\bigr\|_q\bigr)
\notag
\\
&\le C\|f\|_p + \frac16 \| (H_V-\la^2+i\mu \e(\la)\la)^{-1}f\bigr\|_q,
\notag
\end{align}
by \eqref{2.5} if $\la\ge \Lambda$, with $\Lambda$
sufficiently large, 
since $N_1$ and $N_2$ have been
fixed.

If we combine \eqref{2.17}, \eqref{2.18}, \eqref{2.18'} and \eqref{2.23}, we conclude that for $\la \ge \Lambda$
we have
$$\|(H_V-\la^2+i\mu \e(\la)\la)^{-1}f\|_{L^q(M)}
\le C\|f\|_{L^p(M)}
+\frac12 \, \| (H_V-\la^2+i\mu \e(\la)\la)^{-1}f\bigr\|_{L^q(M)}.$$
By \eqref{2.12}, this leads to \eqref{2.11} since
we are assuming, as we may, that $f\in L^2(M)$.
\end{proof}

\noindent {\bf Remark:}
In dimensions $n\ge 5$ the arguments can be simplified a little bit, since, in these cases, we may appeal to the more
straightforward second resolvent formula \eqref{1.4'} instead of relying on \eqref{1.4} (as we must do for
$n=3, \, 4$).  If we do so for $n\ge5$, then we may replace \eqref{2.16} with a simpler variant
\begin{align*}
\bigl(H_V-\la^2+i\mu \e(\la)\bigr)^{-1} f &= \bigl(-\Delta_g-\la^2+i\mu \e(\la)\bigr)^{-1} f
\\
&\quad- \bigl[ \bigl(-\Delta_g-\la^2+i\mu \e(\la)\bigr)^{-1}\bigr] \, \bigl(V_{>N}\cdot (H_V-\la^2+i\mu \e(\la))f\bigr)
\\
&\quad-\bigl[ \bigl(-\Delta_g-\la^2+i\mu \e(\la)\bigr)^{-1}\bigr] \, \bigl(V_{\le N}\cdot (H_V-\la^2+i\mu \e(\la))f\bigr).
\end{align*}
Then the arguments that were used to control $II$ and $III$ in \eqref{2.16} can easily be adapted to control the second and
third terms, respectively, in the right side of the above identity.  As we alluded to earlier, we need to use the more
complicated  second resolvent formula \eqref{1.4} when $n=3,4$ due to the fact that
the {\em form} domains (but not {\em operator} domains) of $H_V$ and $H_0$ coincide in this case, while for
$n\ge5$ we may use \eqref{1.5} since, in these cases, the operator domains coincide.\footnote{We are grateful 
to one of the referees for pointing this out to us.}
\bigskip

\begin{proof}[Proof of Corollary~\ref{corr2.2}]
Let us first prove the quasimode estimates \eqref{2.8}.
To be able to use the uniform Sobolev estimates
in Theorem~\ref{thm2.1} we shall initially assume
that $\la\ge \Lambda$, where $\Lambda=\Lambda(M,q,V)\ge
1$ is as in this theorem.

Proving the quasimode estimate is equivalent to showing that for
$q$ as in \eqref{2.8} we have
$$\bigl\|(H_V-\la^2+i\e(\la)\la)^{-1}
\bigr\|_{L^2\to L^q}\le C\delta(\la,q)
\la^{\sigma(q)-1}(\e(\la))^{-1}, \quad 
\la \ge \Lambda,
$$
or, by duality,
\begin{equation}\label{2.24}
\bigl\|(H_V-\la^2+i\e(\la)\la)^{-1}f\bigr\|_{L^2(M)}
\le C\delta(\la,q)\la^{\sigma(q)-1}(\e(\la))^{-1}
\|f\|_{L^{q'}(M)}, \quad \la \ge \Lambda.
\end{equation}

To prove this we note that \eqref{2.2} and duality
yield
\begin{equation}\label{2.25}
\bigl\|(-\Delta_g-\la^2+i\e(\la)\la)^{-1}\bigr\|_{L^{q'}\to L^2} 
\le C\delta(\la,q)\la^{\sigma(q)-1}(\e(\la))^{-1},
\quad \la\ge 1,
\end{equation}
while, \eqref{2.7} yields
\begin{equation}\label{2.26}
\bigl\| (H_V-\la^2+i\e(\la)\la)^{-1}
\bigr\|_{L^{q'}\to L^{r'}}\le C, \quad
\la \ge \Lambda,
\end{equation}
since, as remarked after \eqref{2.1}, if $(p,q)$
is as in \eqref{2.1} then so is $(q',p')$.

If we use  the decomposition \eqref{2.20} again with $\mu=1$, then by 
\eqref{2.25} we can estimate the first term in the
right side of this equality as follows:
\begin{equation}\label{2.27}
\|A\|_2 \le C\delta(\la,q)\la^{\sigma(q)-1}(\e(\la))^{-1}
\|f\|_{L^{q'}(M)}.
\end{equation}
Since $\frac1{q'}-\frac1{p'}=\frac2n$, by 
\eqref{l.1'}
and H\"older's inequality we also obtain
\begin{multline*}
\|B\|_2\le C\delta(\la,q)\la^{\sigma(q)-1}(\e(\la))^{-1} \|V\|^{\frac12}_{L^{n/2}} 
\bigl\|V^{\frac12} \cdot (H_V-\la^2+i\e(\la)\la)^{-1}f\bigr\|_{\overline{q}^\prime}
\\
\le C\|V\|_{L^{n/2}} \delta(\la,q)\la^{\sigma(q)-1}
(\e(\la))^{-1} \| H_V-\la^2+i\e(\la)\la)^{-1}f\|_{p'},
\end{multline*}
if the pair $(q',p')$ is as in 
\eqref{2.1} and $\frac{1}{\overline{q}^\prime}=\frac{1}{q^\prime}-\frac1n$.

By \eqref{2.7}
$$\| (H_V-\la^2+i\e(\la)\la)^{-1}f\|_{p'} \le
C_{p',V}\|f\|_{q'}, \quad \la\ge \Lambda,
$$
and since $V\in L^{n/2}$ we conclude that $\|B\|_2$ is also dominated
by the right side of \eqref{2.24} for $\la$ as above.

To obtain the quasimode estimate \eqref{2.9} in the Corollary we need to see that the bounds
in \eqref{2.24} are also valid when $1\le \la<\Lambda$,
with $\Lambda=\Lambda(M,q,V)\ge1$ being the fixed constant in Theorem~\ref{thm2.1}.  This just follows
from the fact that $\delta(\la,q)$ and $\e(\la)$ are
assumed to be nonzero and continuous, and also
by the spectral theorem
\begin{multline}\label{spectral}
\bigl\|(H_V-\la^2+i\e(\la)\la)^{-1}f\bigr\|_{L^2(M)}
\le C
\bigl\|(H_V-\la^2+i\e(\Lambda)\Lambda)^{-1}f\bigr\|_{L^2(M)}, \\
 \text{if } \, \, 1\le \la\le \Lambda.
\end{multline}

Let us finish the proof of the Corollary by proving
\eqref{2.9}, which is equivalent to showing that
for $(p,q)$ as in \eqref{2.1} we have
\begin{equation}\label{2.28}
\|(H_V-\la^2+i\mu \e(\la)\la)^{-1}f\|_q
\le C_{\delta,V,q}
\|f\|_p, \quad
\text{if } \, \, \la \ge 1 \, \, 
\text{and } \, \, |\mu|\ge \delta.
\end{equation}
As before, we may assume that $q\in (\tfrac{2n}{n-1},
\tfrac{2n}{n-2}]$ to justify the bootstrap argument.

Since, similar to \eqref{spectral}, by the spectral theorem, we have
\begin{multline}\label{spectral1}
\bigl\|(H_V-\la^2+i\mu\e(\la)\la)^{-1}f\bigr\|_{L^2(M)}
\le C_{\delta_0}
\bigl\|(H_V-\la^2+i\e(\lambda)\lambda)^{-1}f\bigr\|_{L^2(M)}, \\
\text{if } \, \, |\mu|\ge \delta_0 \, \, \,
\text{and } \, \, \la\ge 1.
\end{multline}
Thus, by \eqref{2.8} and duality
\begin{multline}\label{2.29}
\bigl\|(H_V-\la^2+i\mu\e(\la)\la)^{-1}
\bigr\|_{L^p\to L^2}
\le C_{\delta_0}\delta(\la,p')\la^{\sigma(p')-1}
(\e(\la))^{-1}, 
\\
\text{if } \, \, |\mu|\ge \delta_0 \, \, \,
\text{and } \, \, \la\ge 1,
\end{multline}
while by \eqref{2.2} we have
\begin{multline}\label{temp}
\bigl\|(-\Delta_g-\la^2+i\mu\e(\la)\la)^{-1}
\bigr\|_{L^2\to L^q}\le C_{\delta_0}\delta(\la,q)
\la^{\sigma(q)-1}(\e(\la))^{-1}
\\ \text{if } \, \, |\mu|\ge \delta_0 \, \, \,
\text{and } \, \, \la\ge 1,
\end{multline}
Also, by \eqref{2.6}
\begin{equation}\label{2.30}
\bigl\| (-\Delta_g-\la^2+i\mu \e(\la)\la)^{-1}
\bigr\|_{L^p\to L^q}\le C_{\delta_0}, 
\quad \text{if } \, \, |\mu|\ge \delta_0 \, \, \,
\text{and } \, \, \la\ge 1.
\end{equation}

If we then split as in \eqref{2.16} and argue
as before, we find that \eqref{2.30} yields
\begin{equation}\label{2.31}
\|I\|_q\le C_{\delta_0}\|f\|_p,
\end{equation}
and
\begin{equation}\label{2.32}
\|II\|_q+\|III\|_q\le \frac12 \bigl\|
(H_V-\la^2+i\mu\e(\la)\la)^{-1}f\bigr\|_{L^q},
\end{equation}
if for the latter $N_1,\,N_2$ are fixed large
enough and $|\mu|\ge \delta_0$ and $\la\ge 1$.

If we use
\eqref{l.2}
and an earlier argument we obtain
\begin{align*}
\|IV\|_q&\le C_{\delta_0}N^{\frac12}_1N^{\frac12}_2\delta(\la,q)\la^{\sigma(q)-1}
(\e(\la))^{-1}
\bigl\|(H_V-\la^2+i\mu\e(\la)\la)^{-1}f\bigr\|_2
\\
&\le C'_{\delta_0}N^{\frac12}_1N^{\frac12}_2\delta(\la,q)\la^{\sigma(q)-1}
\delta(\la,p')\la^{\sigma(p')-1}
(\e(\la))^{-2}\|f\|_p,
\end{align*}
and since we are assuming \eqref{2.5}
this yields
\begin{equation}\label{2.34}
\|IV\|_q\le C_{\delta_0}\|f\|_p.
\end{equation}

Since \eqref{2.30}, \eqref{2.32} and \eqref{2.34} yield
\eqref{2.28} the proof is complete.
\end{proof}

Now we show another abstract theorem that gives us quasimode estimates for larger exponents.
\begin{theorem}\label{thm2.3}  Assume  $(M,g)$ is a compact Riemannian
manifold of dimension $n\ge5$.  Assume further that \eqref{2.8} holds for some $\tfrac{2(n+1)}{n-1}\leq r< \tfrac{2n}{n-4}$, with $\e(\la)$, $\delta(\la,r)$
 satisfying \eqref{2.3} and \eqref{2.4} respectively.
Then if $V\in L^{n/2}(M)$ we have for $u\in \Dom(H_V)$
\begin{multline}\label{2.60}
\|u\|_{q}\le C_{V,r} \,  \delta(\la,r) \,
\la^{\sigma(q)-1} \, \bigl(\e(\la)\bigr)^{-1}\bigl\|(-\Delta_g+V-\la^2+i \e(\la)\la)u\bigr\|_2, \\
 \, \, \, \text{if } \, \, \,\,
\la \ge 1\,\,\,\, r< q\leq \tfrac{2n}{n-4}.
\end{multline}
Similarly, for n=3 or n=4, assume that \eqref{2.8} holds for some $\tfrac{2(n+1)}{n-1}\leq r< \infty$, with $\e(\la)$, $\delta(\la,r)$
 satisfying \eqref{2.3} and \eqref{2.4}, we have
\begin{multline}\label{2.61}
\|u\|_{q}\le C_{V,r} \,  \delta(\la,r) \,
\la^{\sigma(q)-1} \, \bigl(\e(\la)\bigr)^{-1}\bigl\|(-\Delta_g+V-\la^2+i \e(\la)\la)u\bigr\|_2, \\
 \, \, \, \text{if } \, \, 
\la \ge 1,\,\,\, r< q<\infty.
\end{multline}
\end{theorem}
Here compared with the non-perturbed case \eqref{2.2}, we have $\delta(\la, r)$ on the right side of \eqref{2.60} and \eqref{2.61} instead of $\delta(\la, q)$
for larger exponents $q$. This is because we are using the bound \eqref{2.8} for the exponent $r$ in our proof. And as we can see in the first section, except for the case $q_c=\frac{2(n+1)}{n-1}$, for our applications we have $\delta(\la,q)\equiv \sqrt{\e(\la)}$ for all larger exponents in the quasimode estimates.

\begin{proof}[Proof of Theorem \ref{thm2.3}]
Throughout the proof we shall assume that 
\begin{equation}\label{2.61'}
\tfrac{2(n+1)}{n-1}\le r<q\leq \tfrac{2n}{n-4}, \,\,\,\text{if}\,\,\, n\ge 5,\,\,\,\text{or}\,\,\,\tfrac{2(n+1)}{n-1}\le r<q< \infty,\,\,\ \text{if}\,\,\, n=3, 4.
\end{equation}
Note that proving \eqref{2.60} is equivalent to showing that for $q$ satisfying \eqref{2.61'}
\begin{equation}\label{2.62}
\bigl\|(H_V-\la^2+i \e(\la)\la)^{-1}f\bigr\|_q \leq C_{V,r} \,  \delta(\la,r) \,
\la^{\sigma(q)-1} \, \bigl(\e(\la)\bigr)^{-1} \|f\|_2,
 \, \, \, \text{if } \, \, 
\la \ge 1.
\end{equation}
As before, in order to justify a bootstrapping argument that follows, we shall temporarily assume that for $q$ as in \eqref{2.61'}
\begin{equation}\label{2.63}
\bigl\| (H_V-\la^2+i\e(\la)\la)^{-1}f\bigr\|_{L^q(M)}
<\infty \quad \text{if } \, \, f\in L^2(M).
\end{equation}
We shall give the proof of \eqref{2.63} later in Lemma \ref{Sobv}  by obtaining Sobolev type inequalities for the operator $H_V$.

Fix a smooth bump function $\beta\in C_0^\infty(1/4,4)$ with $\beta\equiv 1$ in $(1/2,2)$, and let $P=\sqrt{\Delta_g}$, write
\begin{multline}\label{2.64}
(H_V-\la^2+i\e(\la)\la)^{-1}f
\\ \qquad \qquad\qquad= \beta(P/\la)(H_V-\la^2+i\e(\la)\la)^{-1}f+ \big(1-\beta(P/\la)\big)(H_V-\la^2+i\e(\la)\la)^{-1}f \\
=A+B. \qquad\qquad \qquad\qquad\qquad \qquad\qquad\qquad \qquad\qquad\qquad \qquad\qquad \quad
\end{multline}

To deal with the first term, note that since $\la^{-\alpha}\tau^\alpha \beta(\tau/\la)$ is a symbol of order 0, by Theorem 4.3.1 in \cite{SFIO2}, $\la^{-\alpha}(-\Delta_g)^{\frac{\alpha}{2}} \beta(P/\la)$ is a 0 order pseudo-differential operator, thus
\begin{equation}\label{2.65}
\|(-\Delta_g)^{\frac{\alpha}{2}} \beta(P/\la)\|_{L^r\rightarrow L^r} \lesssim \la^{\alpha},\,\,\,\,\, \text{if}\,\,\, 1<r<\infty.
\end{equation} So by Sobolev estimates, \eqref{2.65} and \eqref{2.8}, if $\alpha=n(\frac1r-\frac1q)$, we have
\begin{equation}\label{2.66}
\begin{aligned}
\|A\|_q &\leq \|(\Delta_g)^{\frac{\alpha}{2}}\beta(P/\la)(H_V-\la^2+i\e(\la)\la)^{-1}f\|_r  \\
&\le \la^{n(\frac1r-\frac1q)} \|(H_V-\la^2+i\e(\la)\la)^{-1}f\|_r \\
&\le  C_{V,r} \,  \delta(\la,r) \,
\la^{n(\frac1r-\frac1q)}\la^{\sigma(r)-1}\|f\|_2.  \\
\end{aligned}
\end{equation}
Since $n(\frac1r-\frac1q)+\sigma(r)=\sigma(q)$, the first term is dominated by the right side of \eqref{2.62}.

To bound the second term, we shall use the second resolvent formula \eqref{1.4} to write
\begin{equation}\label{2.66'}
\begin{aligned}
&\big(1-\beta(P/\la)\big)(H_V-\la^2+i \e(\la)\la)^{-1}f
\\
&= \big(1-\beta(P/\la)\big)(-\Delta_g-\la^2+i \e(\la)\la)^{-1}f
\notag
\\
&-\bigl(1-\beta(P/\la)\big)
\bigl[ |V_{>N}|^{1/2} (-\Delta_g-\la^2-i\e(\la)\la)^{-1}\bigr]^*
\bigl(V^{\frac12}\cdot (H_V-\la^2+i \e(\la)\la)^{-1}f\bigr)
\notag
\\
&-\big(1-\beta(P/\la)\big)
\bigl[ |V_{\le N}|^{1/2} (-\Delta_g-\la^2-i\e(\la)\la)^{-1}\bigr]^*
\bigl(V^{\frac12}\cdot (H_V-\la^2+i\e(\la)\la)^{-1}f\bigr)
\notag
\\
&=I-II-III. \notag
\end{aligned}
\end{equation}

Since the function $1-\beta(\tau/\la)$ vanishes in a dyadic neighborhood of $\la$, it is easy to see that 
$$\big(1-\beta(\tau/\la)\big)(\tau^2-\la^2+i \e(\la)\la)^{-1}(\tau^2+\la^2)
$$
is a symbol of order zero, again by Theorem 4.3.1 in \cite{SFIO2},
$$\big(1-\beta(P/\la)\big)(-\Delta_g-\la^2+i \e(\la)\la)^{-1}(-\Delta_g+\la^2)$$
is a 0 order pseudo-differential operator, thus
\begin{equation}\label{2.67}
\|\big(1-\beta(P/\la)\big)(-\Delta_g-\la^2+i \e(\la)\la)^{-1}f\|_r \lesssim \|(-\Delta_g+\la^2)^{-1}f\|_r,\,\,\,\,\, \text{if}\,\,\, 1<r<\infty.
\end{equation}
So by \eqref{2.67}, Sobolev estimates, the proof of \eqref{l.1} and the fact that 
\begin{multline}\bigl(1-\beta(P/\la)\big)
\bigl[ |V_{>N}|^{1/2} (-\Delta_g-\la^2-i\e(\la)\la)^{-1}\bigr]^*\\=\bigl[ |V_{>N}|^{1/2} \bigl(1-\beta(P/\la)\big)(-\Delta_g-\la^2-i\e(\la)\la)^{-1}\bigr]^*,
\end{multline}
we have for $q$ satisfying \eqref{2.61'}
\begin{equation}\label{2.68}
\begin{aligned}
\|II\|_q&\le C \|V_{>N}\|^{\frac12}_{L^{n/2}}
\bigl\|V^{\frac12}\cdot
(H_V-\la^2+i \e(\la)\la)^{-1}f
\bigr\|_{\overline{p}}  \\
&\le
C\|V_{>N}\|^{\frac12}_{L^{n/2}} \|V\|^{\frac12}_{L^{n/2}}
\cdot
\bigl\| (H_V-\la^2+i\e(\la)\la)^{-1}f\bigr\|_{L^q},
\end{aligned}
\end{equation}
where $\frac{1}{\overline{p}}-\frac1q=\frac1n$. By \eqref{2.15} we can
fix $N$ large enough so that
$C\|V_{>N}\|^{\frac12}_{L^{n/2}} \|V\|^{\frac12}_{L^{n/2}}<1/4$, yielding the bounds
\begin{equation}\label{2.69}
\|II\|_q\le \frac14 \, \bigl\|
(H_V -\la^2+i \e(\la)\la )^{-1}f
\bigr\|_q.
\end{equation}

To bound the third term, note that since
$( \frac{\la^2}{\tau^2+\la^2})^{\frac12}
$
is a symbol of order 0, by Theorem 4.3.1 in \cite{SFIO2},
$$(-\Delta_g/\la^2+1)^{-\frac12}$$
is a 0 order pseudo-differential operator, thus
if $\frac{1}{\overline{p}}=\frac1q-\frac1n$ then by Sobolev estimates
\begin{equation}\label{2.83'}
\|(-\Delta_g+\la^2)^{-1}f\|_q\le C\|(-\Delta_g+\la^2)^{-\frac12}f\|_{\overline{p}} \le C\la^{-1}\|f\|_{\overline{p}}.
\end{equation}
Thus, \eqref{2.67} and \eqref{2.83'} and our earlier arguments (i.e., the proof of Lemma~\ref{simple}) yield
\begin{equation}\label{2.84'}
\begin{aligned}
\|III\|_q\le& C\la^{-1}N^{1/2}\bigl\|V^{\frac12}\cdot
(H_V-\la^2+i \e(\la)\la)^{-1}f
\bigr\|_{\overline{p}} \\ \le& C\la^{-1}N^{1/2}\|V\|^{\frac12}_{L^{n/2}} \, \bigl\|
(H_V-\la^2+i \e(\la)\la)^{-1}f
\bigr\|_q.
\end{aligned}
\end{equation}
If we choose $\Lambda$ such that $C\Lambda^{-1}N^{1/2}\|V\|^{\frac12}_{L^{n/2}}=\frac14$, we conclude that
\begin{equation}\label{2.85'}
\|III\|_q\le \frac14 \, \bigl\|
(H_V-\la^2+i \e(\la)\la)^{-1}f
\bigr\|_q,\,\,\,\text{if}\,\,\,\la\ge\Lambda.
\end{equation}

Also note that for $q$ satisfying \eqref{2.61'}, we have $\frac12-\frac1q \leq \frac2n$. By Sobolev estimates, if $\alpha=n(\frac12-\frac1q)$
\begin{multline}\label{2.70}
\|\big(1-\beta(P/\la)\big)(-\Delta_g-\la^2+i \e(\la)\la)^{-1}f\|_q \\
\le \|(-\Delta_g)^{\frac\alpha2}\big(1-\beta(P/\la)\big)(-\Delta_g-\la^2+i \e(\la)\la)^{-1}f\|_2
\end{multline}
Since the symbol of the operator on the right side of \eqref{2.70} satisfies 
\begin{equation}\label{2.71}
\tau^\alpha \big(1-\beta(\tau/\la)\big)(\tau^2-\la^2+i \e(\la)\la)^{-1}\le \la^{\alpha-2},
\end{equation}
a combination of \eqref{2.70} and \eqref{2.71} yields the bounds
\begin{equation}\label{2.72}
\|I\|_q\le \la^{n(\frac12-\frac1q)-2}\|f\|_2,
\end{equation}
which is better than the right side of \eqref{2.60} and \eqref{2.61}, due to the condition on $\e(\la)$ and $\delta(\la,r)$.

If we combine \eqref{2.66}, \eqref{2.69}, \eqref{2.85'} and \eqref{2.72}, we conclude that for $\la \ge \Lambda$
we have
\begin{multline}\label{2.74}
\|(H_V-\la^2+i \e(\la)\la)^{-1}f\|_{L^q(M)}
\le C_{V,r} \,  \delta(\la,r) \,
\la^{n(\frac1r-\frac1q)}\la^{\sigma(r)-1}\|f\|_{L^2(M)} \\
+\frac12 \, \| (H_V-\la^2+i \e(\la)\la)^{-1}f\bigr\|_{L^q(M)}.
\end{multline}
By \eqref{2.63}, this leads to \eqref{2.60} and \eqref{2.61} for $\la \ge \Lambda$ since
we are assuming that $f\in L^2(M)$. On the other hand, by \eqref{spectral}, the quasimode estimates for $1\le \la < \Lambda$ follow as a corollary of the special case when $\la=\Lambda$.

To finish the proof of Theorem \ref{thm2.3} we shall need the following lemma which gives us \eqref{2.63}.
\begin{lemma}\label{Sobv}
 Assume  $(M,g)$ is a compact Riemannian
manifold of dimension $n\ge3$, if $V\in L^{n/2}(M)$, there exists a constant $N_0>1$ large enough such that 
\begin{equation}\label{2.76}
\|u\|_{q}\le \bigl\|(-\Delta_g+V+N_0)u\bigr\|_{p(q)},\,\,\,\text{if}\,\,\,\tfrac{1}{p(q)}-\tfrac1q=\tfrac2n \,\,\text{and}\,\, \tfrac{n}{n-2}<q< \infty.
\end{equation}
\end{lemma}

The condition on $q$ in \eqref{2.76} is necessary since we do not have the corresponding Sobolev inequalities even for the non perturbed operator at the two endpoints $p=1$ or $q=\infty$. Also observe that for $q$ satisfying \eqref{2.61'}, we have $p(q)\le 2$. Thus, by the above inequality, we have $\|u\|_{L^q(M)}<\infty$ for $q$ satisfying \eqref{2.61'} if $u\in \text{Dom}(H_V)$, which implies \eqref{2.63}.

To prove \eqref{2.76}, note that it is equivalent to showing that
\begin{equation}\label{2.77}
\bigl\| (H_V+N_0)^{-1}f\bigr\|_{L^q(M)}
\le C\|f\|_{L^p(M)}
\,\,\,\quad\text{if} \,\,\, f\in L^2(M),
\end{equation}
with $(p,q)$ as in \eqref{2.76}. 
 By duality, it suffices prove this inequality when 
\begin{equation}\label{2.78}
\tfrac{n}{n-2}<q\le \tfrac{2n}{n-2}.
\end{equation}

We are assuming \eqref{2.78}, since by \eqref{s51}
in the appendix we have
$$u\in L^q(M), \quad 2\le q\le \tfrac{2n}{n-2}
\, \, \text{if } \, \, \,
u\in \Dom(H_V).$$
Thus for $q$ as in \eqref{2.78}
\begin{equation}\label{2.79}
\bigl\| (H_V+N_0)^{-1}f\bigr\|_{L^q(M)}
<\infty \quad \text{if } \, \, f\in L^2(M).
\end{equation}

As before, in proving \eqref{2.77}, since $L^2$ is dense in $L^p$,
we shall assume that $f\in L^2(M)$ to be
able to use \eqref{2.12} to justify a bootstrapping
argument that follows.

We shall use the second resolvent formula \eqref{1.4} to write 
\begin{align}\label{2.80}
(H_V+N_0)^{-1}f
= &(-\Delta_g+N_0)^{-1}f
-\bigl[ |V_{>N}|^{1/2}(-\Delta_g+N_0)^{-1}\bigr]^*
\bigl(V^{\frac12}\cdot (H_V+N_0)^{-1}f\bigr)
\notag
\\
&\qquad-\bigl[ |V_{\le N}|^{1/2}(-\Delta_g+N_0)^{-1}\bigr]^*
\bigl(V^{\frac12}\cdot (H_V+N_0)^{-1}f\bigr)
\\
=&I-II-III. \notag
\end{align}

By the Sobolev estimates for the 
unperturbed operator we have
\begin{equation}\label{2.81}
\|I\|_q\le C\|f\|_p,
\end{equation}
where the constant $C$ does not depend on $N_0$. Similarly our earlier arguments yield
$$\|II\|_q\le C \|V_{>N}\|^{\frac12}_{L^{n/2}}
\bigl\|V^{\frac12}\cdot
(H_V+N_0)^{-1}f
\bigr\|_{\overline{p}} \le 
C\|V_{>N}\|^{\frac12}_{L^{n/2}} \|V\|^{\frac12}_{L^{n/2}}
\cdot
\bigl\| (H_V+N_0)^{-1}f\bigr\|_{L^q},
$$
using H\"older's inequality and the fact that $\frac{1}{\overline{p}}=\frac1q+\frac1n$ in the last step.  By \eqref{2.15} we can
fix $N$ large enough so that
$C\|V_{>N}\|^{\frac12}_{L^{n/2}} \|V\|^{\frac12}_{L^{n/2}}
<1/4$, yielding the bounds
\begin{equation}\label{2.82}
\|II\|_q<\frac14 \, \bigl\|
(H_V+N_0)^{-1}f
\bigr\|_q.
\end{equation}

To bound the third term, note that since
$( \frac{N_0}{\tau^2+N_0})^{\frac12}
$
is a symbol of order 0, by Theorem 4.3.1 in \cite{SFIO2},
$$(-\Delta_g/N_0+1)^{-\frac12}$$
is a 0 order pseudo-differential operator, thus
\begin{equation}\label{2.83}
\|(-\Delta_g+N_0)^{-1}f\|_q\le C\|(-\Delta_g+N_0)^{-\frac12}f\|_{\overline{p}} \le CN_0^{-1/2}\|f\|_{\overline{p}},
\end{equation}
using Sobolev estimates and the fact that $\frac{1}{\overline{p}}=\frac1q+\frac1n$ in the first inequality. Thus, 
\begin{equation}\label{2.84}
\begin{aligned}
\|III\|_q\le& CN_0^{-1/2}N^{1/2}\bigl\|V^{\frac12}\cdot
(H_V+N_0)^{-1}f
\bigr\|_{\overline{p}} \\ \le& CN_0^{-1/2}N^{1/2}\|V\|^{\frac12}_{L^{n/2}} \, \bigl\|
(H_V+N_0)^{-1}f
\bigr\|_q.
\end{aligned}
\end{equation}

If we choose $N_0$ such that $CN_0^{-1/2}N^{1/2}\|V\|^{\frac12}_{L^{n/2}}<\frac14$ , \eqref{2.81}, \eqref{2.82} and \eqref{2.84} imply
$$\|(H_V+N_0)^{-1}f\|_{L^q(M)}
\le C\|f\|_{L^p(M)}
+\frac12 \, \| (H_V+N_0)^{-1}f\bigr\|_{L^q(M)}.$$
By \eqref{2.79}, this leads to \eqref{2.76}, the proof is complete.
\end{proof}

Let us next show how Theorem~\ref{unifSob} is also a corollary of Theorem~\ref{thm2.1} and Theorem~\ref{thm2.3}.

\begin{proof}[Proof of Theorem~\ref{unifSob}]
We shall use Theorem~\ref{thm2.1} with
\begin{equation}\label{2.35}
\delta(\la,r)=\e(\la)\equiv 1, \, \, \, \la\ge 1,
\end{equation}
and $r=q$ and $r=p= p(q)'$ satisfying \eqref{1.5}.

Then by the spectral projection estimates of the fourth author \cite{sogge88} we have the quasimode estimates
\eqref{2.2} for the unperturbed operators $H_0=-\Delta_g$.  The uniform Sobolev estimates \eqref{2.7}
are due to Dos Santos Ferreira, Kenig and Salo~\cite{DKS}.
Also, it is a simple exercise using \eqref{1.12} to check that for $(p,q)$ as above we have
$\sigma(q)+\sigma(p')-2<0$, and so \eqref{2.5} is also
trivially valid.  

Thus, by inequality \eqref{2.8}
in Corollary~\ref{corr2.2} and Theorem~\ref{thm2.3}, we have \eqref{qm} for
$q\in [\tfrac{2(n+1)}{n-1}, \tfrac{2n}{n-4}]$ if $n\ge 5$, and $q\in [\tfrac{2(n+1)}{n-1}, \infty)$ if $n=3, \text{or}\,\, 4$.
If we use the bound for $q=\tfrac{2(n+1)}{n-1}$ along
with H\"older's inequality and the trivial quasimode
estimate for $q=2$ (which follows from the spectral theorem), we also see that \eqref{qm} is valid
for $2<q<\tfrac{2(n+1)}{n-1}$.

The other inequality in Corollary~\ref{corr2.2},
\eqref{2.9}, also trivially implies the uniform
Sobolev estimates \eqref{1.6} in the region where
$\Re \zeta \ge 1$.  Since the bounds for $\{\zeta \in \Omega_\delta: \, \, \Re \zeta<1\}$ are valid for the unperturbed operators $H_0=-\Delta_g$ by \cite{DKS} we can use the quasimode estimates \eqref{qm} for $\la=1$
and the proof that \eqref{2.7} implies \eqref{2.9} to see that the uniform Sobolev bounds in Theorem~\ref{unifSob} in the region where $\Re \zeta<1$ are also valid,
which finishes the proof. \end{proof}

%

Next, let us also see how we can use Theorem~\ref{thm2.1} and Theorem~\ref{thm2.3} to prove Theorem~\ref{sphere} which says that when $(M,g)$ is the standard
sphere we can improve Theorem~\ref{unifSob} by obtaining the inequalities for a larger range of exponents when $V\in L^{n/2}(S^n)$.

\begin{proof}[Proof of Theorem~\ref{sphere}]
It is easy to modify the proof of Theorem~\ref{unifSob} to obtain the uniform Sobolev estimates for $S^n$ which involve 
the improved range of exponents in \eqref{1.17}.    As in the preceding proof we shall use Theorem~\ref{thm2.1}
with $\delta(\la,r)=\e(\la)\equiv 1$ when $\la\ge1$.  Here $r=q$ and $r=p=p(q)'$ are assumed to be
as in \eqref{1.17}.  A simple calculation using \eqref{1.12} then shows that we have $\sigma(q)+\sigma(p')-2\in [-1, -1+1/2n]$
and so \eqref{2.5} is trivially valid.  As a result, $\text{for}\,\, q<\tfrac{2n}{n-3}$, we would have the bounds in \eqref{1.21} and \eqref{1.22} when $\Re \zeta$ and
$\la$ are larger than one, respectively, if we had the quasimode estimates \eqref{2.2}  and the uniform Sobolev estimates \eqref{2.7}
for the unperturbed operators
$H_0$, for $\e(\la)$ and $\delta(\la,r)$ as above and exponents satisfying \eqref{1.17}.  
The quasimode estimates are due to the fourth author
\cite{sogge86} (see also \cite{SHSo}), and the uniform Sobolev estimates are due to S.\ Huang and this author~\cite{SHSo}. 

Since the remaining larger exponents $q$ in \eqref{1.22} follows from the case $q<\tfrac{2n}{n-3}$ and Theorem~\ref{thm2.3}, and the cases where $\zeta\in \Omega_\delta$ has $\Re \zeta<1$ or $\la\ge 1$ in \eqref{1.21} follow from our earlier arguments, 
the proof is complete.
\end{proof}

\subsection*{ Spectral projection estimates for larger exponents}

${}$ 
\newline

 Let us conclude this section by briefly reviewing how if, in addition to assuming \eqref{1.2} (i.e., $V\in L^{n/2}$), we assume
 that $V_{-}=\max\{0, -V\} \in {\mathcal K}(M)$, then we can obtain spectral projection and quasimode estimates for exponents which
 are larger than those in Theorems \ref{unifSob}--\ref{sphere} or Corollary~\ref{corr2.2}.
 
 Recall that $V$ is in the Kato class ${\mathcal K}(M)$ if
\begin{equation}\label{kato}\lim_{r\searrow 0} \sup_{x}\int_{B_r(x)}
h_n\bigl(d_g(x,y)\bigr) \, |V(y)| \, dy=0,
\end{equation}
where 
\begin{equation*}
h_n(r)=
\begin{cases}
\log(2+ r^{-1}), \quad \text{if } \, n=2
\\
r^{2-n},\quad \text{if } \, n\ge3.
\end{cases}
\end{equation*}
Here $d_g(x,y)$ is the geodesic distance between $x$ and $y$ in $M$ and $B_r(x)$ denotes the geodesic ball of radius
$r$ about $x$.

Let us first show that we can use estimates like \eqref{2.8} to obtain certain spectral projection estimates. Specifically, if 
\begin{equation}\label{chi}
\chi^V_{[\la,\la+\e(\la)]}=\1_{[\la,\la+\e(\la)]}(\sqrt{H_V})
\end{equation}
is the projection onto the part of the spectrum of $\sqrt{H_V}$ in the interval $[\la,\la+\e(\la)]$, then, by the spectral theorem \eqref{2.8}
implies that
\begin{equation}\label{2.36}
\|\chi^V_{[\la,\la+\e(\la)]}f\|_r \le C_V \delta(\la,r) \la^{\sigma(r)} \|f\|_2, \, \, \, \la\ge 1.
\end{equation}
To see this one takes $u$ in \eqref{2.8} to be $\chi^V_{[\la,\la+\e(\la)]}f$ and then uses the spectral theorem to see that
that for this choice of $u$ the right side of \eqref{2.8} is dominated by the right hand side of \eqref{2.36}. 

Next we recall that, if $V_{-}\in {\mathcal K}(M)$, then we have favorable heat kernel bounds (see \cite{Sturm}), and, consequently, if $\beta\in C^\infty_0((1/2,1))$
is a nonnegative function with integral one and if 
$$\widetilde \beta_\la(\tau)=\int_0^\infty e^{-t\tau} \la^2 \beta(\la^2 t) \, dt, \quad
\tau\ge 0, \, \, \la\ge 1,
$$
we have
\begin{equation}\label{2.37}
\bigl\| \widetilde \beta_\la(H_V)\bigr\|_{L^r\to L^q}\lesssim \la^{n(\frac1r-\frac1q)}, \quad \text{if } \, 
2\le r\le q\le \infty.
\end{equation}
For details see \S 6 of \cite{BSS}\footnote{In \cite{BSS} this inequality was only proved under the 
stronger assumption that $V\in {\mathcal K}$; however, since the proof only relied on the heat kernel estimates of
Sturm~\cite{Sturm} which are valid when $V_-\in {\mathcal K}$, it also yields \eqref{2.37}}.  
Arguing as in \cite{BSS} it is a simple matter to use the spectral theorem
and \eqref{2.37} to see that if \eqref{2.36} is valid then we have
\begin{equation}\label{2.38}
\bigl\| \chi^V_{[\la,\la+\e(\la)]}f\bigr\|_q \lesssim \delta(\la,r)\la^{\sigma(r)+n(\frac1r-\frac1q)}\|f\|_2, \, \, \la\ge 1, \, \, \,
\text{if } \, \, q\in (r,\infty]
\end{equation}
when $V_{-}\in {\mathcal K}(M)$.

Based on this and the aforementioned relationships between spectral projection estimates and quasimode estimates, 
if $V\in L^{n/2}(M)$ and $V_{-}\in {\mathcal K}(M)$, by Theorem~\ref{unifSob}, for all $(M,g)$ we can also obtain
\eqref{2.36} with $\e(\la)=\delta(\la,r)\equiv 1$ when $r>\tfrac{2n}{n-4}$ if $n\ge 5$, or $r=\infty$ if $n=3,\, 4$, since $\sigma(r)+n(\tfrac1r-\tfrac1q)=\sigma(q)$ if
$\tfrac{2(n+1)}{n-1}\le r<q\le \infty$.  Thus, for such exponents we recover the universal bounds
in \cite{BSS} while for smaller ones the ones Theorem~\ref{unifSob} is stronger since it only requires $V\in L^{n/2}(M)$.

In the case of the standard sphere $S^n$, if $V\in L^{n/2}(M)$ and $V_{-}\in{\mathcal K}(M)$ , we can similarly obtain \eqref{2.36} with $\e(\la)=\delta(\la,r)\equiv 1$ for $r=\infty$ when
$n=3,\,4$, and $r> \tfrac{2n}{n-4}$ when $n\ge5$.



We note that Theorem~\ref{unifSob} says that when $n=3$ or $n=4$ we have \eqref{2.36} with $\e(\la)=\delta(\la,r)\equiv 1$ 
for all $2<r<\infty$.  As noted in \cite{BSS}, such spectral projection estimates can break down for $r=\infty$ on $S^n$ in all dimensions
if one merely assumes $V\in L^{n/2}(S^n)$, and there is related recent results for general manifolds in Frank and Sabin~\cite{FS}.


We have focused here on variants of the spectral projection estimates for larger exponents than the ones in Theorems \ref{unifSob} and
\ref{sphere}.  As we shall see in the next two sections, there are similar results corresponding to Theorems \ref{nonpossob}
and \ref{torussob}.

\newsection{Improved bounds for manifolds of nonpositive curvature}\label{nonpossec}

The main purpose of this section is to prove Theorem~\ref{nonpossob}.  Consequently, we shall assume throughout this section that $n\ge 3$ and that $(M,g)$ is
an $n$-dimensional manifold all of whose sectional curvatures are nonpositive.  In \S~\ref{2d} we shall prove that the quasimode estimates in Theorem~\ref{nonpossob} are valid in the two dimensional case if in addition to \eqref{1.2} we assume that $V$
is a Kato potential.

By Corollary~\ref{corr2.2} and Theorem~\ref{thm2.3} we would have
Theorem~\ref{nonpossob} if we knew that
for $(p,q)$ exponents satisfying
\begin{equation}\label{3.1}
\min\bigl(p',q)>\tfrac{2(n+1)}{n-1}
\quad \text{and } \, \, \tfrac1p-\tfrac1q=\tfrac2n
\end{equation}
we had the classical quasimode estimates
\begin{equation}\label{3.2}
\|u\|_r\lesssim \la^{\sigma(r)-1}\, \bigl(\e(\la)\bigr)^{-1/2}
\|(-\Delta_g-(\la+i\e(\la))^2\|_2, \, \,
\text{for } \, r=q,p', \, \, \text{and } \, \la\ge 1,
\end{equation}
as well as
\begin{equation}\label{3.3}
\|u\|_q\lesssim \|(-\Delta_g-(\la+i\e(\la))^2)u\|_p,
\quad \la\ge1,
\end{equation}
where here and throughout this section we shall take
\begin{equation}\label{3.4}
\e(\la)=\bigl(\log(2+\la)\bigr)^{-1}.
\end{equation}
Even though we have replaced $\la^2+i\e(\la)\la$ by
$(\la+i\e(\la))^2$ here to simplify some calculations
to follow, \eqref{3.2} and \eqref{3.3} are equivalent to
\eqref{2.2} and \eqref{2.6}, respectively, with 
$\delta(\la,r)=\sqrt{\e(\la)}$ as in \eqref{3.4} in the former.

Even though the first inequality is a consequence of spectral projection estimates in
Hassell and Tacy~\cite{HassellTacy} following earlier
results of B\'erard~\cite{Berard} and even though  the 
resolvent estimates are in \cite{BSSY} and \cite{ShYa},
let us sketch their proofs since we shall need to adapt
them in order to show that we also get improved 
quasimode estimates for $q=q_c=\tfrac{2(n+1)}{n-1}$, which is missing in \eqref{3.2}.  We cannot appeal
to Corollary~\ref{corr2.2} to obtain these estimates
since it is not known whether the uniform Sobolev estimates \eqref{3.3} are valid when $q=q_c$.  The quasimode estimates for this exponent are analogs involving
$L^{n/2}$ potentials of those of two of us in
\cite{SBLog}, which treated the $V\equiv0$ case.

Let us start with the sketch of \eqref{3.2}.  Since
both $r=p'$ and $r=q$ in \eqref{3.2} are smaller than
$2n/(n-4)$ when $n\ge 4$, by the discussion at the end of the last section, it is simple to see that \eqref{3.2}
is equivalent to the spectral projection estimates
for the unperturbed operator $H_0=-\Delta_g$:
\begin{equation}\tag{3.2$'$}\label{3.2'}
\|\chi_{[\la, \la +\e(\la)]} f\|_{r} \lesssim
\sqrt{\e(\la)} \, \la^{\sigma(r)} \, \|f\|_2,
\, \, \la\ge 1, \, \, r> \tfrac{2(n+1)}{n-1}, \, \,
\end{equation}
with $r$ as in \eqref{3.2} (see \cite{SoggeZelditchQMNote}).  We shall actually indicate why this inequality is valid for all $r>\tfrac{2(n+1)}{n-1}$.  Here $\chi_{[\la, \la +\e(\la)]}$ is the operator projecting onto the part of the spectrum of $\sqrt{-\Delta_g}$ in the shrinking 
intervals $[\la, \la +\e(\la)]$.

To establish this fix a  real-valued function $a
\in {\mathcal S}(\R)$
satisfying
\begin{equation}\label{3.5}
\text{supp } \Hat a \subset (-\delta_0,\delta_0)
\quad \text{and  } \, \, a(t)\ge 1, \, \, t\in [-1,1],
\end{equation}
where $\delta_0>0$ will be specified later on.
We then claim that \eqref{3.2'} would be a consequence of the following:
\begin{equation}\label{3.6}
\bigl\| a\bigl((\e(\la))^{-1}(P-\la)\bigr)h\bigr\|_r\lesssim \sqrt{\e(\la)} \, 
\la^{\sigma(r)}\, \|h\|_2, \, \, \la\ge 1, \,
r> \tfrac{2(n+1)}{n-1}, \quad
\text{if } \, \, 
P=\sqrt{-\Delta_g}.
\end{equation}
To verify this claim one just takes $h$ to be
$\widetilde\chi_{[\la, \la +\e(\la)]} f$
where $$\widetilde \chi_{[\la,\la+\e(\la)]}(\tau)=\1_{[\la,\la+\e(\la)]}(\tau) \cdot
\bigl(a\bigl((\e(\la))^{-1}(\la-\tau)\bigr)\bigr)^{-1}.$$
Since this function has sup-norm
smaller than one
and since $a\bigr((\e(\la))^{-1}(P-\la)\bigr)h= \chi_{[\la, \la +\e(\la)]} f$, one obtains \eqref{3.2'} from \eqref{3.6}
and the spectral theorem.

We next observe that, by duality, \eqref{3.6} is
equivalent to the statement that 
$$\bigl\| a\bigl((\e(\la))^{-1}(P-\la)\bigr)h\bigr\|_{L^{r'}(M)\to L^2(M)} \lesssim \sqrt{\e(\la)} \,\la^{\sigma(r)}, \, \, \la\ge1 \quad \text{if } \, r>\tfrac{2(n+1)}{n-1}.$$
By a routine $TT^*$ argument this is equivalent to the
following
\begin{multline}\label{3.6'}\tag{3.6$'$}
\bigl\| b\bigl((\e(\la))^{-1}(P-\la)\bigr)\bigr\|_{L^{r'}(M)\to L^r(M)}\lesssim \e(\la)\la^{2\sigma(r)}, \\\la\ge 1, \quad\text{if } \, r>\tfrac{2(n+1)}{n-1}, 
\, \, \text{and } \, b(\tau)=\bigl(a(\tau)\bigr)^2.
\end{multline}

Next, since, by the first part
of \eqref{3.5}, $\Hat b$ is 
supported in $(-2\delta_0,2\delta_0)$, it follows from Fourier's inversion theorem, Euler's formula and the first part
of \eqref{3.5} that
\begin{multline}\label{3.7}
b\bigl((\e(\la))^{-1}(P-\la)\bigr)h
=\frac{\e(\la)}{\pi}
\int_{-T}^T \Hat b\bigl(\e(\la)t\bigr)\,
 e^{-it\la} \, \bigl(\cos tP\bigr)h \, dt 
\\+b\bigl((\e(\la))^{-1}(
P+\la)\bigr)h,
\quad \text{where } \, \, T=2\delta_0 \cdot(\e(\la))^{-1}.
\end{multline}
Since $\la\ge1$ and $P\ge0$ using crude eigenfunction bounds one obtains
$$\bigl\| b\bigl((\e(\la))^{-1}(
P+\la)\bigr)\bigr\|_{L^1(M)\to L^\infty(M)} =O(\la^{-N}), \quad \la\ge1, \, \, \, N=1,2,3,\dots,$$
and consequently we would have \eqref{3.6'} if we could show that for small enough fixed $\delta_0>0$ we have
\begin{multline}\label{3.8}
\Bigl\| 
\int_{-T}^T \Hat b\bigl(\e(\la)t\bigr)\,
 e^{-it\la} \, \cos tP \, dt 
\Bigr\|_{L^{r'}(M)\to L^r(M)} \lesssim \la^{2\sigma(r)}, \, \, \la\ge 1, \\ \text{if } \, r>\tfrac{2(n+1)}{n-1},
\quad \text{and } \, \, T=2\delta_0 \cdot (\e(\la))^{-1}.
\end{multline}

Next, let us fix $\eta\in C^\infty_0(\R)$ satisfying
\begin{equation}\label{3.9}
\eta(t)=1, \, \, t\in (-1/2,1/2) \quad \text{and } \, \, \text{supp }\eta \subset (-1,1).
\end{equation}
Then it follows from the universal spectral projection estimates of one of us \cite{sogge88} that
\begin{equation}\label{3.10}
\Bigl\| \int \eta(t) \, \Hat b\bigl(\e(\la)t\bigr)\,
 e^{-it\la} \, (\cos tP)f \, dt \Bigr\|_{r}\lesssim \la^{2\sigma(r)} \|f\|_{r'}, \, \, \la\ge 1,
 \end{equation}
 for all $r>2$.  Consequently we would have \eqref{3.8} if we could show that when
 $\delta_0$ as in \eqref{3.5} and \eqref{3.8} is sufficiently small we have
\begin{equation}\label{3.11}
\Bigl\| 
\int (1-\eta(t))\, \Hat b\bigl(\e(\la)t\bigr)\,
 e^{-it\la}  \cos tP \, dt 
\Bigr\|_{L^{r'}(M)\to L^r(M)} \lesssim \la^{2\sigma(r)}, \,  \la\ge 1,  \text{if }  r>\tfrac{2(n+1)}{n-1}.
\end{equation}

Since the function
$$\tau \to \Psi_\la(\tau)=\int (1-\eta(t))\, \Hat b\bigl(\e(\la)t\bigr)\,
 e^{-it\la}  \cos t\tau \, dt$$
 clearly satisifes
 $$|\Psi_\la(\tau)|\lesssim (\e(\la))^{-1},$$
 it follows from the spectral theorem that
 \begin{equation}\label{3.12}
\Bigl\| 
\int (1-\eta(t))\, \Hat b\bigl(\e(\la)t\bigr)\,
 e^{-it\la}  \cos tP \, dt 
\Bigr\|_{L^{2}(M)\to L^2(M)} \lesssim (\e(\la))^{-1}=\log(2+\la).
\end{equation}
We claim that if we also had for some $c_0<\infty$
 \begin{equation}\label{3.12.1}
\Bigl\| 
\int (1-\eta(t))\, \Hat b\bigl(\e(\la)t\bigr)\,
 e^{-it\la}  \cos tP \, dt 
\Bigr\|_{L^{1}(M)\to L^\infty(M)} \lesssim \la^{\frac{n-1}2} e^{c_0T}\lesssim \la^{\frac{n-1}2} \la^{c_0\delta_0},
\end{equation}
then for $\delta_0$ small enough depending on $r$, we would have \eqref{3.10}.  This just follows from
a simple interpolation argument and the observation that if $\theta =2/r$ then
$(1-\theta)\cdot \tfrac{n-1}2<2\sigma(r)$ provided that $r>\tfrac{2(n+1)}{n-1}$.

One can prove \eqref{3.12.1} using the Hadamard parametrix after lifting the calculation to the universal
cover of $(M,g)$ as in B\'erard~\cite{Berard} and Hassell and Tacy~\cite{HassellTacy} (see also \cite{SoggeHangzhou}).
This completes the proof of \eqref{3.6'} and hence that of \eqref{3.2}.

\medskip

The proof of \eqref{3.3} is very similar.  As in \cite[\S 2]{BSSY} we shall use the formula
\begin{equation}\label{3.13}
\bigl(-\Delta_g-(\la+i\e(\la))^2\bigr)^{-1}f
=\frac{i}{\la+i\e(\la)}\int_0^\infty e^{i\la t}e^{-\e(\la)t}\, (\cos tP)f \, dt.
\end{equation}

If $\eta$ is as in \eqref{3.9}, we shall write
$$\bigl(-\Delta_g-(\la+i\e(\la))^2\bigr)^{-1}=T^0_\la+T^1_\la+R_\la,$$
where if $T=2\delta_0\cdot (\e(\la))^{-1}$ is as in \eqref{3.7}
\begin{equation}\label{3.14}
T^0_\la = \frac{i}{\la+i\e(\la)}\int_0^\infty \eta(t) \, \eta(t/T) \, e^{i\la t}e^{-\e(\la)t}\, \cos tP \, dt
\end{equation}
is a local operator, while
\begin{equation}\label{3.15}
T^1_\la = \frac{i}{\la+i\e(\la)}\int_0^\infty (1-\eta(t)) \, \eta(t/T) \, e^{i\la t}e^{-\e(\la)t}\, \cos tP \, dt,
\end{equation}
and
\begin{equation}\label{3.16}
R_\la = \frac{i}{\la+i\e(\la)}\int_0^\infty (1- \eta(t/T)) \, e^{i\la t}e^{-\e(\la)t}\, \cos tP \, dt
\end{equation}

To prove \eqref{3.3}, by duality, it suffices to handle the case where $q\in (\tfrac{2(n+1)}{n-1}, \tfrac{2n}{n-2}]$, in which case
the estimate is equivalent to the statement that
\begin{multline}\label{3.3'}\tag{3.3$'$}
\bigl\| \, (-\Delta_g-(\la+i\e(\la))^2)^{-1}\, \bigr\|_{L^{p(q)}(M)\to L^q(M)}=O(1), \, \, \la\ge 1
\\
\text{if } \, q\in (\tfrac{2(n+1)}{n-1},\tfrac{2n}{n-2}] \quad \text{and } \, \, \tfrac1{p(q)}-\tfrac1q=\tfrac2n.
\end{multline}
In view of the above decomposition, this would follow from
\begin{equation}\label{3.17}
\|S_\la\|_{L^{p(q)}(M)\to L^q(M)}=O(1)
\quad \text{if } \, \, S_\la=T^0_\la, \, T^1_\la \, \, \text{or } \, R_\la.
\end{equation}

As observed in \cite{ShYa} the bounds for $R_\la$ are in immediate consequence of \eqref{3.2'} and
a simple orthogonality argument after observing that $\sigma(q)+\sigma((p(q))')=1$ if $(p(q),q)$ are as in
\eqref{3.3'} and
$$\tau\to m_\la(\tau)=\frac{i}{\la+i\e(\la)}\int_0^\infty (1- \eta(t/T)) \, e^{i\la t}e^{-\e(\la)t}\, \cos t\tau \, dt$$
satisfies
\begin{equation}\label{blah}
|m_\la(\tau)|\lesssim (\e(\la))^{-1} \bigl(1+(\e(\la))^{-1}|\la-\tau|\bigr)^{-N} \, \, \forall \, N, \, \, 
\text{if } \, \tau\ge 0, \, \, \la\ge 1,
\end{equation}
assuming that, as above, $T=2\delta_0\cdot (\e(\la))^{-1}$.

The local operator $T_\la^0$ was estimated in 
\cite{DKS}  and then later in \cite{BSSY} (see also \cite{sogge88}) where it was
shown that this operator enjoys the bounds in \eqref{3.17} even for the larger range of exponents
where $q>\tfrac{2n}{n-1}$.  One proves this result using stationary phase and Stein's oscillatory integral theorem in
\cite{steinbeijing}.  For this step it is convenient to assume, as we may, that the injectivity radius of $(M,g)$ is ten or more.

Based on this only one estimate in \eqref{3.17} remains.  We just need to handle $T^1_\la$, i.e., if 
$T=2\delta_0\cdot \log(2+\la)$ with $\delta_0$ small enough
$$\la^{-1}
\Bigl\| \, \int_0^\infty (1-\eta(t)) \, \eta(t/T) e^{i\la t} e^{-\e(\la)t} \cos tP\, dt\,
\Bigr\|_{L^{p(q)}(M)\to L^q(M)}=O(1).$$
Since $q\le (p(q))'$ if $(p(q),q)$ are is in \eqref{3.3'} or \eqref{3.16}, by H\"older's inequality, this would follow from
\begin{equation}\label{3.17.1}
\Bigl\|  \int_0^\infty (1-\eta(t))\eta(t/T) e^{i\la t} e^{-\e(\la)t} \cos tP\, dt
\Bigr\|_{L^{r}(M)\to L^{r'}(M)}=O(\la), \,  \text{if } \, 
r'<\tfrac{2n(n+1)}{n^2-n-4},
\end{equation}
assuming that $\delta_0>0$ is small.  Here, we are using the fact that
$(p(q))'<\tfrac{2n(n+1)}{n^2-n-4}$ (see \eqref{3.1}).

One can repeat the proof of \eqref{3.12} to see that
\begin{equation}\label{3.18}
\Bigl\| \, \int_0^\infty (1-\eta(t))\eta(t/T) e^{i\la t} e^{-\e(\la)t} \cos tP\, dt\,
\Bigr\|_{L^{2}(M)\to L^{2}(M)}=O(T)=O(\log(2+\la)).
\end{equation}
Also, by using the Hadamard parametrix and arguing as in \cite{Berard} one can adapt the proof of
\eqref{3.12.1} to see that
\begin{equation}\label{3.19}
\Bigl\| \, \int_0^\infty (1-\eta(t))\eta(t/T) e^{i\la t} e^{-\e(\la)t} \cos tP\, dt\,
\Bigr\|_{L^{1}(M)\to L^{\infty}(M)}=O(\la^{\frac{n-1}2} \, \la^{c_0\delta_0}),\end{equation}
if $\delta_0>0$ is small.  Since
$$\tfrac{2n(n+1)}{n^2-n-4}<\tfrac{2(n-1)}{n-3},$$
we have 
$$\tfrac{n-1}2 \cdot (1-\theta)<1 \quad \text{if } \, 
\theta =2/r' \, \, \text{and } \, \, r'<\tfrac{2(n-1)}{n-3},
$$
we obtain \eqref{3.17} via interpolation if $\delta_0=\delta_0(r')$ is small enough.

This completes our proof of Theorem~\ref{nonpossob} except for the quasimode estimates \eqref{3.20} for
the critical exponent $q=q_c=\tfrac{2(n+1)}{n-1}$, which we shall handle in the next subsection.


\subsection*{Improved quasimode bounds for the critical exponent}

As we noted before, we cannot appeal to Corollary~\ref{corr2.2} to obtain
improved quasimode estimates for the critical exponent $q_c=\tfrac{2(n+1)}{n-1}$ on manifolds of 
nonpositive curvature since we do not have the uniform Sobolev estimates \eqref{3.3}
when $q=q_c$.
Despite this, we can use the above arguments to obtain \eqref{3.20} which extends the critical quasimode estimates of two of us
\cite{SBLog} for the $V\equiv0$ case to include singular potentials when $n\ge3$.  In a later section we shall prove analogous
estimates for the two-dimensional case.

%
%
%

To prove the quasimode estimates in \eqref{3.20}, we shall of course use the fact that by \cite{SBLog} we have \eqref{3.20} when $V\equiv0$ which
is equivalent to the following
\begin{equation}\label{3.22}
\bigl\| \, (-\Delta_g-(\la+i\e(\la))^2)^{-1}\, \bigr\|_{L^2(M)\to L^{q_c}(M)} \lesssim
\la^{\sigma(q_c)-1}\, \bigl(\e(\la)\bigr)^{-1+\delta_n},
\end{equation}
as well as the following bounds for the spectral projection operators associated to $H_0=-\Delta_g$:
\begin{equation}\label{3.23}
\bigl\| \chi_{[\la,\la+\e(\la)] }\bigr\|_{L^2(M)\to L^{q_c}(M)} \lesssim \la^{\sigma(q_c)} \, \bigl(\e(\la)\bigr)^{\delta_n}.
\end{equation}

To proceed, just as before we shall write
\begin{equation}\label{3.27.1}
\bigl(-\Delta_g-(\la+i\e(\la))^2\bigr)^{-1}=T_\la+R_\la, \quad
\text{where } \, T_\la=T_0+T^1_\la,\end{equation}
with $T_\la^0$, $T^1_\la$ and $R_\la$ as in \eqref{3.14}, \eqref{3.15} and \eqref{3.16}, respectively.

Since $R_\la=m_\la(\sqrt{H_0})$ with $m_\la(\tau)$ as in \eqref{blah} one can use
\eqref{3.23} and a simple orthogonality argument to see that
\begin{equation}\label{3.24}
\|R_\la\|_{L^2(M)\to L^{q_c}(M)} \lesssim \bigl(\e(\la)\bigr)^{-1+\delta_n} \,  \la^{\sigma(q_c)-1},
\end{equation}
and also
\begin{equation}\label{3.25}
\bigl\| R_\la \circ (-\Delta_g-(\la+i\e(\la))^2)\bigr\|_{L^2(M)\to L^{q_c}(M)}
\lesssim  \bigl(\e(\la)\bigr)^{-1+\delta_n} \,  \la^{\sigma(q_c)-1} \cdot (\la \,  \e(\la)).
\end{equation}

If we set $T_\la=T^0_\la+T^1_\la$ as above, then since 
$T_\la=(-\Delta_g-(\la+i\e(\la))^2)^{-1}-R_\la$, we trivially obtain from \eqref{3.22} and \eqref{3.24} the bound
\begin{equation}\label{3.26}
\|T_\la\|_{L^2(M)\to L^{q_c}(M)}\lesssim \bigl(\e(\la)\bigr)^{-1+\delta_n} \, \la^{\sigma(q_c)-1}.
\end{equation}

We noted before that
$$\|T^0_\la\|_{L^{p(q_c)}(M)\to L^{q_c}(M)}=O(1) \quad 
\text{if } \, \, \tfrac1{p(q_c)}-\tfrac1{q_c}=\tfrac2n.$$
Additionally, by our earlier argument, if the $\delta_0>0$ used to define $T^1_\la$ is small enough we also have
\begin{equation}\label{T1}\|T^1_\la\|_{L^{p(q_c)}(M)\to L^{q_c}(M)}=O(1),\end{equation}
by H\"older's inequality as $q_c<(p(q_c))'$ and $(p(q_c))'<\tfrac{2(n-1)}{n-3}$.  

If we combine the last two estimates we conclude that
\begin{equation}\label{3.27}
\|T_\la\|_{L^{p(q_c)}(M)\to L^{q_c}(M)}=O(1).
\end{equation}

To use these bounds write
\begin{align}\label{3.28}
u&= \bigl(-\Delta_g-(\la+i\e(\la))^2\bigr)^{-1} \circ (-\Delta_g-(\la+i\e(\la))^2)u
\\
&=T_\la\bigl(-\Delta_g+V-(\la+i\e(\la))^2\bigr)u \, + \, T_\la\bigl( V_{\le N}\cdot u\bigr) \, + T_\la\bigl(V_{>N}\cdot u\bigr) \notag
\\
&\qquad\qquad+R_\la\bigl(-\Delta_g-(\la +i\e(\la))^2\bigr)u \
\notag
\\
&= I +II +III + IV, \notag
\end{align}
with $V_{\le N}$ and $V_{>N}$ as in \eqref{2.13}.

By \eqref{3.26}
\begin{equation}\label{3.29}
\|I\|_{q_c}\lesssim \bigl(\e(\la)\bigr)^{-1+\delta_n} \la^{\sigma(q_c)-1} \|(H_V-(\la+i\e(\la))^2)u\|_2,
\end{equation}
and by \eqref{3.25} we similarly obtain
\begin{multline}\label{3.30}
\|IV\|_{q_c}
\lesssim  \bigl(\e(\la)\bigr)^{-1+\delta_n} \,  \la^{\sigma(q_c)-1} \cdot (\la \,  \e(\la)) \|u\|_2
\\
\lesssim  \bigl(\e(\la)\bigr)^{-1+\delta_n} \,  \la^{\sigma(q_c)-1} \|(H_V-(\la+i\e(\la))^2)u\|_2,
\end{multline}
using the spectral theorem in the last inequality.

If we use \eqref{2.15} along with H\"older's inequality and \eqref{3.27} 
along with the arguments from \S~\ref{abstractsec}, we conclude that we can fix
$N$ large enough so that
\begin{equation}\label{3.31}
\|III\|_{q_c}\le \tfrac12 \, \|u\|_{q_c}.
\end{equation}
Also, \eqref{3.26}  and \eqref{2.14} yield for this fixed $N$
\begin{multline}\label{3.32}
\|II\|_{q_c} \le C_N  \bigl(\e(\la)\bigr)^{-1+\delta_n} \,  \la^{\sigma(q_c)-1} \|u\|_2
\\
\lesssim  \bigl(\e(\la)\bigr)^{-1+\delta_n} \,  \la^{\sigma(q_c)-1} \|(H_V-(\la+i\e(\la))^2)u\|_2,
\end{multline}
using the spectral theorem and the fact that $\e(\la)\cdot \la\ge1$ if $\la\ge 1$.

Combining \eqref{3.29}, \eqref{3.30}, \eqref{3.31} and \eqref{3.32} yields
$$\|u\|_{q_c}\lesssim \bigl(\e(\la)\bigr)^{-1+\delta_n} \,  \la^{\sigma(q_c)-1} \|(H_V-(\la+i\e(\la))^2)u\|_2,$$
and since this is equivalent to \eqref{3.22}, the proof of the quasimode estimates for $q=q_c$ in  Theorem~\ref{nonpossob} is complete.

\newsection{Improved bounds for tori}\label{torussec}

In this section we shall prove Theorem~\ref{torussob}.  Let us start by going over the proof of quasimode and
uniform Sobolev estimates for the unperturbed operator
$H_0=-\Delta_{\Tn}$, which involve the exponent $q=\tfrac{2n}{n-2}$:
\begin{align}\label{4.1}
\|u\|_{L^{\frac{2n}{n-2}}(\Tn)}&\lesssim
\la^{-1/2}(\e(\la))^{-1/2}\|(-\Delta_g-(\la+i\e(\la))^2)
u
\|_{L^2(\Tn)}
\\
\|u\|_{L^{\frac{2n}{n-2}}(\Tn)}&\lesssim 
\| (-\Delta_g-(\la+i\e(\la))^2)
u\|_{L^{\frac{2n}{n+2}}(\Tn)}, \label{4.2}
\end{align}
for $\la\ge1$ with
\begin{equation}\label{4.3}
\e(\la)=\la^{-1/3+\delta_0}, \quad \forall \, \delta_0>0.
\end{equation}
Recall that $\sigma\bigl(\tfrac{2n}{n-2}\bigr)=\tfrac12$, and so \eqref{4.1}
corresponds to \eqref{2.2} for $q=\tfrac{2n}{n-2}$ with the
optimal
$\delta(\la,q)=\sqrt{\e(\la)}$.

Even though these estimates are in \cite{BSSY} for 
$n=3$ and Hickman~\cite{Hickman} for other dimensions,
let us start by reviewing their proofs, since, 
as in the preceding section, we shall need to modify
them to handle the estimates for $H_V$, especially the ones involving exponents $q$ for which appropriate
uniform Sobolev estimates are unavailable, which includes
the case where $q=q_c$.

The main estimate that is used to prove these two inequalities is a discrete version of the Stein-Tomas restriction theorem:
\begin{multline}\label{4.4}
\|\chi_{[\la,\la+\rho]}f\|_{L^{q_c}(\Tn)}
\lesssim (\rho\la)^{1/q_c} \, \la^{\e_0} \, \|f\|_{L^2(\Tn)}, \, \, \forall \, \e_0>0, \\ \text{if } \,
\, \, \, q_c=\tfrac{2(n+1)}{n-1} \, \, \text{and }
\, \la^{-1}\le \rho\le 1.
\end{multline}
Here, $\chi_I$ denotes the spectral projection operator
associated with the interval $I$ for $H_0$.
Since $\sigma(q_c)=1/q_c$, this represents a substantial improvement
over the unit band ($\rho=1$) spectral projection estimates
of one of us \cite{sogge88}.  On the other hand, unlike
\eqref{4.1}, it does not involve $\delta(\rho)=\sqrt{\rho}$.  Indeed, no such estimate can be valid for $\rho$ close to
the associated wavelength $\la^{-1}$.

Hickman~\cite{Hickman} proved \eqref{4.4} using the
decoupling estimates of Bourgain and Demeter~\cite{BourgainDemeterDecouple}.  Specifically,
Hickman showed that \eqref{4.4} is a consequence of Theorem 2.2 in
\cite{BourgainDemeterDecouple}.  Before that,
Bourgain, Shao, Yao and one of us \cite{BSSY} obtained
a somewhat weaker form of \eqref{4.4} when $n=3$ in which it was required that $\la^{-1/3}\le \rho\le1$.  This paper preceded the decoupling estimates of Bourgain and Demeter, and instead relied on multilinear techniques
of Bourgain and Guth~\cite{BourgainGuth}.

We shall require an equivalent form of \eqref{4.4}:
\begin{multline}\label{4.4'}\tag{4.4$'$}
\bigr\|\, m_{\la,\rho}(\sqrt{H_0})f\, \bigr\|_{L^{q_c}(\Tn)}
\lesssim \|m_{\la,\rho}\|_\infty \cdot (\rho\la)^{2/q_c}\la^{\e_0}\, \|f\|_{L^{q_c'}(\Tn)} \, \, \forall \, \e_0>0, 
\\ \,  
\text{if } \, \text{supp }m_{\la,\rho}\subset [\la,\la
+\rho]
\, \, \text{and } \,
 \, \la^{-1}\le \rho\le 1.
\end{multline}
After observing that \eqref{4.4} and orthogonality implies that
$\|m_{\la,\rho}\|_{L^2(\Tn)\to L^{q_c}(\Tn)}
=O\bigl((\rho\la)^{1/q_c}\la^{\e_0}\bigr)$ for all
$\e_0>0$, one obtains \eqref{4.4'} from this and 
a standard $TT^*$ argument.

Let us now briefly recall the proof of \eqref{4.1}.
As we mentioned earlier, it is equivalent to the
statement that
\begin{equation}\label{4.1'}\tag{4.1$'$}
\|\chi_{[\la, \la+\e(\la)]}\|_{L^2(\Tn)\to L^{\frac{2n}{n-2}}(\Tn)} \lesssim \sqrt{\e(\la)} \, \la^{\frac12}, 
\quad \e(\la)=\la^{-1/3+\delta_0}, \, \, \,
\forall \, \delta_0>0.
\end{equation}
If $a_0\in {\mathcal S}(\R)$ satisfies
\begin{equation}\label{4.5}
a_0(0)=1 \quad \text{and } \, \, \text{supp }
\Hat a_0\subset (-1/2,1/2),
\end{equation}
then \eqref{4.1'} is equivalent to the statement 
that $a_0\bigl((\e(\la))^{-1}(\la-P)\bigr)$, $P=
\sqrt{H_0}$, maps $L^2(\Tn)$ to $L^{\frac{2n}{n-2}}(\Tn)$
with norm $O(\sqrt{\e(\la)}\la^{\frac12})$, and by a simple
$TT^*$ argument this in turn is equivalent to the 
statement that
\begin{equation}\label{4.1''}\tag{4.1$''$}
\bigl\| a\bigl((\e(\la))^{-1}(\la-P)\bigr)f\bigr\|_{L^{\frac{2n}{n-2}}(\Tn)}\lesssim \e(\la)\, \la \,
\|f\|_{L^{\frac{2n}{n+2}}(\Tn)}, \quad \text{with } \, a(\tau)
=\bigl(a_0(\tau)\bigr)^2.
\end{equation}
Next we note that
\begin{multline*}a\bigl((\e(\la))^{-1}(\la-P)\bigr)f
=\frac{\e(\la)}{2\pi}
\int \Hat a\bigl(\e(\la)t\bigr) \, e^{i\la t}
e^{-itP}f \, dt
\\
= \frac{\e(\la)}{\pi}\int \Hat a\bigl(\e(\la)t\bigr) \, e^{i\la t}
(\cos tP)f \, dt + a\bigl((\e(\la))^{-1}(\la+P)\bigr)f.
\end{multline*}
Since $(\e(\la))^{-1}, \, \la\ge1$ and $P$ is a positive operator, it is
a simple matter to use either Sobolev estimates
or spectral projection estimates from \cite{sogge88} to see that the operator in the last term in the right maps
$L^2(\Tn)$ to $L^{\frac{2n}{n-2}}(\Tn)$ with norm
$O(\la^{-N})$ for any $N$.  Thus we would have
\eqref{4.1''}, and consequently \eqref{4.1}, if we
could show that
\begin{multline}\label{4.6}
\| Tf\|_{L^{\frac{2n}{n-2}}(\Tn)}
\lesssim \la \, \|f\|_{L^{\frac{2n}{n+2}}(\Tn)}, \quad
\e(\la)=\la^{-\frac13+\delta_0},
\\
\text{where } Tf=\int \Hat a\bigl(\e(\la)t\bigr) \, e^{i\la t}
(\cos tP)f \, dt.
\end{multline}

Note that, by \eqref{4.5}, the integrand vanishes
when $|t|>2(\e(\la))^{-1}$.  To exploit this,
let us fix a Littlewood-Paley bump function
$\beta\in C^\infty_0((1/2,2))$ satisfying
\begin{equation}\label{4.7}
\sum_{j=-\infty}^\infty \beta(2^{-j}t)\equiv 1, \, \, 
t>0,
\end{equation}
and set
\begin{equation}\label{4.8}
\beta_0(t)=1-\sum_{j=1}^\infty \beta(2^{-j}|t|)
\in C^\infty_0(\R).
\end{equation}
Using these we can split the operator in \eqref{4.6}
as
\begin{equation}\label{4.9}
Tf=\sum_{j=0}^\infty T_jf,
\end{equation}
where
\begin{multline}\label{4.10}
T_0f= \int \beta_0(t) \, \Hat a\bigl(\e(\la)t\bigr) \, e^{i\la t}
(\cos tP)f \, dt \\
\text{and } \, 
T_jf= \int \beta(2^{-j}|t|) \,  \Hat a\bigl(\e(\la)t\bigr) \, e^{i\la t}
(\cos tP)f \, dt, \, \, j=1,2,\dots.
\end{multline}

Clearly then, \eqref{4.6} would be a consequence of the
following
\begin{equation}\label{4.11}
\|T_jf\|_{L^{\frac{2n}{n-2}}(\Tn)}\lesssim 
2^{-\delta j}\la \|f\|_{L^{\frac{2n}{n+2}}(\Tn)}, \, \,
\, j=0,1,2,\dots,
\end{equation}
for some $\delta>0$ which depends on $n$ and the
$\delta_0>0$ in \eqref{4.3}.

The bound for $j=0$ is a simple consequence of the
spectral projection estimates of one of us
\cite{sogge88}.  It is simple to check that the remaining bounds follow, by
interpolation from the following two estimates:
\begin{equation}\label{4.12}
\|T_jf\|_{L^{\frac{2(n+1)}{n-1}}(\Tn)}
\lesssim 
 \la^{\e_0}\la^{\frac{n-1}{n+1}} 2^{\frac{2}{n+1}j}
\|f\|_{L^{\frac{2(n+1)}{n+3}}(\Tn)}, \quad \forall 
\, \e_0>0,
\end{equation}
and
\begin{equation}\label{4.13}
\|T_jf\|_{L^{\infty}(\Tn)}
\lesssim \la^{\frac{n-1}2} 2^{\frac{n+1}2j}
\|f\|_{L^{1}(\Tn)}.
\end{equation}
Indeed, since $\tfrac{n-2}{2n}=\theta \tfrac{n-1}{2(n+1)}
+(1-\theta)\tfrac1\infty$, with $\theta = 
\tfrac{(n+1)(n-2)}{n(n-1)}$, by interpolation \eqref{4.12} and \eqref{4.13} yield for all $\e_0>0$
\begin{equation}\label{4.13''}
\|T_j\|_{L^{\frac{2n}{n+2}}(\Tn)\to L^{\frac{2n}{n-2}}
(\Tn)}\lesssim \la^{1+\e_0}
\la^{-\frac1n} 2^{\frac{3j}n},
\end{equation}
which implies \eqref{4.11}, since by \eqref{4.3} and \eqref{4.5},
$T_j=0$ for $2^j$ larger than a fixed constant times
$\la^{-\frac13+\delta_0}$.  So, given any fixed
$\delta_0$ as in \eqref{4.3}, we obtain \eqref{4.11} 
with $\delta= \delta_0/n$ if
the loss 
$\e_0>0$ here is small enough.

To finish our proof of \eqref{4.1} it remains to prove
\eqref{4.12} and \eqref{4.13}.

The first inequality follows from applying \eqref{4.4'}
with $\rho=2^{-j}$ since
$$\int \beta(2^{-j}|t|) \,  \Hat a\bigl(\e(\la)t\bigr) \, e^{i\la t}
\cos (t\tau) \, dt =O\bigl(2^j(1+2^j|\la-\tau|)^{-N}),
\quad
\forall \, N \, \,  \text{if }  \, \la\ge1 \, 
\, \text{ and } \, \,  \tau\ge 0.$$
Note that the integral in the left vanishes if $2^j$ is larger than a fixed
multiple of $(\e(\la))^{-1}$.

The remaining inequality, \eqref{4.13}, amounts to
showing that the kernel $K_j(x,y)$ of $T_j$ satisfies
\begin{equation}\label{4.13'}\tag{4.13$'$}
K_j(x,y)=O(\la^{\frac{n-1}2} 2^{\frac{n+1}2 j}).
\end{equation}
If we relate $\Tn$ to $(-\pi,\pi]^n$ and 
the wave kernel $\cos tP$ on $\Tn$ to the Euclidean one (see, e.g., \cite[\S 3.5]{SoggeHangzhou}),
we can write this kernel as follows
\begin{equation}\label{4.14}
K_j(x,y)=(2\pi)^{-n}
\sum_{\ell \in {\mathbb Z}^n}
\int_{-\infty}^\infty
\beta(2^{-j}|t|) \,  \Hat a\bigl(\e(\la)t\bigr) \, e^{i\la t}
(\cos t\sqrt{-\Delta_{\Rn}})(x,y+\ell) \, dt,
\end{equation}
with $(\cos t\sqrt{-\Delta_{\Rn}})(x,y+\ell)$ denoting
the wave kernel in $\Rn$.  If we call the $\ell$-th
summand above $K_{j,\ell}(x,y)$ then by using
stationary phase and arguing as
in \cite{BSSY} or \cite[\S 3.5]{SoggeHangzhou} shows
that 
\begin{equation}\label{4.15}
|K_{j,\ell}(x,y)|\lesssim \la^{\frac{n-1}2}
(1+|x-y-\ell|)^{-\frac{n-1}2}
\lesssim \la^{\frac{n-1}2} (1+|\ell|)^{-\frac{n-1}2},
\quad x,y\in (-\pi,\pi]^n.
\end{equation}
Furthermore, by Huygens' principle $K_{j,\ell}(x,y)=0$
when $x,y\in (-\pi,\pi]^n$ and $|\ell|$ is larger than
a fixed multiple of $2^j$.  Therefore, for such $x,y$ we have
\begin{equation}\label{4.15'}
|K_{j}(x,y)|\lesssim \la^{\frac{n-1}2} \sum_{\{\ell \in {\mathbb Z}^n: \, 
|\ell|\le 2^j \}} (1+|\ell|)^{-\frac{n-1}2}
\lesssim \la^{\frac{n-1}2} 2^{\frac{n+1}2j},
\end{equation}
as desired.

Let us now see how we can use this argument to prove
the uniform Sobolev estimates \eqref{4.2}. 
As was the case in \S~\ref{nonpossec}, we shall make use of 
the splitting of the resolvent operator
$(-\Delta_{\Tn}-(\la+i\e(\la))^2)^{-1}$ as in
\eqref{3.13}--\eqref{3.17}, where $\e(\la)$ now as in \eqref{4.3}.  In our setting, we may simplify things
a bit compared to the argument in \S~\ref{nonpossec}
by taking $T=(\e(\la))^{-1}$, with, as we said,
$\e(\la)$ is now as in \eqref{4.3}.
We then would obtain \eqref{4.2}
if we had \eqref{3.17} in the current setting.

The bounds there for $R_\la$ follow from a simple
orthogonality argument and \eqref{4.1'}.  Also,
just as before the bounds in \eqref{3.17}
for the local operator are known (see \cite{DKS}, 
\cite{BSSY}).

To prove the bounds for the remaining operator
$T^1_\la$ in \eqref{3.17}, we split up the integral
dyadically as before by
writing
$$T^{1}_\la = T^{1,0}_\la +\sum_{j=1}^\infty T^{1,j}_\la,
$$
where for $j=1,2,\dots$
\begin{equation}\label{4.16}
T^{1,j}_\la = \frac{i}{\la+i\e(\la)}\int_0^\infty 
\beta(2^{-j}t) \,
(1-\eta(t)) \, \eta(t/T) \, e^{i\la t}e^{-\e(\la)t}\, \cos tP \, dt,
\end{equation}
and $T^{1,0}_\la$ is given by an analogous
formula with $\beta(2^{-j}t)$ replaced by
$\beta_0(t)\in C^\infty_0(\Rn)$.  Since
$T^{1,0}_\la$ is a local operator which shares
the same properties as $T^0_\la$, we have the analog
of \eqref{3.17} with $S_\la=T^{1,0}_\la$.  As a result,
we would have the remaining inequality, \eqref{3.17}
with $S_\la=T^1_\la$ if we could show that
when \eqref{4.3} is valid we have, as before
for some $\delta>0$ depending on $\delta_0$ and $n$
\begin{equation}\label{4.17}
\|T^{1,j}_\la\|_{L^{\frac{2n}{n+2}}(\Tn)
\to L^{\frac{2n}{n-2}}(\Tn)}
\lesssim 2^{-\delta j}.
\end{equation}

Since $T^{1,j}_\la=0$ when $2^j$ is larger than a fixed
multiple of $(\e(\la))^{-1}$ by the proof of \eqref{4.1}
we would obtain this estimate via interpolation via
the following two estimates:
\begin{equation}\label{4.18}
\|T_\la^{1,j}f\|_{L^{\frac{2(n+1)}{n-1}}(\Tn)}
\lesssim 
\la^{-1}\cdot \la^{\e_0}\la^{\frac{n-1}{n+1}} 2^{\frac{2}{n+1}j}
\|f\|_{L^{\frac{2(n+1)}{n+3}}(\Tn)}, \quad \forall 
\, \e_0>0,
\end{equation}
and
\begin{equation}\label{4.19}
\|T_\la^{1,j}f\|_{L^{\infty}(\Tn)}
\lesssim \la^{-1}\cdot\la^{\frac{n-1}2} 2^{\frac{n+1}2j}
\|f\|_{L^{1}(\Tn)}.
\end{equation}
Due to the $(\la+i\e(\la))^{-1}$ factor in 
\eqref{4.16} one sees from this formula and \eqref{4.10}
that $T^{1,j}_\la$ behaves like $\la^{-1} T_j$ and so
it is clear that the proof of \eqref{4.12} and 
\eqref{4.13} yield \eqref{4.18} and \eqref{4.19}, respectively.  This finishes our proofs of 
\eqref{4.1} and \eqref{4.2}.

\medskip

Using \eqref{4.1} and \eqref{4.2} along with Corollary~\ref{corr2.2} we obtain the bounds in
Theorem~\ref{torussob} involving $q=\tfrac{2n}{n-2}$.


\subsection*{Quasimode and uniform Sobolev estimates for the critical exponent }

Suppose that 
$
q_c=\frac{2(n+1)}{n-1}, \frac{1}{p(q_c)}-\frac{1}{q_c}=\tfrac2n
$, and, as in Theorem~\ref{torussob}, let us assume that for an arbitrary  fixed $\delta_0>0$
\begin{equation}\label{4.20}
\e(\la)=\la^{-\frac15+\delta_0},
\,\,\text{if}\,\,n\ge 4,\,\,\text{and}\,\,\,  \e(\la)=\la^{-\frac{3}{16}+\delta_0},\,\,\text{if}\,\,n=3.
\end{equation}
 We then recall that the estimates in Theorem~\ref{torussob} for $q=q_c$ say that for $u\in\text{Dom}(H_V)$ we have
\begin{equation}\label{4.21}
\|u\|_{L^{q_c}(\Tn)}
\lesssim \la^{\e_0}
(\e(\la))^{-\frac{n+3}{2(n+1)}}
\la^{-\frac{n+3}{2(n+1)}} 
\|(H_V-(\la + i\e(\la))^2)u\|_{L^2(\Tn)},
\, \, \la \ge 1,
\end{equation}
as well as 
\begin{equation}\label{4.22}
\|u\|_{L^{q_c}(\Tn)}
\lesssim 
\|(H_V-(\la + i\e(\la))^2)u\|_{L^{p(q_c)}(\Tn)}, \, \, \la \ge 1.
\end{equation}

As noted before, the inequality \eqref{4.21} is equivalent
to the spectral projection estimates
\begin{multline}\label{4.22'}\tag{4.23$'$}
\|\chi^V_{[\la,\la+\rho]}f\|_{L^{q_c}(\Tn)}\lesssim
(\rho\la)^{1/q_c} \, \la^{\e_0} \, \|f\|_{L^2(\Tn)}, \\
 \forall \, \delta_0>0
\, \,\, \rho\in [\la^{-\frac{1}{5}+\delta_0}, 1]\,\,\,\text{if}\,\,\, n\ge 4\,\,\,\text{or}\,\,\, \rho\in [\la^{-\frac{3}{16}+\delta_0}, 1]\,\,\,\text{if}\,\,\, n=3.
\end{multline}
This is weaker than the $V\equiv0$ results of
Hickman~\cite{Hickman}, i.e., \eqref{4.4}.  Even though the
$\rho$-intervals  in \eqref{4.22'} do not shrink to
$\{1\}$ as  $n\to \infty$, it would be interesting to try
to improve the range of $\rho$ in this inequality.

Since it is straightforward to check that for $\e(\la)$ satisfying \eqref{4.20}, \eqref{2.5'} is valid, by Corollary~\ref{corr2.2}, we would have
\eqref{4.21} and \eqref{4.22} if we knew that
for such $\e(\la)$,
we had the quasimode estimates
\begin{equation}\label{4.23}
\|u\|_{q_c}\lesssim \la^{\e_0}
(\e(\la))^{-\frac{n+3}{2(n+1)}}
\la^{-\frac{n+3}{2(n+1)}} 
\|(-\Delta_g-(\la+i\e(\la))^2u\|_2,\,\,
\text{for }\,\, \la\ge 1,
\end{equation}
and 
\begin{equation}\label{4.23'}
\|u\|_{p(q_c)^\prime}\lesssim 
(\e(\la))^{-1}
\la^{\sigma(p(q_c)^\prime)-1} 
\|(-\Delta_g-(\la+i\e(\la))^2u\|_2,\,\,
\text{for }\,\, \la\ge 1,
\end{equation}
as well as
\begin{equation}\label{4.24}
\|u\|_{q_c}\lesssim \|(-\Delta_g-(\la+i\e(\la))^2)u\|_{p(q_c)},
\quad \la\ge1.
\end{equation}

Inequality \eqref{4.23} follows from Hickman's estimate \eqref{4.4} and a simple orthogonality argument. 
And \eqref{4.23'} follows from the same argument by using the general spectral projection estimates of the fourth author \cite{sogge88}.
As we shall see at the end of this section, we can get a better bound than the right side of  \eqref{4.23'} for $q=p(q_c)^\prime$, but 
here as long as \eqref{2.5'} is valid for $\e(\la)$ satisfying \eqref{4.20}, the powers of $\e(\la)$ in inequalities \eqref{4.23} and \eqref{4.23'}, which are numbers between $[-1, -\frac12]$, are not crucial in the proof of \eqref{4.22}.

Now let us see how we can modify the proof of 
\eqref{4.2} to obtain \eqref{4.24}.
We shall make use of 
the splitting of the resolvent operator
$(-\Delta_{\Tn}-(\la+i\e(\la))^2)^{-1}$ as in
\eqref{3.13}--\eqref{3.17}, where $\e(\la)$ now as in \eqref{4.20}.
We then would obtain \eqref{4.24}
if we had \eqref{3.17} in the current setting.

Unlike previous cases, we do not have sharp spectral projection bounds here for the exponent $q_c$. The operator $R_\la$ will be dealt with differently after we established the desired bounds for $T_\la$.

As we noted earlier the local operator $T^0_\la$ always
satisfies the desired bounds in the uniform Sobolev
estimates regardless of the choice of $\e(\la)$:
$$\|T^0_\la\|_{L^{p(q_c)}(\Tn)\to L^{q_c}(\Tn)}
=O(1) \quad \text{if } \, \,
\tfrac1{p(q_c)}-\tfrac1{q_c}=\tfrac2n,
\, \, \, \text{i.e.,} \quad  p(q_c)=\tfrac{2n(n+1)}
{n^2+3n+4}.$$

For the operator $T_\la^1$, just as in the proof of \eqref{4.2}, we
shall need to use the dyadic decomposition
$$T^1_\la =T^{1,0}_\la + \sum_{j=1}^\infty T^{1,j}_\la$$
exactly as before where for $j=1,2,3,\dots$ $T^{1,j}$
is given by \eqref{4.16} and for $j=0$ the analog of this identity with $\beta(2^{-j}t)$ replaced by
$\beta_0(t)\in C^\infty_0(\R)$.  Since the factor
$(1-\eta(t))$ in each of these integrals vanishes
near the origin, the quasimode estimates in \cite{sogge88} imply that
$\|T_\la^{1,0}\|_{L^{p(q_c)}(\Tn)\to L^{q_c}(\Tn)}=O(1)$,
or, alternately one can use the fact that $T_\la^{1,0}$ behaves like $T_\la^0$ and deduce this from arguments in
\cite{DKS}, \cite{BSSY} or \cite{sogge88}.  Based
on the desired bounds for $j=0$, we conclude that
if we could show that for some $\delta>0$
\begin{equation}\label{4.25}
\|T^{1,j}_\la \|_{L^{p(q_c)}(\Tn)\to L^{q_c}(\Tn)}
=O(2^{-j\delta}) \quad \text{if } \, \,
\tfrac1{p(q_c)}-\tfrac1{q_c}=\tfrac2n, \, \, \, j=1,2,3,\dots,
\end{equation}
then we would obtain $\|T_\la\|_{L^{p(q_c)}(\Tn)\to L^{q_c}(\Tn)}=O(1)$.  As before, $\delta$ here depends on the various parameters in \eqref{4.20}, and, in
order to get the bounds in \eqref{4.25} we are
lead to assume that $\e(\la)$ as
in \eqref{4.20}.

In order to prove \eqref{4.25} we claim that, by
interpolation, it suffices to prove the following
three inequalities
\begin{equation}\label{4.26}
\|T^{1,j}_\la\|_{L^{p}(\Tn)\to L^q(\Tn)}\lesssim 
\la^{\e_0}\la^{-\frac{1}{n}} 2^{\frac{3j}n}, \quad \forall \, \e_0>0,
\quad \text{if } \, \,  q= \tfrac{2n}
{n-2},
\end{equation}
\begin{multline}\label{4.26'}
\|T^{1,j}_\la\|_{L^{p}(\Tn)\to L^q(\Tn)}\lesssim \la^{\e_0}
2^j,\\
\quad \forall \, \e_0>0, \,\,\text{if } \, \,  q= \tfrac{2n}
{n-3},\,\,\text{for}\,\,n\ge 4, \,\,\text{or}\,\,\, q=\infty\,\,\text{for}\,\,n=3,
\end{multline}
and
\begin{equation}\label{4.27}
\|T^{1,j}_\la\|_{L^p(\Tn)\to L^{q}(\Tn)}
\lesssim \la^{\e_0} \la^{-\frac{1}{n}}
2^{\frac{n^2+2n-2}{n^2}j}, \quad \forall \, \e_0>0,\,\,
 \text{if } \, \, 
q= \tfrac{2n^2}
{(n-1)(n-2)}.
\end{equation}
where $\frac1p-\frac1q=\frac2n$.

To verify this claim we note that if $p=q_c^\prime$ and
$q=p(q_c)^\prime=\tfrac{2n(n+1)}{n^2-n-4}$, when $n\geq 4$, we have
$$\tfrac1{q} =\theta \cdot \tfrac {n-2}{2n}
+(1-\theta) \cdot \tfrac{(n-1)(n-2)}{2n^2}, \quad \text{if } \, \, \theta = \tfrac{n^2-3n-2}{(n+1)(n-2)}.$$
Consequently, by interpolation, \eqref{4.26} and
\eqref{4.27} yield for any $\e_0>0$
\begin{equation}\label{4.28}
\begin{aligned}
\|T^{1,j}_\la \|_{L^p(\Tn)\to L^q(\Tn)}
&\lesssim \la^{\e_0} 
\bigl(\la^{-\frac{1}{n}} 2^{\frac{3j}n}\bigr)^{\frac{n^2-3n-2}{(n+1)(n-2)}}\cdot
\bigl(\la^{-\frac{1}{n}}
2^{\frac{n^2+2n-2}{n^2}j}\bigr)^{\frac{2n}{(n+1)(n-2)}}
\\
&=\la^{\e_0} \la^{-\frac1{n}} \cdot
2^{\frac{5}{n}j}.
\end{aligned}
\end{equation}

When $n=3$, the above argument does not work since $\tfrac{n^2-3n-2}{(n+1)(n-2)}<0$ if $n=3$. Instead, we shall use interpolation between \eqref{4.26'} and \eqref{4.27}. More precisely, note that if $p=q_c^\prime$ and
$q=p(q_c)^\prime=12$, we have
$$\tfrac1q=\tfrac1{12} =\theta \cdot \tfrac 19
+(1-\theta) \cdot \tfrac1\infty, \quad \text{if } \, \, \theta = \tfrac34.$$

By interpolation, \eqref{4.26'} and
\eqref{4.27} yield for any $\e_0>0$
\begin{equation}\label{4.29}
\begin{aligned}
\|T^{1,j}_\la \|_{L^p(\Tn)\to L^q(\Tn)}
&\lesssim \la^{\e_0} 
\bigl(\la^{-\frac{1}{3}} 2^{\frac{13}{9}j}\bigr)^{\frac34}\cdot
\bigl(
2^{j}\bigr)^{\frac14}
\\
&=\la^{\e_0} \la^{-\frac1{4}} \cdot
2^{\frac{4}{3}j}.
\end{aligned}
\end{equation}

By duality, \eqref{4.28} and \eqref{4.29} leads to \eqref{4.25} if we fix $\delta_0>0$ in
\eqref{4.20} and choose $\e_0$ here to be sufficiently small since $T^{1,j}_\la=0$ if $2^j$ is larger than
a fixed constant times $(\e(\la))^{-1}$, which, satisfies \eqref{4.20}.

Now we shall give the proof of \eqref{4.26}-\eqref{4.27}.
The first inequality, \eqref{4.26}, follows from \eqref{4.13''}, since, as noted before, the operator $T_\la^{1,j}$ behaves like $\la^{-1}T_j$.

To prove the second inequality, first note that, if $n\ge 4$, by Theorem 2.7 in \cite{BourgainDemeterDecouple} and a simple orthogonality argument, we have
\begin{equation}\label{4.30}
\|\chi_{[\la,\la+\rho]}\|_{L^2(\Tn)\to L^q(\Tn)}
\lesssim \rho^{\frac12}\la^{\e_0}\la^{\sigma(q)}, \quad \forall \, \e_0>0,\,\,\rho\in [\la^{-1},1]
\,\,\,\text{if}\,\,\,q=\tfrac{2(n-1)}{n-3}
\end{equation}


As a consequence of \eqref{4.30}, we have 
\begin{equation}\label{4.30'}
\|T^{1,j}_\la \|_{L^2(\Tn)\to L^q(\Tn)}
\lesssim \la^{\e_0}\la^{\sigma(q)-1}2^{j/2},\quad \forall \, \e_0>0\,\,\,\text{if}\,\,\,q=\tfrac{2(n-1)}{n-3}
\end{equation}
by an orthogonality argument, since, as noted before
$T^{1,j}_\la =\la^{-1} m_{\la,j}(\sqrt{H_0})$ where
$m_{\la,j}(\tau) = O(2^j(1+2^j|\tau-\la|)^{-N})$ for any
$N$ if $\tau\ge 0$ and $\la\ge1$.

Inequality \eqref{4.26'} now just follows
from \eqref{4.30'} and \eqref{4.19} via interpolation.
Indeed, since
$$\tfrac{n-3}{2n} =\theta \cdot \tfrac{n-3}{2(n-1)}
+(1-\theta) \cdot \tfrac1\infty, \quad 
$$ 
$$\tfrac{n+1}{2n} =\theta \cdot \tfrac{1}{2}
+(1-\theta) \cdot 1,
$$
with $\theta=\tfrac{n-1}{n}$,
we deduce that for all $\e_0>0$  and $\theta$
as above we have
\begin{align*}
\|T^{1,j}_\la\|_{L^{\frac{2n}{n+1}}(\Tn)\to L^{\frac{2n}{n-3}}(\Tn)}
&\lesssim \la^{\e_0}
\bigl(\la^{-\frac{n-3}{2(n-1)}} 2^{j/2}\bigr)^{\theta}\cdot 
\bigl(
\la^{\frac{n-3}2} 2^{\frac{n+1}2j}
\bigr)^{1-\theta}
\\
&=\la^{\e_0} \,2^{j},
\end{align*}
as desired.

The $n=3$ case in \eqref{4.26'} follows from exactly the same argument by using the fact that
\begin{equation}\label{4.31}
\|\chi_{[\la,\la+\rho]}\|_{L^2(\mathbb{T}^3)\to L^\infty(\mathbb{T}^3)}
\lesssim \rho^{\frac12}\la^{\e_0+1}, \quad \forall \, \e_0>0,\,\,\rho\in [\la^{-1},1].
\end{equation}
If we take $\rho=\la^{-1}$ in the above inequality, \eqref{4.31} is equivalent to counting the lattice points on a sphere, which has a general upper bound in any dimensions, i.e., 
\begin{equation}\label{4.31'}
\|\chi_{[\la,\la+\la^{-1}]}\|_{L^2(\mathbb{T}^n)\to L^\infty(\mathbb{T}^n)}
\lesssim \la^{\frac{n-2}{2}+\e_0}, \quad \forall \, \e_0>0,\,\,n\ge 2.
\end{equation}
See e.g., \cite{bourgain2012restriction} for a more detailed discussion about inequality \eqref{4.31'}. \eqref{4.31} now follows from \eqref{4.31'} by a simple orthogonality argument.

The third inequality, \eqref{4.27}, involves the pair of exponents $(p,q)$ which is the intersection of Stein-Tomas restriction line where $q=\frac{n+1}{n-1}p^\prime$ and the uniform Sobolev line where $\frac1p-\frac1q=\frac2n$. More precisely, note that by \eqref{4.4}, after using the same argument as in the proof of \eqref{4.30'}, we have
\begin{equation}\label{4.32}
\|T^{1,j}_\la \|_{L^2(\Tn)\to L^{q_c}(\Tn)}
\lesssim \la^{\e_0}\la^{-1}2^j(\la2^{-j})^{1/q_c},\quad \forall \, \e_0>0\,\,\,\text{if}\,\,\,q_c=\tfrac{2(n+1)}{n-1}.
\end{equation}
Now \eqref{4.27} follows
from \eqref{4.32} and \eqref{4.19} via interpolation.
Indeed, since
$$\tfrac{(n-1)(n-2)}{2n^2} =\theta \cdot \tfrac1{q_c}
+(1-\theta) \cdot \tfrac1\infty,
$$
$$\tfrac{n^2+n+2}{2n^2} =\theta \cdot \tfrac12
+(1-\theta) \cdot 1,
$$
with $\theta=\tfrac{(n+1)(n-2)}{n^2}$, and $\tfrac{n^2+n+2}{2n^2}=\tfrac{(n-1)(n-2)}{2n^2}+\tfrac2n$, 
we deduce that for all $\e_0>0$  and $\theta$
as above we have
\begin{align*}
\|T^{1,j}_\la\|_{L^{p(q_c)}(\Tn)\to L^{(p(q_c))'}(\Tn)}
&\lesssim \la^{\e_0}
\bigl(\la^{-\frac{n+3}{2(n+1)}} 2^{\frac{n+3}{2(n+1)}j}\bigr)^{\theta}\cdot 
\bigl(
\la^{\frac{n-3}2} 2^{\frac{n+1}2j}
\bigr)^{1-\theta}
\\
&=\la^{\e_0} \, \la^{-\frac1{n}}2^{\frac{n^2+2n-2}{n^2}j},
\end{align*}
as desired.

For the remaining operator $R_\la$, we claim that it has the same mapping properties as the operator $T_\la^{1,j}$ where $2^j\approx \e(\la)^{-1}$. Recall that in proving \eqref{4.26}, \eqref{4.26'} and \eqref{4.27}, the only properties we required for the operator $T_\la^{1,j}$ are
\begin{equation}\label{4.33}
|T_\la^{1,j}(\tau)|=O\big(2^j(1+2^j|\tau-\la|)^{-N}\big),
\end{equation}
and
\begin{equation}\label{4.34}
|T_\la^{1,j}(x,y)|=O(\la^{\frac{n-3}{2}}2^{\frac{n+1}{2}j})
\end{equation}
Similarly for the operator $R_\la$, if $2^j\approx \e(\la)^{-1}$, by \eqref{blah} we have 
\begin{equation}\label{4.35}
|R_\la(\tau)|=O\big(2^j(1+2^j|\tau-\la|)^{-N}\big).
\end{equation}
And for the other kernel bounds \eqref{4.34}, if we use the dyadic decomposition
$R_\la=\sum_{k=0}^{\infty} R_\la^k$,
where 
\begin{equation}\label{4.36}
R_\la^k = \frac{i}{\la+i\e(\la)}\int_0^\infty \beta\big({2^{-k+1}\e(\la)t\big)}\big(1- \eta(\e(\la))t\big) \, e^{i\la t}e^{-\e(\la)t}\, \cos tP \, dt,
\end{equation}
and argue as in \eqref{4.14}-\eqref{4.15} using stationary phase, we have 
$$|R_\la^k(x,y)|\lesssim \la^{\frac{n-3}{2}}2^{\frac{n+1}{2}k}(\e(\la))^{\frac{n+1}{2}} e^{-2^k}.
$$
After summing over $k$, we conclude that
\begin{equation}\label{4.37}
|R_\la(x,y)|=O(\la^{\frac{n-3}{2}}(\e(\la))^{\frac{n+1}{2}}).
\end{equation}
As a consequence of \eqref{4.35} and \eqref{4.37}, by using the same argument as for the operator $T_\la^{1,j}$, we obtain from this that 
\begin{equation}\label{4.38}
\|R_\la \|_{L^{p(q_c)}(\Tn)\to L^{q_c}(\Tn)}
=O(1),
\end{equation}
which completes the proof of \eqref{4.24}.

\subsection*{Quasimode and uniform Sobolev estimates for general exponents }
Now we will see how we can modify the above argument to show that \eqref{1.15} and \eqref{1.15'} hold for general exponents $q$. 
We shall first give the proof of \eqref{1.15}, since essentially it does not require sharp spectral projection bounds. To see this, by Corollary~\ref{corr2.2},
we would have \eqref{1.15} if we knew that for $(p,q)$ exponents satisfying \eqref{1.17}, 
we had the quasimode estimates
\begin{equation}\label{4.39}
\|u\|_r\lesssim \la^{\sigma(r)-1}\, \bigl(\e(\la)\bigr)^{-1}
\|(-\Delta_g-(\la+i\e(\la))^2u\|_2, \, \,
\text{for } \, r=q,p', \, \, \text{and } \, \la\ge 1,
\end{equation}
as well as
\begin{equation}\label{4.40}
\|u\|_q\lesssim \|(-\Delta_g-(\la+i\e(\la))^2)u\|_p,
\quad \la\ge1,
\end{equation}
where
\begin{equation}\label{4.41}
\e(\la)=\begin{cases}\la^{-\beta_1(n, p(q)^\prime)+\delta_0},\,\,\forall\, \delta_0>0 \,\,\,\text{if}\,\,\, \frac{2n}{n-1}< q<\frac{2n}{n-2} \\
\la^{-\beta_1(n,q)+\delta_0},\,\,\forall\, \delta_0>0  \,\,\,\text{if}\,\,\, \frac{2n}{n-2}\le q<\frac{2n}{n-3}
\end{cases}
\end{equation}

We shall give the explicit form of $\beta_1(n,q)$ later in \eqref{4.48}. Roughly speaking, it is a number that decreases from $\tfrac13$ to $0$ as $q$ increases from $\frac{2n}{n-2}$ to $\frac{2n}{n-3}$.

Here \eqref{4.39} follows easily from the spectral projection bounds of \cite{sogge88} and a simple orthogonality argument. We shall obtain an improvement over \eqref{4.35} at the end of this section by modifying the previous argument that was used to prove \eqref{4.1}.   Right now the bounds in \eqref{4.39} is sufficient since for
 $\e(\la)$ satisfying \eqref{4.41},  \eqref{2.5'} is valid for all for $(p,q)$ exponents satisfying \eqref{1.17} by \eqref{4.39}.

To prove \eqref{4.40}, by duality, it suffices to handle the case where $q\in [\frac{2n}{n-2},\frac{2n}{n-3})$. As before we shall split the resolvent operator as 
$$(-\Delta_g-(\la+i\e(\la))^2)^{-1}=T_\la^0+T_\la^1+R_\la.
$$
As noted earlier the local operator $T^0_\la$ always satisfies the desired bounds 
regardless of the choice of $\e(\la)$. That is 
\begin{equation}\label{4.42}
\|T^0_\la\|_{L^{p}(\Tn)\to L^{q}(\Tn)}
=O(1), \quad \text{if } \, \,
\tfrac1{p}-\tfrac1{q}=\tfrac2n,
\, \, \, \text{and} \,\,\, \tfrac{2n}{n-1}<q<\tfrac{2n}{n-3},
\end{equation}
see e.g. \cite{ShYa} and \cite{SHSo} for a proof of the above inequality. 

For the operator $T_\la$, we
shall need to use the dyadic decomposition
$$T^1_\la =T^{1,0}_\la + \sum_{j=1}^\infty T^{1,j}_\la$$
exactly as before where for $j=1,2,3,\dots$ $T^{1,j}$
is given by \eqref{4.16} and for $j=0$ the analog of this identity with $\beta(2^{-j}t)$ replaced by
$\beta_0(t)\in C^\infty_0(\R)$. 
The operator $T_\la^{1,0}$ behaves like $T_\la^0$ and it is not hard to see that it satisfies \eqref{4.42}. Based
on the desired bounds for $j=0$, we conclude that
if we could show that for some $\delta>0$
\begin{equation}\label{4.43}
\|T^{1,j}_\la \|_{L^{p(q)}(\Tn)\to L^{q}(\Tn)}
=O(2^{-j\delta}) \quad \text{if } \, \,
\tfrac1{p}-\tfrac1{q}=\tfrac2n, \, \, \, j=1,2,3,\dots,
\end{equation}
then we would obtain 
$$\|T_\la\|_{L^{p}(\Tn)\to L^{q}(\Tn)}=O(1),$$
as well as 
$$\|R_\la\|_{L^{p}(\Tn)\to L^{q}(\Tn)}=O(1),$$
since, as mentioned before, the operator $R_\la$ behaves like $T_\la^{1,j}$.

Given \eqref{4.26}, \eqref{4.26'} and \eqref{4.27}, the above inequality now follows easily from an interpolation argument. First, for $\frac{2n}{n-2}\le q\le\frac{2n^2}{(n-1)(n-2)}$, write 
\begin{equation}\label{4.44}
\tfrac1{q} =\theta_1 \cdot \tfrac {n-2}{2n}
+(1-\theta_1) \cdot \tfrac{(n-1)(n-2)}{2n^2}, \quad \text{if } \, \, \theta_1 = \tfrac{2n^2}{n-2}(\tfrac1q-\tfrac{(n-1)(n-2)}{2n^2}).
\end{equation}
Consequently, by interpolation, \eqref{4.26} and
\eqref{4.27} yield for any $\e_0>0$
\begin{equation}\label{4.45}
\begin{aligned}
\|T^{1,j}_\la \|_{L^p(\Tn)\to L^q(\Tn)}
&\lesssim \la^{\e_0} 
\bigl(\la^{-\frac{1}{n}} 2^{\frac{3j}n}\bigr)^{\theta_1}\cdot
\bigl(\la^{-\frac{1}{n}}
2^{\frac{n^2+2n-2}{n^2}j}\bigr)^{1-\theta_1}
\\
&=\la^{\e_0} \la^{-\frac1{n}} \cdot
2^{\frac{n^2+2n-2}{n^2}j}2^{-\frac{n^2-n-2}{n^2}\theta_1 j}.
\end{aligned}
\end{equation}
Similarly, for $\frac{2n^2}{(n-1)(n-2)} \le q< \frac{2n}{n-3}$, write 
\begin{equation}\label{4.46}
\tfrac1{q} =\theta_2\cdot \tfrac{(n-1)(n-2)}{2n^2} +(1-\theta_2) \cdot \tfrac {n-3}{2n}
, \quad \text{if } \, \, \theta_2 = n^2(\tfrac1q-\tfrac{n-3}{2n}).
\end{equation}
By interpolation, \eqref{4.26'} and
\eqref{4.27} yield for any $\e_0>0$
\begin{equation}\label{4.47}
\begin{aligned}
\|T^{1,j}_\la \|_{L^p(\Tn)\to L^q(\Tn)}
&\lesssim \la^{\e_0} 
\bigl(\la^{-\frac{1}{n}}2^{\frac{n^2+2n-2}{n^2}j}\bigr)^{\theta_2}\cdot
\bigl(2^{j}
\bigr)^{1-\theta_2}
\\
&=\la^{\e_0} \la^{-\frac{\theta_2}{n}} \cdot
2^{j}\cdot 2^{\frac{2n-2}{n^2}\theta_2 j}.
\end{aligned}
\end{equation}

As a result, given $\theta_1$ and $\theta_2$ as in \eqref{4.44} and \eqref{4.46}, if we define
\begin{equation}\label{4.48}
\beta_1(n,q)=\begin{cases}
\tfrac{n}{n^2+2n-2-(n^2-n-2)\theta_1},\,\,\,\text{if}\,\,\, \tfrac{2n}{n-2}\le q\le \tfrac{2n^2}{(n-1)(n-2)}\\
\tfrac{n\theta_2}{n^2+(2n-2)\theta_2},\,\,\,\text{if}\,\,\, \frac{2n^2}{(n-1)(n-2)} \le q< \frac{2n}{n-3},
\end{cases}
\end{equation}
by \eqref{4.45} and \eqref{4.47}, we obtain \eqref{4.43} if we fix $\delta_0>0$ in
\eqref{4.41} and choose $\e_0$ above to be sufficiently small since $T^{1,j}_\la=0$ if $2^j$ is larger than
a fixed constant times $(\e(\la))^{-1}$. Thus, the proof of \eqref{4.40} is complete.

To conclude, we shall give the proof of \eqref{1.15'}. We shall focus on the case $\frac{2(n+1)}{n-1}<q\le \frac{2n}{n-2}$, since, 
 the estimates for $q>\frac{2n}{n-2}$ follows as a corollary of Theorem~\ref{thm2.3}.

To proceed, note that by Corollary~\ref{corr2.2} as well as \eqref{4.40}, we would have \eqref{1.15'} if we knew that for $(p,q)$ exponents satisfying 
\begin{equation}\label{4.49}
\tfrac{2(n+1)}{n-1}<q\le \tfrac{2n}{n-2}
\quad \text{and } \, \, \tfrac1p-\tfrac1q=\tfrac2n,
\end{equation}
we had the quasimode estimates
\begin{equation}\label{4.50}
\|u\|_r\lesssim \la^{\sigma(r)-1}\, \bigl(\e(\la)\bigr)^{-1/2}
\|(-\Delta_g-(\la+i\e(\la))^2u\|_2, \, \,
\text{for } \, r=q,p', \, \, \text{and } \, \la\ge 1,
\end{equation}
where we shall take
\begin{equation}\label{4.52}
\e(\la)=\la^{-\beta_2(n, q)+\delta_0}, \, \, \,
\forall \, \delta_0>0,
\end{equation}
with
\begin{equation}\label{4.52'}
\beta_2(n, q)=\tfrac{(n-1)^2q-2(n-1)(n+1)}{(n+1)(n-1)q-2(n+1)^2+8}.
\end{equation}
Actually, given \eqref{4.40} and \eqref{4.50}, in order to apply Corollary~\ref{corr2.2}, it suffices to check \eqref{2.5'} is valid, which is equivalent to $\e(\la)\ge \la^{-\frac12}$ when $\frac{2(n+1)}{n-1}<q\le \frac{2n}{n-2}$. However, for such exponents $q$ we have
$$\min(\beta_1(n, q), \beta_2(n, q))\le 1/3,
$$
which implies \eqref{2.5'}. Also as before the inequality for $r=p^\prime$ \eqref{4.53} is not crucial for our proof.  Indeed,   simple quasimode estimates as in \eqref{4.39} are  sufficient for our use.

Note that compared with \eqref{4.39}, the power on $\e(\la)$ in \eqref{4.50} is sharp, which, as before, is equivalent to the spectral projection estimates
\begin{equation}\label{4.53}
\|\chi_{[\la,\la+\rho]}f\|_{L^{q}(\Tn)}\lesssim
\rho^{1/2} \la^{\sigma(q)} \, \|f\|_{L^2(\Tn)},
\, \, \forall \rho\ge \e(\la).
\end{equation}

To prove \eqref{4.53}, if we repeat the argument in \eqref{4.5}-\eqref{4.10}, by using a $TT^*$ argument, it suffices to prove that for $q>\tfrac{2(n+1)}{n-1}$, and $\e(\la)$ satisfying \eqref{4.52}
\begin{equation}\label{4.54}
\| Tf\|_{L^q(\Tn)}
\lesssim \la^{2\sigma(q)} \, \|f\|_{L^{q^\prime}(\Tn)},
\end{equation}
where
\begin{equation}\label{4.55}
Tf=\int \Hat a\bigl(\e(\la)t\bigr) \, e^{i\la t}
(\cos tP)f \, dt,
\end{equation}
with $a\in {\mathcal S}(\R)$ defined as in \eqref{4.1''}.

As before we shall split the operator in \eqref{4.55}
as
\begin{equation}
\nonumber
Tf=\sum_{j=0}^\infty T_jf,
\end{equation}
where the operator $T_j$ is defined as in \eqref{4.10}.

Clearly then, \eqref{4.55} would be a consequence of the
following
\begin{equation}\label{4.56}
\|T_jf\|_{L^q(\Tn)}\lesssim 
2^{-\delta j}\la^{2\sigma(q)} \|f\|_{L^{q{^\prime}}(\Tn)}, \, \,
\, j=0,1,2,\dots,
\end{equation}
for some $\delta>0$ which depends on $n$ and the
$\delta_0>0$ in \eqref{4.3}.

The bound for $j=0$ is a simple consequence of the
spectral projection estimates of one of us
\cite{sogge88}, while the remaining bounds follow by interpolation from \eqref{4.12} and \eqref{4.13}. Indeed, since
for any $q>\tfrac{2(n+1)}{n-1}$,
$\tfrac1q=\theta\cdot\tfrac{n-1}{2(n+1)}+(1-\theta)\cdot\tfrac1\infty
$, with $\theta=\frac{2(n+1)}{(n-1)q}$,  \eqref{4.12} and \eqref{4.13} yield for all $\e_0>0$
\begin{equation}\label{4.57}
\|T_j\|_{L^{q}(\Tn)\to L^{q^\prime}
(\Tn)}\lesssim \la^{2\sigma(q)+\e_0-\frac{n-1}2\cdot\frac{(n-1)q-2(n+1)}{(n-1)q}} 
\,
2^{j\frac{n+1}2\cdot\frac{(n-1)q-2(n+1)}{(n-1)q}+j\frac{2}{n+1}\cdot\frac{2(n+1)}{(n-1)q}}.
\end{equation}
As a result, given any fixed
$\delta_0$ as in \eqref{4.52}, we obtain \eqref{4.56} if
the loss 
$\e_0>0$ here is small enough, since by \eqref{4.3} and \eqref{4.5},
$T_j=0$ for $2^j$ larger than a fixed constant times
$\e(\la)^{-1}$ defined as in \eqref{4.52}. 

For later use, note that the above argument works for any $n\ge2$. When $n=2$, it gives the following analog of \eqref{4.53}
\begin{equation}\label{4.58}
\|\chi_{[\la,\la+\rho]}f\|_{L^{q}(\Tn)}\lesssim
\rho^{1/2} \la^{\sigma(r)} \, \|f\|_{L^2(\Tn)},
\, \, \forall\, \rho\ge \la^{-\frac{q-6}{3q-10}+\delta_0}, \,\,
 \, \, \forall \, \delta_0>0,\,\,\,\text{if}\,\,\, q>6.
\end{equation}
In particular, at the point $q=\infty$,  we have
\begin{equation}\label{4.59}
\|\chi_{[\la,\la+\rho]}f\|_{L^{\infty}(\Tn)}\lesssim
\rho^{1/2} \la^{\sigma(r)} \, \|f\|_{L^2(\Tn)},
\, \, \forall\, \rho\ge \la^{-\frac{1}{3}},
\end{equation}
by using \eqref{4.13} directly without interpolation with \eqref{4.12}.

\medskip

\noindent {\bf Remark:}
We shall briefly mention that improvements over the inequality \eqref{4.53} can be made in several ways. First, if we take $\rho=\la^{-1}$ in \eqref{4.53}, it 
is conjectured by Bourgain in \cite{bourgain1993eigenfunction} that for $n\ge 3$
\begin{equation}\label{4.61}
\|\chi_{[\la,\la+\la^{-1}]}f\|_{L^{q}(\Tn)}\lesssim
 \la^{\frac{n-2}{2}-\frac nq+\delta_0} \, \|f\|_{L^2(\Tn)},
\, \, \forall \, \delta_0>0, \,\,\la\ge 1 \,\,\,\text{and} \,\,\,q \ge \tfrac{2n}{n-2}.
\end{equation}

As in \eqref{4.32}, by Theorem 2.7 in \cite{BourgainDemeterDecouple}, \eqref{4.61} holds for all $q\ge \tfrac{2(n-1)}{n-3}$, which is currently the best partial results for this problem. It is interesting and not known to the authors that whether one can use \eqref{4.61} for $q\ge \tfrac{2(n-1)}{n-3}$ to improve the range of $\rho$ in the inequality \eqref{4.53} when $\frac{2(n+1)}{n-1}<\rho\le \frac{2n}{n-2}$.  

On the other hand, as in \cite{BSSY} and \cite{Hickman}, we can slightly improve the kernel bound \eqref{4.15'}, and thus obtain an improvement on the range of $\e(\la)$ in inequalities such as \eqref{4.40}, \eqref{4.50} and \eqref{4.53}, by exploiting the cancellation between different terms in \eqref{4.14} using exponential sum estimates. We omit the details here for simplicity.

\newsection{Improved quasimode estimates when $n=2$}\label{2d}

The purpose of this section is to derive improved quasimode estimates under certain geometric assumptions for $n=2$. Throughout this 
section we shall assume that $V\in \mathcal{K}(M)$ satisfying \eqref{kato}, since in two dimensions  $V\in L^1(M)$ can not ensure that the associated Schr\"odinger operator is self-adjoint. For a proof of self-adjointness of Schr\"odinger operators with Kato potentials, see e.g., \cite{BSS}.

Unlike what was the case for higher dimensions in Theorem~\ref{unifSob}, we cannot improve the universal
quasimode bounds in \cite{BSS} when $n=2$.  We can, however, improve the bounds in Theorem~\ref{nonpossob} in two-dimensions
by removing the smallness assumption on $V$ that was made in \cite{BSS}, and we can also obtain new bounds for two-dimensional 
tori.

First, let us see that we have the following analog of Theorem \ref{nonpossob}.
\begin{theorem}\label{thm5.1}
Assume that $(M,g)$ is a Riemannian surface of nonpositive curvature and that $V\in \mathcal{K}(M)$.
Then for $q\geq 6$, if
\begin{equation}\label{5.1}
\delta(q)=\begin{cases}
\frac{1}{72}, \,\,\,\text{if}  \,\,q=6 \\
\frac{1}{2},\,\,\,\,\,\text{if}\,\, q>6,
\end{cases}
\end{equation}
we have for $u\in \Dom(H_V)$ and $\la\ge1$,
\begin{equation}\label{5.2}
\|u\|_{q}\lesssim \la^{\sigma(q)-1}\, \bigl(\e(\la)\bigr)^{-1+\delta(q)} \, 
\bigl\| (H_V-(\la+i\e(\la))^2)u\bigr\|_2,
\end{equation}
where $\e(\la)=(\log(2+\la))^{-1}$. Consequently
\begin{equation}\label{5.3}
\bigl\|\chi^V_{[\la,\la+\e(\la)]} f \bigr\|_{q} \lesssim \la^{\sigma(q)} \, (\log(2+\la))^{-\delta(q)} \, \|f\|_2.
\end{equation}
\end{theorem}

To prove \eqref{5.2}, as before we shall use the fact that by \cite{SBLog} and \cite{HassellTacy}, we have \eqref{5.2} when $V\equiv0$, which
is equivalent to the following
\begin{equation}\label{5.4}
\bigl\| \, (-\Delta_g-(\la+i\e(\la))^2)^{-1}\, \bigr\|_{L^2(M)\to L^{q}(M)} \lesssim
\la^{\sigma(q)-1}\, \bigl(\e(\la)\bigr)^{-1+\delta(q)},
\end{equation}
as well as the following bounds for the spectral projection operators associated to $H_0=-\Delta_g$:
\begin{equation}\label{5.5}
\bigl\| \chi_{[\la,\la+\e(\la)]}\bigr\|_{L^2(M)\to L^{q}(M)} \lesssim \la^{\sigma(q)} \, \bigl(\e(\la)\bigr)^{\delta(q)}.
\end{equation}

The proof of \eqref{5.2} is based on the same idea as in the critical exponent case for higher dimensions.
 And unlike in higher dimensions, where we are able to prove uniform sobolev estimates for certain range of exponents, 
 the fact that $\delta(q)=1/2$ for $q>6$ is not crucial in our proof.

\begin{proof}[Proof of Theorem \ref{thm5.1}]

As in \cite{BSS}, we shall first prove \eqref{5.2} for the exponent $q=\infty$, and then use it to obtain \eqref{5.2} for $6\leq q<\infty$.

To proceed, just as before we shall write
\begin{equation}\label{5.6}
\bigl(-\Delta_g-(\la+i\e(\la))^2\bigr)^{-1}=T_\la+R_\la, \quad
\text{where } \, T_\la=T_0+T^1_\la,\end{equation}
with $T_\la^0$, $T^1_\la$ and $R_\la$ as in \eqref{3.14}, \eqref{3.15} and \eqref{3.16}, respectively.

Since $R_\la=m_\la(\sqrt{H_0})$ with $m_\la(\tau)$ as in \eqref{blah} one can use
\eqref{5.5} and a simple orthogonality argument to see that for all $q\geq 6$
\begin{equation}\label{5.7}
\|R_\la\|_{L^2(M)\to L^{q}(M)} \lesssim \la^{\sigma(q)-1}\, \bigl(\e(\la)\bigr)^{-1+\delta(q)},
\end{equation}
and also
\begin{equation}\label{5.8}
\bigl\| R_\la \circ (-\Delta_g-(\la+i\e(\la))^2)\bigr\|_{L^2(M)\to L^{q}(M)}
\lesssim  \la^{\sigma(q)-1}\, \bigl(\e(\la)\bigr)^{-1+\delta(q)} \cdot (\la \,  \e(\la)).
\end{equation}

If we set $T_\la=T^0_\la+T^1_\la$ as above, then since 
$T_\la=(-\Delta_g-(\la+i\e(\la))^2)^{-1}-R_\la$, we trivially obtain from \eqref{5.4} and \eqref{5.7} the bound
\begin{equation}\label{5.9}
\|T_\la\|_{L^2(M)\to L^{q}(M)}\lesssim \la^{\sigma(q)-1}\, \bigl(\e(\la)\bigr)^{-1+\delta(q)}.
\end{equation}

Note that by \eqref{3.19}, if the $\delta_0>0$ used to define $T^1_\la$ is small enough we have
\begin{equation}\label{5T1}\|T^1_\la\|_{L^{1}(M)\to L^{\infty}(M)}=O(\la^{-1/2}\la^{c_0\delta_0})\ll 1.
\end{equation}

Also for the local operator $T^0_\la$,  we have the following kernel estimates
\begin{equation}\label{k1}
|T^0_\la(x,y)|\leq
\begin{cases}
C_0|\log(\la d_g(x,y)/2)|,\hspace{2mm} \text{if} \hspace{2mm} d_g(x,y)\leq \la^{-1}\\
C_0\la^{-1/2}\big(d_g(x,y)\big)^{-1/2},\hspace{2mm}\text{if} \hspace{2mm} \la^{-1}\leq d_g(x,y)\leq 1\\
\end{cases}
\end{equation}
which comes from using stationary phase and the formulas
\begin{equation}\label{5.12}
S^0_\la = \frac{i}{\la+i\e(\la)}\int_0^\infty \eta(\la t) \, \eta(t/T) \, e^{i\la t}e^{-\e(\la)t}\, \cos tP \, dt
\end{equation}
and 
\begin{equation}\label{5.13}
S^1_\la = \frac{i}{\la+i\e(\la)}\int_0^\infty (\eta(t)-\eta(\la t)) \, \eta(t/T) \, e^{i\la t}e^{-\e(\la)t}\, \cos tP \, dt
\end{equation}
separately.

To see this, note that the multiplier associated to the operator $S^0_\la$ is
\begin{equation}
 S^0_\la(\tau) = \frac{i}{\la+i\e(\la)}\int_0^\infty \eta(\la t) \, \eta(t/T) \, e^{i\la t}e^{-\e(\la)t}\, \cos t\tau \, dt,
    \nonumber
\end{equation}
Using integration by parts, it is not hard to see that  for $j=0,1,2,...$,
\begin{equation}\label{5.14}
|\tfrac{d^j}{d\tau^j}S_0(\tau)|\le \begin{cases}
  C_j\lambda^{-2-j}\ \ \ \ \ \ {\rm if}\ \ |\tau|\le\lambda,\\
  C_j|\tau|^{-2-j}\ \ \ \ \ {\rm if}\ \ |\tau|>\lambda.
\end{cases}
\end{equation}
Given \eqref{5.14}, if  we argue as in the proof of \cite{SFIO2} Theorem 4.3.1, along with a change of variables, we have $|S_0(x,y)|\le C_0|\log(\la d_g(x,y)/2)| \textbf{1}_{d_g(x,y)<\la^{-1}}(x,y) $. The kernel for the operator $S^1_\la$ is a consequence of stationary phase argument after using Hadamard parametrix, see \cite{BSSY} and \cite{ShYa} for more details.

Since by heat kernel methods we have $\text{Dom}(H_V)\subset L^\infty(M)$ when $n=2$, by the very definition of the Kato space, $S^0_\la(Vu)$ is given by an absolutely convergent integral. Thus, if $\Lambda=\Lambda(M,V)\ge 1$ is sufficiently large, we have
since $V\in {\mathcal K}$
\begin{equation}\label{5.15}
\|S^0_\la(Vu)\|_{L^{\infty}(M)}\leq 1/4 \|u\|_{L^{\infty}(M)},\,\,\,\, \text{if} \,\,\,\, \la\geq \Lambda.
\end{equation}

To use these bounds write
\begin{align}\label{5.16}
u&= \bigl(-\Delta_g-(\la+i\e(\la))^2\bigr)^{-1} \circ (-\Delta_g-(\la+i\e(\la))^2)u
\\
&=T_\la\bigl(-\Delta_g+V-(\la+i\e(\la))^2\bigr)u \, + \, T_\la\bigl( V_{\le N}\cdot u\bigr) \, + T_\la\bigl(V_{>N}\cdot u\bigr) \notag
\\
&\qquad\qquad+R_\la\bigl(-\Delta_g-(\la +i\e(\la))^2\bigr)u \
\notag
\\
&= I +II +III + IV, \notag
\end{align}
with $V_{\le N}$ and $V_{>N}$ as in \eqref{2.13}.

By \eqref{5.9}
\begin{equation}\label{5.17}
\|I\|_{\infty}\lesssim \bigl(\e(\la)\bigr)^{-1/2} \la^{-1/2} \|(H_V-(\la+i\e(\la))^2)u\|_2,
\end{equation}
and by \eqref{5.8} we similarly obtain
\begin{multline}\label{5.18}
\|IV\|_{\infty}
\lesssim  \bigl(\e(\la)\bigr)^{-1/2} \,  \la^{-1/2} \cdot (\la \,  \e(\la)) \|u\|_2
\\
\lesssim  \bigl(\e(\la)\bigr)^{-1/2} \,  \la^{-1/2} \|(H_V-(\la+i\e(\la))^2)u\|_2,
\end{multline}
using the spectral theorem in the last inequality.

If we use \eqref{5T1}, \eqref{k1} and  \eqref{5.15}, along with H\"older's inequality, we conclude that we can fix
$N$ large enough so that
\begin{equation}\label{5.19}
\|III\|_{\infty}\le \tfrac12 \, \|u\|_{\infty}, \,\,\,\, \text{if}\,\,\, \la\geq \Lambda.
\end{equation}
Also, \eqref{5.9}  and \eqref{2.14} yield for this fixed $N$
\begin{multline}\label{5.20}
\|II\|_{\infty} \le C_N  \bigl(\e(\la)\bigr)^{-1/2} \,  \la^{-1/2} \|u\|_2
\\
\lesssim  \bigl(\e(\la)\bigr)^{-1/2} \,  \la^{-1/2} \|(H_V-(\la+i\e(\la))^2)u\|_2,
\end{multline}
using the spectral theorem and the fact that $\e(\la)\cdot \la\ge1$ if $\la\ge 1$.

Combining \eqref{5.17}, \eqref{5.18}, \eqref{5.19} and \eqref{5.20} yields
\begin{equation}\label{5.21}
\|u\|_{\infty}\lesssim \bigl(\e(\la)\bigr)^{-1/2} \,  \la^{-1/2} \|(H_V-(\la+i\e(\la))^2)u\|_2,\,\,\,\, \text{if}\,\,\, \la\geq \Lambda.
\end{equation}

To obtain the quasimode estimate \eqref{5.2} for $q=\infty$, we need to see that the bounds
in \eqref{5.21} are also valid when $1\le \la<\Lambda$, As before this just follows from the fact that
$$\bigl\|(H_V-\la^2+i\e(\la)\la)^{-1}f\bigr\|_{L^2(M)}
\le C
\bigl\|(H_V-\la^2+i\e(\Lambda)\Lambda)^{-1}f\bigr\|_{L^2(M)}, 
\quad \text{if } \, \, 1\le \la\le \Lambda,
$$
where $C$ is a constant that depends on $\Lambda$.

Now we shall prove  \eqref{5.2} for $6\leq q<\infty$.   We shall focus on the term $III$, since by \eqref{5.7}, \eqref{5.8}, and \eqref{5.9}, the other three terms are easily bounded by the right side of \eqref{5.2}.
Note that by \eqref{k1}, we have
$$
\sup_y\big(\int_M|T^0_\la(x,y)|^q dx\big)^{1/q}\leq C\la^{-2/q}, \,\,\,\,\text{if}\,\,\,  6\leq q<\infty.
$$
Whence by Minkowski's integral inequality,
\begin{equation}\label{5.22}\|T^0_\la\|_{L^{1}(M)\to L^{q}(M)}\leq C\la^{-2/q}.
\end{equation}
If we combine \eqref{5T1} and \eqref{5.22}, by H\"older's inequality,
$$
\|T_\la(V_{>N}u)\|_q\leq C\la^{-2/q}\|V_{>N}u\|_1\leq C\la^{-2/q}\|V\|_1\|u\|_\infty.
$$
Since we have just proved that
\begin{equation}
\|u\|_{\infty}\lesssim \bigl(\e(\la)\bigr)^{-1/2} \,  \la^{-1/2} \|(H_V-(\la+i\e(\la))^2)u\|_2,
\nonumber
\end{equation}
we conclude that the term $III$ is dominated by the right side of \eqref{5.2}.
\end{proof}

We can also obtain the following improved quasimode estimates for the 2-dimensional torus:
\begin{theorem}\label{thm5.2}
Let $\mathbb{T}^2$ denotes the 2 dimensional torus with flat metric, assume that $V\in \mathcal{K}(\mathbb{T}^2)$.
Then for $q> 6$, if
\begin{equation}\label{5.23}
\e(\la)=\e(\la,q)=\begin{cases} \la^{-\frac{q-6}{3q-10}+\delta_0} \quad \forall \,\delta_0>0 \,\,\, \text{if}\,\,\, 6<q<\infty \\
 \la^{-\frac13} \,\,\, \text{if}\,\,\, q=\infty
\end{cases}
\end{equation}
we have for $u\in \Dom(H_V)$ and $\la\ge1$,
\begin{equation}\label{5.24}
\|u\|_{L^q({\mathbb T}^2)} \lesssim \la^{\sigma(q)-1}\, \bigl(\e(\la)\bigr)^{-1/2} \, 
\bigl\| (H_V-(\la+i\e(\la))^2)u\bigr\|_{L^2({\mathbb T}^2)}.
\end{equation} 
Similarly, if $\e(\la)\geq \la^{-1/5}$, we have
\begin{equation}\label{5.25}
\|u\|_{{L^6({\mathbb T}^2)}}\lesssim \la^{\e_0}\, \bigl(\la\cdot\e(\la)\bigr)^{-5/6} \, 
\bigl\| (H_V-(\la+i\e(\la))^2)u\bigr\|_{L^2({\mathbb T}^2)}.
\end{equation} 
\end{theorem}

To prove \eqref{5.24}, we shall of course use the fact that by the spectral projection bounds in \eqref{4.58} and \eqref{4.59},  if $\e(\la)$ satisfies \eqref{5.23}, we have \eqref{5.24} when $V\equiv0$ which
is equivalent to the following
\begin{equation}\label{5.26}
\bigl\| \, (-\Delta_g-(\la+i\e(\la))^2)^{-1}\, \bigr\|_{L^2(M)\to L^{q}(M)} \lesssim
\la^{\sigma(q)-1}\, \bigl(\e(\la)\bigr)^{-1/2}.
\end{equation}
Also for the critical point $q=6$ we shall use 
\begin{multline}\label{5.27}
\bigl\| \, (-\Delta_g-(\la+i\e(\la))^2)^{-1}\, \bigr\|_{L^2(M)\to L^{6}(M)} \lesssim
\la^{\e_0}\bigl(\la\e(\la)\bigr)^{-5/6}, \\
\forall\, \e_0>0, \,\,\,\text{if}\,\,\, \la^{-1}\le \e(\la)\le 1,
\end{multline}
which is a consequence of the spectral projection estimates in \eqref{4.4}.

Now let us see how we can modify the proof of \eqref{5.2} to obtain \eqref{5.24} and \eqref{5.25}. As before, we shall first prove \eqref{5.2} for the exponent $q=\infty$, and then use it to obtain similar inequalities for $6\leq q<\infty$.

To proceed, write
\begin{equation}
\nonumber
\bigl(-\Delta_g-(\la+i\e(\la))^2\bigr)^{-1}=T_\la+R_\la, \quad
\text{where } \, T_\la=T_0+T^1_\la,\end{equation}
with $T_\la^0$, $T^1_\la$ and $R_\la$ as in \eqref{3.14}, \eqref{3.15} and \eqref{3.16}, respectively.

Since $R_\la=m_\la(\sqrt{H_0})$ with $m_\la(\tau)$ as in \eqref{blah} one can use
\eqref{4.58}, \eqref{4.59} and a simple orthogonality argument to see that for all $q> 6$
\begin{equation}\label{5.28}
\|R_\la\|_{L^2(\mathbb{T}^2)\to L^{q}(\mathbb{T}^2)} \lesssim \la^{\sigma(q)-1}\, \bigl(\e(\la)\bigr)^{-1/2},
\end{equation}
and also
\begin{equation}\label{5.29}
\bigl\| R_\la \circ (-\Delta_g-(\la+i\e(\la))^2)\bigr\|_{L^2(\mathbb{T}^2)\to L^{q}(\mathbb{T}^2)}
\lesssim  \la^{\sigma(q)-1}\, \bigl(\e(\la)\bigr)^{-1/2} \cdot (\la \,  \e(\la)).
\end{equation}

If we set $T_\la=T^0_\la+T^1_\la$ as above, then since 
$T_\la=(-\Delta_g-(\la+i\e(\la))^2)^{-1}-R_\la$, we trivially obtain from \eqref{5.26} and \eqref{5.28} the bound
\begin{equation}\label{5.30}
\|T_\la\|_{L^2(\mathbb{T}^2)\to L^{q}(\mathbb{T}^2)}\lesssim \la^{\sigma(q)-1}\, \bigl(\e(\la)\bigr)^{-1/2}, \,\,\,\text{if}\,\,\, q>6.
\end{equation}

For the operator $T^1_\la$, we claim that if $\e(\la)\ge \la^{-1/3}$ as in \eqref{5.23}, we have 
\begin{equation}\label{5.31}\|T^1_\la\|_{L^{1}(\mathbb{T}^2)\to L^{\infty}(\mathbb{T}^2)}=O(1).
\end{equation}
To see this, we shall split the integral dyadically as before by writing
$$T^{1}_\la = T^{1,0}_\la +\sum_{j=1}^\infty T^{1,j}_\la,
$$
where for $j=1,2,\dots$
\begin{equation}
\nonumber
T^{1,j}_\la = \frac{i}{\la+i\e(\la)}\int_0^\infty 
\beta(2^{-j}t) \,
(1-\eta(t)) \, \eta(t/T) \, e^{i\la t}e^{-\e(\la)t}\, \cos tP \, dt,
\end{equation}
and $T^{1,0}_\la$ is given by an analogous
formula with $\beta(2^{-j}t)$ replaced by
$\beta_0(t)\in C^\infty_0(\Rn)$.

If $j=0$, by using the spectral projection estimates of one of us \cite{sogge88} and the fact that 
\begin{equation}
\nonumber
T^{1,0}_\la(\tau) \lesssim \la^{-1}(1+|\la-\tau|)^{-N}\,\,\,\, \forall \, N \, \,  \text{if }  \, \la\ge1 \, 
\, \text{ and } \, \,  \tau\ge 0,
\end{equation}
 it is not hard to obtain
\begin{equation}\label{5.32}
\|T^{1,0}_\la\|_{L^{1}(\mathbb{T}^2)\to L^{\infty}(\mathbb{T}^2)}=O(1).
\end{equation}
On the other hand, if $j>0$, by using \eqref{4.19} for $n=2$, we have
\begin{equation}\label{5.33}
\|T_\la^{1,j}f\|_{L^{\infty}(\mathbb{T}^2)} \lesssim 
2^{\frac32 j}\la^{-\frac12} \|f\|_{L^{1}(\mathbb{T}^2)}, \, \,
\, j=1,2,\dots .
\end{equation}
Since $T^{1,j}_\la=0$ if $2^j$ is larger than
a fixed constant times $(\e(\la))^{-1}$, after summing over $j$, if $\e(\la)\ge \la^{-1/3}$, we obtain \eqref{5.31}.

As for the local operator $T_\la^0$, by repeating the argument in \eqref{5.12}-\eqref{5.15}, we have the following kernel estimates
\begin{equation}\label{5.34}
|T^0_\la(x,y)|\leq
\begin{cases}
C_0|\log(\la d_g(x,y)/2)|,\hspace{2mm} \text{if} \hspace{2mm} d_g(x,y)\leq \la^{-1}\\
C_0\la^{-1/2}\big(d_g(x,y)\big)^{-1/2},\hspace{2mm}\text{if} \hspace{2mm} \la^{-1}\leq d_g(x,y)\leq 1, \\
\end{cases}
\end{equation}
which is independent of the choice of $\e(\la)$.

To use these bounds write
\begin{align}\label{5.35}
u&= \bigl(-\Delta_g-(\la+i\e(\la))^2\bigr)^{-1} \circ (-\Delta_g-(\la+i\e(\la))^2)u
\\
&=T_\la\bigl(-\Delta_g+V-(\la+i\e(\la))^2\bigr)u \, + \, T_\la\bigl( V_{\le N}\cdot u\bigr) \, + T_\la\bigl(V_{>N}\cdot u\bigr) \notag
\\
&\qquad\qquad+R_\la\bigl(-\Delta_g-(\la +i\e(\la))^2\bigr)u \
\notag
\\
&= I +II +III + IV, \notag
\end{align}
with $V_{\le N}$ and $V_{>N}$ as in \eqref{2.13}.

By \eqref{5.30}
\begin{equation}\label{5.36}
\|I\|_{\infty}\lesssim \bigl(\e(\la)\bigr)^{-1/2} \la^{-1/2} \|(H_V-(\la+i\e(\la))^2)u\|_2,
\end{equation}
and by \eqref{5.29} we similarly obtain
\begin{multline}\label{5.37}
\|IV\|_{\infty}
\lesssim  \bigl(\e(\la)\bigr)^{-1/2} \,  \la^{-1/2} \cdot (\la \,  \e(\la)) \|u\|_2
\\
\lesssim  \bigl(\e(\la)\bigr)^{-1/2} \,  \la^{-1/2} \|(H_V-(\la+i\e(\la))^2)u\|_2,
\end{multline}
using the spectral theorem in the last inequality.

If we use \eqref{5.31}, \eqref{5.34}, and the definition of Kato class, we conclude as before that we can fix
$N$ large enough so that
\begin{equation}\label{5.38}
\|III\|_{\infty}\le \tfrac12 \, \|u\|_{\infty}, \,\,\,\, \text{if}\,\,\, \la\geq \Lambda.
\end{equation}
Also, \eqref{5.30}  and \eqref{2.14} yield for this fixed $N$
\begin{multline}\label{5.39}
\|II\|_{\infty} \le C_N  \bigl(\e(\la)\bigr)^{-1/2} \,  \la^{-1/2} \|u\|_2
\\
\lesssim  \bigl(\e(\la)\bigr)^{-1/2} \,  \la^{-1/2} \|(H_V-(\la+i\e(\la))^2)u\|_2,
\end{multline}
using the spectral theorem and the fact that $\e(\la)\cdot \la\ge1$ if $\la\ge 1$.

Combining \eqref{5.36}, \eqref{5.37}, \eqref{5.38} and \eqref{5.39} yields
\begin{equation}\label{5.40}
\|u\|_{\infty}\lesssim \bigl(\e(\la)\bigr)^{-1/2} \,  \la^{-1/2} \|(H_V-(\la+i\e(\la))^2)u\|_2,\,\,\,\, \text{if}\,\,\, \la\geq \Lambda.
\end{equation}

To obtain the quasimode estimate \eqref{5.24} for $q=\infty$, we need to see that the bounds
in \eqref{5.21} are also valid when $1\le \la<\Lambda$, As before this just follows from the fact that
$$\bigl\|(H_V-\la^2+i\e(\la)\la)^{-1}f\bigr\|_{L^2(\mathbb{T}^2)}
\le C
\bigl\|(H_V-\la^2+i\e(\Lambda)\Lambda)^{-1}f\bigr\|_{L^2(\mathbb{T}^2)}, 
\quad \text{if } \, \, 1\le \la\le \Lambda,
$$
where $C$ is a constant that depend on $\Lambda$.

Now we shall prove quasimode estimates for $q<\infty$. First, if $6<q<\infty$, by using \eqref{5.28}, \eqref{5.29} and \eqref{5.30}, we see that the terms $I$, $II$, and $IV$ are bounded by the right side of \eqref{5.24}. Thus, we only need to focus on the third term $III$. 
Note that by \eqref{5.34}, we have
$$
\sup_y\big(\int_M|T^0_\la(x,y)|^q dx\big)^{1/q}\leq C\la^{-2/q}, \,\,\,\,\text{if}\,\,\,  6\leq q<\infty.
$$
Whence by Minkowski's integral inequality,
\begin{equation}\label{5.41}\|T^0_\la\|_{L^{1}(\mathbb{T}^2)\to L^{q}(\mathbb{T}^2)}\leq C\la^{-2/q}.
\end{equation}

The $T_\la^{1,0}$ operator behaves like the local operator and we can also use the spectral projection estimates in \cite{sogge88} to get
\begin{equation}\label{5.42}\|T^{1,0}_\la\|_{L^{1}(\mathbb{T}^2)\to L^{q}(\mathbb{T}^2)}\leq C\la^{-2/q}.
\end{equation}
To obtain the analog of \eqref{5.42} for the operator $T_\la^{1,j}$, we shall use interpolation between \eqref{5.33} and the following estimates:
\begin{equation}\label{5.43}
\|T_\la^{1,j}f\|_{L^{2}(\mathbb{T}^2)} \lesssim 
2^{\frac j2 }\la^{-\frac12} \|f\|_{L^{1}(\mathbb{T}^2)}, \, \,
\, j=1,2,\dots,
\end{equation}
which follows from applying a dual version of \eqref{4.58} with $\rho=2^{-j}$ as well as the fact that
$$T_\la^{1,j}(\tau) =O\bigl(2^j(1+2^j|\la-\tau|)^{-N}),
\quad
\forall \, N \, \,  \text{if }  \, \la\ge1 \, 
\, \text{ and } \, \,  \tau\ge 0.$$

Since $\frac1q=\frac12\cdot \theta+\frac1\infty\cdot(1-\theta)$, with $\theta=\frac2q$, by interpolation between \eqref{5.33} and \eqref{5.43}, we get
\begin{equation}\label{5.44}\|T^{1,j}_\la\|_{L^{1}(\mathbb{T}^2)\to L^{q}(\mathbb{T}^2)}\leq C\la^{-1/2}2^{j(\frac32-\frac2q)}, \,\,\,\text{if}\,\,\, 2<q<\infty.
\end{equation}
After summing over the
$j\in {\mathbb N}$ with
$2^j\lesssim \e(\la)^{-1}$, we conclude that
\begin{equation}\label{5.44'}
\|T^1_\la\|_{L^{1}(\mathbb{T}^2)\to L^{q}(\mathbb{T}^2)}\leq C \la^{-1/2}\e(\la)^{\frac2q-\frac32}+C\la^{-2/q}.
\end{equation}
Thus, we would have 
\begin{equation}\label{5.45}\|T^1_\la\|_{L^{1}(\mathbb{T}^2)\to L^{q}(\mathbb{T}^2)}\leq C\la^{-2/q},
\end{equation}
if we knew $\e(\la)\geq \la^{-\frac{q-4}{3q-4}}$. However, since $\frac{q-4}{3q-4}>\frac{q-6}{3q-10}$, this yields \eqref{5.45} for all $\e(\la)$ satisfying \eqref{5.23}, if $6<q<\infty$.

If we combine \eqref{5.41} and \eqref{5.45}, by H\"older's inequality,
$$
\|T_\la(V_{>N}u)\|_q\leq C\la^{-2/q}\|V_{>N}u\|_1\leq C\la^{-2/q}\|V\|_1\|u\|_\infty
$$
Since we have just proved that
\begin{equation}
\|u\|_{\infty}\lesssim \bigl(\e(\la)\bigr)^{-1/2} \,  \la^{-1/2} \|(H_V-(\la+i\e(\la))^2)u\|_2,
\nonumber
\end{equation}
we conclude that the term $III$ is dominated by the right side of \eqref{5.2}, this finishes the proof of \eqref{5.24}.

To conclude the section we shall prove \eqref{5.25}, by using \eqref{5.27}, \eqref{4.4}, and repeating the arguments above, we can easily see that the terms $I$, $II$, and $IV$ are bounded by the right side of \eqref{5.25}. For the third term $III$, if we combine \eqref{5.41} and \eqref{5.44'}, and use \eqref{5.24} for $q=\infty$ as above, we have
\begin{equation}\label{5.46}
\begin{aligned}
\|T_\la(V_{>N}u)\|_6&\leq C(\la^{-1/3}+\la^{-1/2}\e(\la)^{\frac13-\frac32})\|V_{>N}u\|_1\\
&\leq C(\la^{-1/3}+\la^{-1/2}\e(\la)^{-\frac76})\|V\|_1\|u\|_\infty \\
&\leq \bigl(\bigl(\e(\la)\bigr)^{-1/2} \,  \la^{-5/6}+\bigl(\e(\la)\bigr)^{-5/3} \,  \la^{-1} \bigr)\|(H_V-(\la+i\e(\la))^2)u\|_2,
\end{aligned}
\end{equation}
which is bounded by the right side of \eqref{5.25} if $\e(\la)\geq\la^{-1/5}$. Thus, the proof of \eqref{5.25} is complete.

\newsection{ Appendix: Self-adjointness and limited Sobolev estimates}\label{appendix}

As we stated before, for brevity, $dx$  denotes the Riemannian volume element for $(M,g)$.

\begin{proposition}\label{self-adjoint}
For $n\geq3$, if $V\in L^{n/2}(M)$ the quadratic form,
\begin{equation}\label{s0}
q_V(u,v)=-\int_M Vu \, \overline{v}\, dx +\int -\Delta_g u \, \overline{v}\, dx, \quad
u,v \in \Dom(\sqrt{-\Delta_g+1}),
\end{equation}
is bounded from below and defines a unique semi-bounded self-adjoint operator $H_V$ on $L^2$.  Moreover, $C^\infty(M)$ constitutes a form core\footnote{Recall that a \emph{form core} for $q_V$ is a subspace $S$ which approximates elements $u$ in the domain of the form in that there exists a sequence $u_m \in S$ satisfying $\lim_m \|u-u_m\|^2 + q_V(u-u_m,u-u_m) = 0$.} for $q_V$.
\end{proposition}

\begin{proof}
Since $(-\Delta_g+1)^{1/2}$ is self-adjoint, by perturbation theory (specifically the KLMN Theorem
(see \cite[Theorem X.17]{reed1975ii}) it suffices to prove that for any $0<\e<1$ there is a constant
$C_\e<\infty$ so that
\begin{equation}\label{s1}
\int |V| \, |u|^2 \, dx \le \e \, \bigl\|(-\Delta_g+1)^{1/2}u\bigr\|_2^2 +C_\e\|u\|_2^2, \quad
u\in \Dom(\sqrt{H_0}),
\end{equation}
where $H_0=-\Delta_g+1$.

To prove this, for each small $\delta>0$ choose a maximal $\delta$-separated collection of points
$x_j\in M$, $j=1,\dots,N_\delta$, $N_\delta\approx \delta^{-n}$.  Thus, $M=\cup B_j$ if $B_j$ is the $\delta$-ball about
$x_j$, and if $B_j^*$ is the $2\delta$-ball with the same center,
\begin{equation}\label{s2}
\sum_{j=1}^{N_\delta} {\bf 1}_{B^*_j}(x)\le C_M,
\end{equation}
where $C_M$ is independent of $\delta \ll 1$ if ${\bf 1}_{B_j^*}$ denotes the indicator function of $B^*_j$.  Since
$V\in L^{n/2}(M)$, for any fixed $\e$, we can choose $\delta>0$ small enough so that
\begin{equation}\label{s3}
C_M\Bigl(\, C_0\sup_{x\in M}\|V\|_{L^{n/2}(B(x,2\delta))}\, \Bigr) <\e,
\end{equation}
where $C_0$ is the constant in  \eqref{s4} below.

Now for each $B_j$, define a smooth bump function $\phi_j$ with $\phi_j\equiv 1$ on $B_j$, and $\phi_j\equiv 0$ outside on $B^*_j$.
Since $M=\cup B_j$, we have
\begin{equation}\label{s4}
\begin{aligned}
\int |V| \, |u|^2 \, dx &\le \sum_j \int |V| \, |\phi_ju|^2 \, dx \\
&\le (\sup_{x\in M}\|V\|_{L^{n/2}(B(x,2\delta))})\sum_j \|\phi_ju\|^2_{\frac{2n}{n-2}} \\
&\le  C_0(\sup_{x\in M}\|V\|_{L^{n/2}(B(x,2\delta))})\sum_j \|\nabla (\phi_ju)\|^2_{2} \\
&\le C_0(\sup_{x\in M}\|V\|_{L^{n/2}(B(x,2\delta))})\sum_j \Big( \|\nabla (u)\|^2_{L^2(B_j^*)}+ \|(\nabla\phi_j)u)\|^2_{2}\Big)\\
&\le \e \, \bigl\|(-\Delta_g+1)^{1/2}u\bigr\|_2^2 +C_\e\|u\|_2^2, \quad 
u\in \Dom(\sqrt{H_0}),
\end{aligned}
\end{equation}
where $H_0=-\Delta_g+1$.   Here we have used Sobolev estimates as well as \eqref{s3}.
\end{proof}

If $u\in \Dom(\sqrt{-\Delta_g+1})$ then $-\Delta_gu$ and $Vu$ are both distributions.  If $H_V$ is the self-adjoint
operator given by the Proposition, then $\Dom(H_V)$ is all such $u$ for which $-\Delta_gu+Vu\in L^2$. 

If we take $\e=1/2$ in \eqref{s1}
we indeed get for large enough $N$
\begin{multline}\label{s5}
\|\sqrt{-\Delta_g+1}u\|_2^2 =\int (-\Delta_g+1)u \, \overline{u} \, dy
\le 2\int(-\Delta_g+V+N)u\, \overline{u}\, dy
\\
=2\bigl\|\sqrt{H_V+N}u\bigr\|_2^2, \quad \text{if } H_V=-\Delta_g+V.
\end{multline}
Thus, $(-\Delta_g+1)^{1/2}(H_V+N)^{-1/2}$ and $(H_V+N)^{-1/2}(-\Delta_g+1)^{1/2}$ are bounded on $L^2$.
Since $(-\Delta_g+1)^{-1/2}$ is a compact operator on $L^2$, so must be $(H_V+N)^{-1/2}$.
From this we conclude that the self-adjoint operator $H_V$ has {\em discrete spectrum}. 

A combination of Sobolev estimtes for the unperturbed operator and \eqref{s5} also gives us 
\begin{equation}\label{s51}
\|u\|_{\frac{2n}{n-2}}\leq C\bigl\|\sqrt{H_V+N}u\bigr\|_2,\,\,\,\,\text{if} \,\,\, u\in \text{Dom}(H_V).
\end{equation}


Note that in the above inequality \eqref{s4} and thus \eqref{s5}, we need the condition that $n\geq 3$, because we do not have 
a suitable Sobolev 
inequality at $\frac{2n}{n -2}$ when $n=2$. Additionally, if $n\geq 5$,  by an analogous argument as in \eqref{s4}, we have for any $0<\e<1$ there is a constant
$C_\e<\infty$ so that
\begin{equation}\label{s6}
\int |Vu|^2 \, dx \le \e \, \bigl\|(-\Delta_g+1)u\bigr\|_2^2 +C_\e\|u\|_2^2, \quad
u\in \Dom(H_0),
\end{equation}
where $H_0=-\Delta_g+1$.

Inequality \eqref{s6} also appears in \cite[Theorem X.21]{reed1975ii} under a weaker assumption on $V$.  The reason it does not hold when $n=3, 4$ is that
we do not have an appropriate Sobolev 
inequality at $\frac{2n}{n -4}$ when $n=3, 4$. As a consequence of \eqref{s6},  we have for large enough N
\begin{equation}\label{s7}
C_1\|(-\Delta_g+1)u\|_2 \le \bigl\|(H_V+N)u\bigr\|_2\le C_2\|(-\Delta_g+1)u\|_2 , \quad \text{if } H_V=-\Delta_g+V.
\end{equation}

After replacing $V$ by $V+N$ to simplify the notation, we may assume, as we have throughout starting with \eqref{1.3}, that
\eqref{s4} holds with $N=0$.   This just shifts the spectrum and does not change the eigenfunctions.  In this case the spectrum of $H_V$ is positive and its eigenfunctions therefore
are distributional solutions of
$$H_V e_\la =\la^2 e_\la, \quad \text{some }\, \la>0,$$
which means here that $\la$ is the eigenvalue of the ``first order'' operator $\sqrt{H_V}$, i.e.,
\begin{equation}\label{ef}
P_V e_\la =\la e_\la, \quad \text{if } \, \, P_V=\sqrt{H_V}.
\end{equation}


%
%

\bibliography{refs9}
\bibliographystyle{abbrv}
\end{document}